\documentclass[letterpaper, journal, twoside]{IEEEtran}

 \IEEEpubidadjcol 
\IEEEoverridecommandlockouts 

\usepackage{graphics} 
\usepackage{epsfig} 
\usepackage{amsmath} 
\usepackage{amssymb} 

\usepackage{algpseudocode}
\usepackage{algorithm}
\usepackage{epstopdf}

\usepackage{verbatim}

\usepackage{amsthm}

\usepackage{color}
\usepackage{cancel}

\usepackage{enumitem}

\usepackage{tikz} 
\usepgflibrary{arrows}

\usepackage{hyperref}

\usepackage[normalem]{ulem}
\usepackage{enumitem}

 \def\change{}

\newtheorem{assumption}{\bf Assumption}

\newtheorem{theorem}{\bf Theorem}
\newtheorem{proposition}{\bf Proposition}
\newtheorem{lemma}{\bf Lemma}

\newtheorem{remark}{\bf Remark}

 \usepackage{cite}

\title{Stability and performance analysis of NMPC: Detectable stage costs and general terminal costs}
\author{Johannes K\"ohler$^1$, Melanie N. Zeilinger$^1$, Lars Gr\"une$^2$
\thanks{$^1$Institute for Dynamic Systems and Control, ETH Zürich, Zürich CH-8092, Switzerland.}
\thanks{$^2$Mathematical Institute, University of Bayreuth, 95440 Bayreuth, Germany.}
\thanks{Lars Grüne was supported by DFG Research Grants GR 1569/13-2 and GR 1569/25-1.}
}
\begin{document}
\IEEEoverridecommandlockouts
\IEEEpubid{\begin{minipage}{\textwidth}\ \\[20pt] \\ \\
         \copyright 2022 IEEE.  Personal use of this material is  permitted.  Permission from IEEE must be obtained for all other uses, in  any current or future media, including reprinting/republishing this material for advertising or promotional purposes, creating new  collective works, for resale or redistribution to servers or lists, or  reuse of any copyrighted component of this work in other works.
     \end{minipage}}
\maketitle
\begin{abstract} 
We provide a stability and performance analysis for nonlinear model predictive control (NMPC) schemes \change{subject to input constraints}.  
Given an exponential stabilizability and detectability condition w.r.t. the employed state cost, we provide a sufficiently long prediction horizon to ensure asymptotic stability and a desired performance bound w.r.t. the infinite-horizon optimal controller. 
Compared to existing results, the provided analysis is applicable to positive semi-definite (detectable) cost functions, provides \textit{tight} bounds using a linear programming analysis, and allows for a seamless integration of general positive-definite terminal cost functions in the analysis. 
The practical applicability of the derived theoretical results are demonstrated \change{with numerical examples}. 
\end{abstract}
\section{Introduction}
\subsubsection*{Motivation} 
Nonlinear model predictive control (NMPC), also called receding horizon control, is an optimization-based control technique that generates a feedback law by repeatedly solving finite-horizon open-loop optimal control problems~\cite{grune2017nonlinear,rawlings2017model}. 
The popularity of this control method stems largely from the applicability to general nonlinear systems and the possibility to explicitly handle constraints. 

Standard design methods to ensure closed-loop stability and performance of NMPC schemes include the offline design of suitable terminal ingredients and thus a local control Lyapunov function (CLF), compare~\cite{rawlings2017model,chen1998quasi,mayne2000constrained,limon2006stability}. However, the inclusion of terminal ingredients restricts the feasible set and thus decreases the region of attraction. 
Furthermore, the design of a local CLF can often be non-trivial, e.g., if the linearization is not stabilizable~\cite{worthmann2015model}, online changing setpoints are considered~\cite{limon2018nonlinear}, or for large-scale distributed systems~\cite{giselsson2013feasibility,conte2016distributed}. 
Thus, the study of NMPC schemes without such terminal ingredients is of high practical relevance, compare~\cite{mayne2013apologia} and \cite[Sec.~7.4]{grune2017nonlinear} for general discussions regarding the drawbacks and advantages of terminal ingredients. 

For many control problems, the control goal is naturally specified in terms of output (and input) setpoints~\cite{limon2018nonlinear}, references/paths~\cite{faulwasser2012optimization}, or sets/zones~\cite{ferramosca2010mpc}, which are naturally addressed using positive semi-definite cost functions. Furthermore, we are currently seeing a renaissance of data-driven identification methods used in NMPC, e.g., using linear-parameter varying models~\cite{abbas2018improved} or non-parametric models~\cite{bongard2021robust,manzano2020robust}. 
Such systems are typically specified in terms of an input-output behaviour and correspondingly regulation is most naturally formulated using a positive semi-definite input-output cost. 
%
Hence, in this paper we are interested in deriving stability and performance results for NMPC, without requiring a local CLF as a terminal cost and while also allowing for a positive semi-definite stage cost.

\subsubsection*{Related work}
Stability results for NMPC without a local CLF can be found in~\cite{alamir1995stability,jadbabaie2005stability}, however, the resulting conditions on the prediction horizon are typically difficult to verify a-priori. 
Based on exponential \textit{cost controllability}, verifiable bounds on the prediction horizon ensuring stability have been obtained in~\cite{tuna2006shorter,grune2008infinite}, including a corresponding performance bound w.r.t the infinite horizon optimal cost. 
These bounds are further improved in~\cite{grune2009analysis,grune2010analysis,Wort11}, resulting in tight estimates based on a linear programming (LP) analysis, 
compare also the continuous-time results in~\cite{reble2012unconstrained}.
Furthermore, bounds utilizing some additional terminal weight/cost can be found in~\cite{tuna2006shorter,grune2010analysis,reble2012unifying,reble2012improved,magni2001stabilizing,kohler2021stability}. 
However, the mentioned stability results are in general not applicable to control problems involving positive semi-definite cost functions. 

Stability results for NMPC with positive semi-definite stage costs have been derived in~\cite{grimm2005model} by introducing a notion of \textit{cost-detectability}. 
The resulting bounds have been improved in~\cite[App. A]{Koehler2020Regulation} by using a stronger \textit{cost-observability} condition. 
However, the resulting conditions remain conservative, i.e., the resulting guarantees are only valid for unnecessarily large prediction horizons. 

In summary, despite its relevance for practical applications, the stability and performance analysis of NMPC schemes with positive semi-definite cost is not yet fully explored. 
Furthermore, the impact of general terminal costs on stability and performance in NMPC schemes has received little attention. 

We note that the stability and performance guarantees in the related works~\cite{tuna2006shorter,grune2008infinite,grune2009analysis,grune2010analysis,Wort11,reble2012unconstrained,reble2012unifying,grimm2005model,reble2012improved} are \textit{global} and the corresponding assumptions exclude state constraints or most unstable systems. 
Extensions to characterize a guaranteed region of attraction under less restrictive \textit{local} stabilizability conditions can be found in~\cite{limon2006stability,boccia2014stability,darup2015missing,kohler2021stability,Koehler2020Regulation,kohler2021dynamic} (cf. Remark~\ref{rk:constraints} for a detailed discussion), which is, however, not the focus of the paper.

\subsubsection*{Contribution} 
In this paper, we provide a stability and performance analysis of NMPC schemes with positive semi-definite cost functions and input constraints.  
To account for positive semi-definite stage costs, we consider the \textit{cost-detectability} and \textit{cost-controllability} conditions from~\cite{grimm2005model}, which generalize the \textit{cost-controllability} condition used in~\cite{grune2017nonlinear,worthmann2015model,tuna2006shorter,grune2008infinite,grune2009analysis,grune2010analysis,Wort11} for positive definite cost functions. 
Based on this characterization, we provide (tight) bounds on the closed-loop performance utilizing an LP analysis. 
These results provide quantitative performance bounds depending on the prediction horizon $N$, an exponential detectability constant, and a controllability constant, which also yields simple design guidelines. 

As an additional contribution, we extend this analysis to general terminal costs, which are characterized as an approximate CLF. 
The corresponding theoretical analysis reveals how a terminal cost can improve the stability guarantees or also adversely affect the closed-loop performance. 
In addition, we study simple terminal costs in terms of a scaled state measure (cf.~\cite{grune2010analysis}) or a finite-tail costs (cf.~\cite{magni2001stabilizing,kohler2021stability}), resulting in simple design guidelines. 
We recover the LP results in~\cite{grune2009analysis,grune2010analysis} as a special case when considering positive definite cost functions and a simple scaled terminal weighting. 

We also provide numerical examples, demonstrating the applicability of the derived theory to efficiently tune/design a stabilizing MPC.

\subsubsection*{Outline} 
Section~\ref{sec:problem} presents the setup and the considered NMPC formulation. 
Section~\ref{sec:theory} contains the theoretical analysis. 
Section~\ref{sec:terminal} extends the results to general terminal costs.
Sections~\ref{sec:design}--\ref{sec:term_design} discuss the applicability of the derived theory. 
Section~\ref{sec:num} demonstrates the theoretical results with numerical examples.
Section~\ref{sec:sum} concludes the paper.

\subsubsection*{Notation}
The quadratic norm with respect to a positive definite matrix $Q=Q^\top$ is denoted by $\|x\|_Q^2:=x^\top Q x$. 
The non-negative real numbers are denoted by $\mathbb{R}_{\geq 0}=\{r\in\mathbb{R}|~r\geq 0\}$. 
We denote the set of integers by $\mathbb{I}$ and the set of integers in an interval $[a,b]$ by $\mathbb{I}_{[a,b]}:=\{k\in\mathbb{I}|~a\leq k\leq b\}$. 
By $\mathcal{K}_\infty$ we denote the class of functions $\alpha:\mathbb{R}_{\geq 0}\rightarrow\mathbb{R}_{\geq 0}$, which are continuous, strictly increasing, unbounded, and satisfy $\alpha(0)=0$. 
The point-to-set-distance of a vector $x\in\mathbb{R}^n$ to a set $\mathcal{A}\subset\mathbb{R}^n$ is defined as $\|x\|_{\mathcal{A}}:=\inf_{s\in\mathcal{A}}\|x-s\|$. 
The empty product is defined as $1$, i.e., $\prod_{j=k_1}^{k_2}c_j=1$, for any $k_2<k_1$.

\section{Problem formulation}
\label{sec:problem}
We consider a nonlinear discrete-time system
\begin{align}
\label{eq:sys}
x(k+1)=f(x(k),u(k)),\quad x(0)=x_0,
\end{align}
with state $x(k)\in{X}=\mathbb{R}^n$, control input $u(k)\in\mathbb{U}\subseteq\mathbb{R}^m$, initial condition $x_0\in{X}$, and time step $k\in\mathbb{I}_{\geq 0}$. 
For a given initial state $x\in X$ and input sequence $u(\cdot)\in\mathbb{U}^N$, we denote the solution to~\eqref{eq:sys} after $k$ steps by $x_u(k,x)\in X$, $k\in\mathbb{I}_{[0,N]}$ with $x_u(0,x)=x$. 
We only consider input constraints $\mathbb{U}$, compare Remark~\ref{rk:constraints} for the consideration of state constraints. 

We consider a non-negative stage cost $\ell:X\times\mathbb{U}\rightarrow\mathbb{R}_{\geq 0}$ that should be minimized. 
The ideal performance is achieved by minimizing the infinite-horizon cost $\mathcal{J}_\infty(x_0,u):=\sum_{k=0}^\infty \ell(x_u(k,x),u(k))$ with the corresponding 
 optimal cost $V_{\infty}(x_0):=\inf_{u\in \mathbb{U}^\infty }\mathcal{J}_\infty(x_0,u)$. 
However, the corresponding infinite-horizon optimal solution is typically intractable to compute and hence we study MPC as a closed-loop approximator of the infinite horizon cost. 
In particular, we consider the finite-horizon cost $\mathcal{J}_N(x,u)=\sum_{k=0}^{N-1}\ell(x_u(k,x),u(k))$ with some prediction horizon $N\in\mathbb{I}_{\geq 1}$. 
The value function is given by $V_N(x)=\inf_{u\in\mathbb{U}^N}\mathcal{J}_N(x,u)$ and we denote a corresponding minimizer\footnote{%
We assume that a minimizer exists, which holds for $\mathcal{J}_N$ continuous and $\mathbb{U}$ compact, compare \cite[Prop.~2.4]{rawlings2017model}.} by $u_{N,x}^*$. 
The corresponding control law is given given by $u=\mu(x):=u_{N,x}^*(0)$, i.e., at each time step the first element of the optimal open-loop input sequence $u_{N,x}^*$ is applied to the system. 
The resulting closed-loop state and input is denoted by $x_\mu(k),u_\mu(k)$, $k\in\mathbb{I}_{\geq 0}$ and the infinite-horizon closed-loop performance is $\mathcal{J}_\infty^\mu(x_0):=\sum_{k=0}^\infty \ell(x_\mu(k),u_\mu(k))$. 

The main goal of this paper is to analyze the closed-loop properties of this MPC in terms of stability and performance relative to the infinite horizon optimal performance $V_{\infty}(x_0)$. 
\section{Performance and stability analysis with positive semi-definite stage cost} 
 \label{sec:theory}
In the following, we provide a closed-loop performance and stability analysis. 
In Section~\ref{sec:theory_1}, we provide a cost controllability and detectability condition and recap the stability and performance results from~\cite{grimm2005model,Koehler2020Regulation}. 
Then, in Section~\ref{sec:theory_2} we provide improved performance and stability conditions using an LP analysis. 
In Section~\ref{sec:theory_3}, we provide more intuitive bounds by characterizing the analytical solution to the LP. 
\subsection{Cost controllability and detectability}
 \label{sec:theory_1}
Denote $\ell_{\min}(x):=\inf_{u\in\mathbb{U}}\ell(x,u)$. 
We are particularly interested in the case that $\ell$ is positive semi-definite w.r.t. some set $\mathcal{A}\subseteq X$ that should be stabilized, i.e., $\ell_{\min}(x)\ge 0$ $\forall x \in X$ and $\ell_{\min}(x)=0$ $\forall x\in\mathcal{A}$. 
We follow the analysis in~\cite{grimm2005model} and introduce a \textit{state measure} $\sigma:X\rightarrow\mathbb{R}_{\geq 0}$, which is assumed to be positive definite w.r.t. a set $\mathcal{A}$. 
\begin{assumption}
\label{ass:set} (State measure)
There exist functions $\alpha_1,\alpha_2\in\mathcal{K}_\infty$, such that for all $x\in X$, we have
\begin{align}
\label{eq:set}
&\alpha_1(\|x\|_{\mathcal{A}}) \leq \sigma(x)\leq \alpha_2(\|x\|_{\mathcal{A}}).
\end{align}
\end{assumption}
Note that in case of positive definite costs functions, $\ell_{\min}$ is a natural state measure $\sigma$, compare~\cite[Ass.~3.2]{grune2012nmpc}. 

The following assumptions correspond to the detectability and stabilizability conditions in~\cite[SA~3--4]{grimm2005model}, albeit with simpler linear bounds (cf. \cite[Ass.~3--4]{Koehler2020Regulation}). 
\begin{assumption}
\label{ass:stab} (Exponential cost controllability)
There exist constants $\bar{\gamma}\geq 0$, $\gamma_k\in[0,\bar{\gamma}]$, $k\in\mathbb{I}_{\geq 1}$, such that 
\begin{align}
\label{eq:stab}
V_k(x)\leq \gamma_k\sigma(x),\quad \forall x\in X,\quad k\in\mathbb{I}_{\geq 1}.
\end{align}
\end{assumption}
\begin{assumption}
\label{ass:detect} (Cost detectability)
There exist a function $W:X\rightarrow\mathbb{R}_{\geq 0}$ and constants $\epsilon_{\mathrm{o}}>0$, $\underline{\gamma}_{\mathrm{o}},\overline{\gamma}_{\mathrm{o}}\geq 0$, such that for all $(x,u)\in X\times\mathbb{U}$, the following inequalities hold:
\begin{subequations}
\label{eq:detect}
\begin{align}
\label{eq:detect_1}
\underline{\gamma}_{\mathrm{o}}\sigma(x)\leq W(x)\leq&\overline{\gamma}_{\mathrm{o}}\sigma(x),\\
\label{eq:detect_2}
W(f(x,u))-W(x)\leq& -\epsilon_{\mathrm{o}}\sigma(x)+\ell(x,u).
\end{align}
\end{subequations}
\end{assumption} 
Assumptions~\ref{ass:set}-\ref{ass:stab} ensures that the infinite-horizon optimal cost admits an upper bound $V_\infty(x)\leq \bar{\gamma}\alpha_2(\|x\|_{\mathcal{A}})$.

Assumption~\ref{ass:detect} uses dissipation inequalities~\eqref{eq:detect} with the storage function $W(x)$ to characterize the connection between the stage cost $\ell$ and the state measure $\sigma$. 
By using a telescopic sum for condition~\eqref{eq:detect_2}, we see that a bounded closed-loop cost $\mathcal{J}_\infty^\mu$ implies that $\sigma$ converges to $0$ and thus by~\eqref{eq:set}, the system converges to the set $\mathcal{A}$, i.e., we can achieve the optimal performance only if we drive the system to the set $\mathcal{A}$. 
Similar dissipation characterizations are used in economic MPC with indefinite cost functions $\ell$ to study optimal system operation at some point/set $\mathcal{A}$, compare~\cite{hoger2019relation}.

The relation of Assumptions~\ref{ass:set}--\ref{ass:detect} to stabilizability and detectability properties of the system in case of input-output stage costs $\ell$ will be clarified in Section~\ref{sec:design}.
\begin{remark}
\label{rk:constraints}
\change{(State constraints, unstable systems, and region of attraction)
The global stabilizability condition in Assumption~\ref{ass:stab} tends to be too restrictive for hard state constraints or unstable systems. 
In related works (cf. \cite{boccia2014stability,darup2015missing,kohler2021stability,Koehler2020Regulation,kohler2021dynamic}), this issue is resolved by relaxing Inequality~\eqref{eq:stab} to only hold in some local neighbourhood, e.g., all $x\in X$ where $\sigma(x)\leq c$, $c>0$. 
Then, the resulting closed-loop guarantees are valid on a region of attraction, which is characterized as a a sublevel sets of the Lyapunov function and explicitly depends on the prediction horizon $N$. 
The main technical difficulty stems from establishing which part of the predicted sequence are guaranteed to lie in this local neighbourhood, where Inequality~\eqref{eq:stab} can be invoked.  
This can be done by applying simple bounds~\cite{boccia2014stability,darup2015missing} or using tight contradiction arguments~\cite{limon2006stability,kohler2021stability,Koehler2020Regulation}, which can also be merged with the LP analysis from~\cite{grune2009analysis,grune2010analysis}, compare ~\cite[Sec.~4.1]{kohler2021dynamic}. 
Extending these arguments to the more general setting in the present paper is an open issue.}  
\end{remark}
The following theorem summarizes the stability and performance results in~\cite[Thm.~1]{grimm2005model} and \cite[Thm.~1]{Koehler2020Regulation}. 
\begin{theorem}
\label{thm:grimm}
Let Assumptions~\ref{ass:set}--\ref{ass:detect} hold. 
Then, for any $x\in X$, the function $Y_N:=V_N+W$ satisfies
\begin{subequations}
\label{eq:Lyap}
\begin{align}
\label{eq:Lyap_1}
\epsilon_{\mathrm{o}}\sigma(x)\leq Y_N(x)\leq &(\gamma_N+\gamma_{\mathrm{o}})\sigma(x),\\
\label{eq:Lyap_2}
Y_N(f(x,\mu(x)))-Y_N(x)\leq& -\epsilon_{\mathrm{o}}\alpha_N \sigma(x),
\end{align}
\end{subequations}
with $\alpha_N:=1-\frac{\gamma_N(\gamma_N+\gamma_{\mathrm{o}})}{\epsilon_o^2(N-1)}$. 
If additionally, $N>\underline{N}:=1+\frac{(\overline{\gamma}+\gamma_{\mathrm{o}})\overline{\gamma}}{\epsilon_{\mathrm{o}}^2}$, then $\alpha_N\in(0,1]$. 
Furthermore, for any $x_0\in X$ the set $\mathcal{A}$ is asymptotically stable for the corresponding closed loop $x_\mu(\cdot)$ and the following performance bound holds
\begin{align}
\label{eq:performance}
\alpha_N(\mathcal{J}^\mu_\infty(x_0)+W(x_0))\leq Y_\infty(x_0):=V_\infty(x_0)+W(x_0).
\end{align}
\end{theorem}
\change{The proof is detailed in Appendix~\ref{sec:app_2}, based on Appendix~\ref{app:grimm_terminal}.}  
Theorem~\ref{thm:grimm} ensures asymptotic stability, given a lower bound $\underline{N}$ on the prediction horizon. 
Furthermore, condition~\eqref{eq:performance} provides a suboptimality estimate $\alpha_N\in[0,1)$ w.r.t. the optimal infinite-horizon performance that ensures that we recover infinite-horizon optimal performance as $N\rightarrow\infty$, i.e., $\lim_{N\rightarrow\infty}\mathcal{J}^\mu_\infty(x_0)=V_\infty(x_0)$. 
Hence, Theorem~\ref{thm:grimm} already provides a qualitative stability and performance analysis for MPC schemes with positive semi-definite cost. 

\subsection{Linear programming analysis}
 \label{sec:theory_2}
In the following, we improve the quantitative results in Theorem~\ref{thm:grimm} by providing less conservative conditions on the prediction horizon $N$ and the suboptimality index $\alpha_N$. For positive definite costs, i.e., costs with $\ell_{\min}(x)>0$ $\forall x\not\in\mathcal{A}$, such improvements were derived in \cite{tuna2006shorter,grune2008infinite,grune2009analysis,grune2010analysis} and the comparison in \cite{grune2012nmpc} shows that the achievable improvements are significant. Among these references, \cite{grune2009analysis,grune2010analysis} are based on a linear programming approach and provide the tightest estimates. The following theorem extends this approach to the more general setting considered in this paper. 
\begin{theorem}
\label{thm:main}
Let Assumptions~\ref{ass:stab}--\ref{ass:detect} hold. 
Then, for any $x,x_0\in X$, Inequalities~\eqref{eq:Lyap} and \eqref{eq:performance} hold with $\alpha_N$ according  to\footnote{%
Note that Equation~\eqref{eq:LP_1} uniquely defines $\alpha_N$ based on the minimum of the LP~\eqref{eq:LP} since $\epsilon_{\mathrm{o}}>0$.}
the following LP:
\begin{subequations}
\label{eq:LP}
\begin{align}
\label{eq:LP_1}
&\epsilon_{\mathrm{o}}(\alpha_N-1)\nonumber\\
=&\min_{\tilde{\ell},\tilde{W},\tilde{\sigma},\tilde{V}}\sum_{k=1}^{N-1}\tilde{\ell}_k-\tilde{V}\\
\label{eq:LP_2}
\mathrm{s.t. ~}&\tilde{\sigma}_0=1,\\
\label{eq:LP_3}
&\tilde{\ell}_k\geq 0,~k\in\mathbb{I}_{[0,N-1]},~\tilde{\sigma}_k\geq 0,~\tilde{W}_k\geq 0,~k\in\mathbb{I}_{[0,N]},\\
\label{eq:LP_4}
&\underline{\gamma}_{\mathrm{o}}\tilde{\sigma}_k\leq \tilde{W}_k\leq \overline{\gamma}_{\mathrm{o}}\tilde{\sigma}_k,~k\in\mathbb{I}_{[0,N]},\\
\label{eq:LP_5}
&\tilde{W}_{k+1}-\tilde{W}_k\leq -\epsilon_{\mathrm{o}}\tilde{\sigma}_k+\tilde{\ell}_k,~k\in\mathbb{I}_{[0,N-1]},\\
\label{eq:LP_6}
&\sum_{j=k}^{N-1}\tilde{\ell}_j\leq\gamma_{N-k}\tilde{\sigma}_k,~k\in\mathbb{I}_{[0,N-1]},\\
\label{eq:LP_7}
&\tilde{V}\leq \sum_{j=1}^{k-1}\tilde{\ell}_j+\gamma_{N-k+1}\tilde{\sigma}_k,~k\in\mathbb{I}_{[1,N]}.
\end{align}
\end{subequations}
\end{theorem}
\change{The proof is detailed in Appendix~\ref{sec:app_2}, based on Appendix~\ref{app:terminal_LP}.}   
The LP~\eqref{eq:LP} explicitly computes a worst-case feasible sequence for the decrease of the value function $V_N$, i.e., 
a sequence with the smallest decrease in the value function $V_N$ given that Assumptions~\ref{ass:stab}--\ref{ass:detect} hold and the sequence is open-loop optimal. 
Hence, better estimates can only be derived if some additional assumptions are considered or if stability/performance is not ensured based on a one-step decrease of the value function $V_N$. 
Correspondingly, the derived bound improves the results in Theorem~\ref{thm:grimm} by providing less restrictive conditions in terms of the prediction horizon $N$ to guarantee asymptotic stability or a desired performance bound $\alpha_N$ in~\eqref{eq:performance}.

\subsection{Analytical solution}
 \label{sec:theory_3} 
In the following, we derive an analytical solution to the LP~\eqref{eq:LP} in order to allow for an easier interpretation of the results. 
\change{To be more specific, we compute a lower $\alpha_N\geq \hat{\alpha}_N$ with an analytical formula based on a simpler LP and then provide sufficient conditions to ensure $\alpha_N=\hat{\alpha}_N$.} 
First, we consider the important special case $\sigma=W$, i.e., the detectability characterization in Assumption~\ref{ass:detect} contains an \textit{exponential} decay rate $\eta:=1-\epsilon_{\mathrm{o}}\in(0,1)$, assuming $\epsilon_{\mathrm{o}}\in(0,1)$. 
\begin{theorem}
\label{thm:analytic_sigma_W}
Suppose $\underline{\gamma}_{\mathrm{o}}=\overline{\gamma}_{\mathrm{o}}=1$, and $\epsilon_{\mathrm{o}}\in(0,1)$. 
Then, the solution $\alpha_N$ to the LP~\eqref{eq:LP} satisfies
\begin{align}
\label{eq:hat_alpha_explicit}
&\epsilon_{\mathrm{o}}(1-\alpha_N)\leq \epsilon_{\mathrm{o}}(1-\hat{\alpha}_N)\\
&:=\dfrac{\gamma_1(\gamma_N+\eta)\prod_{j=0}^{N-2}(\eta+\gamma_{N-j})}{\prod_{j=0}^{N-1}(1+\gamma_{N-j})-\gamma_1\prod_{j=0}^{N-2}(\eta+\gamma_{N-j})} 
\nonumber
\end{align}
and $\alpha_N>0$ holds if
 \begin{align}
\label{eq:hat_alpha_explicit_N}
N>\underline{N}:=1+\dfrac{\log(\overline{\gamma})-\log(\epsilon_{\mathrm{o}})}{\log(1+\overline{\gamma})-\log(\overline{\gamma}+\eta)}.
\end{align}
If additionally $\gamma_k=\overline{\gamma}$, $\forall k\in\mathbb{I}_{\geq 1}$, then $\alpha_N=\hat{\alpha}_N$.
\end{theorem}
\change{The proof is detailed in Appendix~\ref{sec:app_2}, based on Appendix~\ref{app:analytic_terminal_sigma_W}.} 
This result provides an easy to evaluate analytical bound for the LP analysis in Theorem~\ref{thm:main} in case $\sigma=W$. 
In particular, Equation~\eqref{eq:hat_alpha_explicit} provides a direct formula how the cost-controllability constants $\gamma_k$, the cost-detectability constant $\eta\in(0,1)$, and the prediction horizon $N$ yield a performance bound $\alpha_N$. 

Next, we consider the special case $\sigma=\ell_{\min}$ (positive definite stage cost). 
In this case, Assumption~\ref{ass:stab} reduces to the exponential cost controllability condition used in~\cite{grune2017nonlinear,worthmann2015model,tuna2006shorter,grune2008infinite,grune2009analysis,grune2010analysis}, while Assumption~\ref{ass:detect} holds trivially with $W=0$, $\epsilon_{\mathrm{o}}=1$.
\begin{theorem}
\cite{grune2009analysis,grune2010analysis}
\label{thm:analytic_grune}
Suppose $\underline{\gamma}_{\mathrm{o}}=\overline{\gamma}_{\mathrm{o}}=0$, $\epsilon_{\mathrm{o}}=1$. 
Then, the solution $\alpha_N$ to the LP~\eqref{eq:LP} satisfies
\begin{align}
\label{eq:hat_alpha_explicit_grune}
&\alpha_N\geq \hat{\alpha}_N:=1-\dfrac{ (\gamma_N-1)\prod_{j=2}^N(\gamma_j-1)}{\prod_{j=2}^N\gamma_j -\prod_{j=2}^N(\gamma_j-1)}.
\end{align}
Furthermore, if $\gamma_k=\sum_{j=0}^{k-1}c_j$, $c_{k+k_2}\leq c_{k}c_{k_1}$, $c_k\geq 0$, $k,k_2\in\mathbb{I}_{\geq 1}$, then $\hat{\alpha}_N=\alpha_N$.
\end{theorem}
This result follows directly from~\cite[Thm.~5.4]{grune2010analysis}, \change{compare also the proof in Appendix~\ref{sec:app_2}, based on Appendix~\ref{app:analytic_terminal_pdf}.}

The main benefit of Theorems~\ref{thm:analytic_sigma_W}--\ref{thm:analytic_grune} is that the analytical expressions provide a compact and easy to evaluate condition and hence can also provide simpler guidelines in terms of the design of the stage cost. 
In particular, in Theorem~\ref{thm:analytic_sigma_W} the suboptimality gap $1-\alpha_N$ and the stabilizing horizon $\underline{N}$ mainly depend on the fraction $\frac{\eta+\overline{\gamma}}{1+\overline{\gamma}}=1-\frac{\epsilon_{\mathrm{o}}}{1+\overline{\gamma}}\in[0,1)$ while the bound in Theorem~\ref{thm:analytic_grune} depends on $\frac{\overline{\gamma}-1}{\overline{\gamma}}=1-\frac{1}{\overline{\gamma}}$. 
Hence, comparing the two bounds, Theorem~\ref{thm:analytic_sigma_W} provides insights how the cost-detectability constant $\epsilon_{\mathrm{o}}$ affects the stability and performance guarantees. 

Note that Theorem~\ref{thm:analytic_grune} showed $\hat{\alpha}_N=\alpha_N$ under the submultiplicativity condition ($c_{k_1+k_2}\leq c_{k_1}c_{k_2}$), while in Theorem~\ref{thm:analytic_sigma_W} we only showed equivalence for the simple case $\gamma_k=\overline{\gamma}$. 
We also observed this identity in all the considered numerical examples using horizon dependent bounds $\gamma_k$ and we conjecture that it holds under rather mild conditions, although a corresponding condition and proof are beyond the scope of this paper.

\section{Anaysis with a general terminal cost}
 \label{sec:terminal} 
In the following, we extend the setup to account for a general terminal cost $V_{\mathrm{f}}:X\rightarrow\mathbb{R}_{\geq 0}$ in the MPC design.
In particular, we demonstrate how a suitably chosen terminal cost may allow for stability results with shorter horizons $N$, while at the same time a too large terminal cost deteriorates the closed-loop performance. 
The practical relevance of the second result is also related to the fact that some MPC approaches advocate a large terminal cost to improve stability (cf., e.g., \cite{tuna2006shorter,grimm2005model}; \cite{limon2006stability}; and \cite{grune2009analysis,grune2010analysis}) without analysing the adversarial effect on the performance (and robustness). 

We first introduce the considered conditions on the terminal cost (Sec.~\ref{sec:terminal_prelim}). 
Then, a simple stability and performance analysis is presented (Sec.~\ref{sec:terminal_theory}).  
Finally, we derive a corresponding LP analysis (Sec.~\ref{sec:terminal_LP}) and provide an analytical solution (Sec.~\ref{sec:terminal_analytic}). 
%
\subsection{Setup}
\label{sec:terminal_prelim}
The corresponding finite-horizon cost is given by $\mathcal{J}_{N,\mathrm{f}}(x,u)=\sum_{k=0}^{N-1}\ell(x_u(k,x),u(k))+V_{\mathrm{f}}(x_u(N,x))$, with the value function $V_{N,\mathrm{f}}(x)=\inf_{u\in\mathbb{U}^N}\mathcal{J}_{N,\mathrm{f}}(x,u)$ and a minimizer $u_{N,\mathrm{f},x}^*\in\mathbb{U}^N$. 
The corresponding MPC control law $\mu_{\mathrm{f}}(x)$, the closed-loop state and input $x_{\mu_{\mathrm{f}}}(k),u_{\mu_{\mathrm{f}}}(k)$, $k\in\mathbb{I}_{\geq 0}$, and the infinite-horizon closed-loop performance $\mathcal{J}_\infty^{\mu_{\mathrm{f}}}(x_0)=\sum_{k=0}^\infty \ell(x_{\mu_{\mathrm{f}}}(k),u_{\mu_{\mathrm{f}}}(k))$ are defined analogously to Section~\ref{sec:problem}.

Analogous to Assumption~\ref{ass:stab}, we assume that the value function with the terminal cost is linearly bounded. 
\begin{assumption}
\label{ass:stab_term} (Exponential cost controllability)
There exist constants $\bar{\gamma}_{\mathrm{f}}\geq 0$, $\gamma_{k,\mathrm{f}}\in[0,\bar{\gamma}_{\mathrm{f}}]$, $k\in\mathbb{I}_{\geq 1}$, such that 
\begin{align}
\label{eq:stab_term}
V_{k,\mathrm{f}}(x)\leq \gamma_{k,\mathrm{f}}\sigma(x),\quad \forall x\in X,\quad k\in\mathbb{I}_{\geq 1}.
\end{align}
\end{assumption}
We consider the following conditions for the terminal cost.
\begin{assumption}
\label{ass:term} (Terminal cost)
There exist constants $\underline{c}_{\mathrm{f}},\overline{c}_{\mathrm{f}},\epsilon_{\mathrm{f}}>0$, such that for any $x\in X$:
\begin{subequations}
\label{eq:term}
\begin{align}
\label{eq:term_1}
\underline{c}_{\mathrm{f}}\sigma(x)\leq V_{\mathrm{f}}(x)\leq&\overline{c}_{\mathrm{f}}\sigma(x),\\
\label{eq:term_2}
\min_{u\in\mathbb{U}}V_{\mathrm{f}}(f(x,u))+\ell(x,u)\leq& (1+\epsilon_{\mathrm{f}}) V_{\mathrm{f}}(x), \\
\label{eq:term_3}
\underline{c}_{\mathrm{f}}\leq \dfrac{\gamma_{1,\mathrm{f}}}{1+\epsilon_{\mathrm{f}}}\leq& \overline{c}_{\mathrm{f}}.
\end{align}
\end{subequations}
\end{assumption}
Condition~\eqref{eq:term_1} reflects the boundedness of the terminal cost. 
Condition~\eqref{eq:term_2} can be viewed as a relaxed CLF condition, where $\epsilon_{\mathrm{f}}=0$ recovers the standard CLF condition~\cite{rawlings2017model,limon2006stability}. 
Analogous relaxed CLF conditions are considered  in~\cite[A~3]{tuna2006shorter}, \cite[Prop.~2]{reble2012improved}, \cite[Prop.~4]{kohler2021stability}, compare also earlier conditions in~\cite[Ass.~5]{grimm2005model} and \cite[Lemma~3]{reble2012unifying}. 
In the limit $\epsilon_{\mathrm{f}}\rightarrow\infty$, Inequality~\eqref{eq:term_2} becomes inactive, relaxing the assumption to boundedness of the terminal cost. 
Inequalities~\eqref{eq:term_3} are non-restrictive and ensure that the derived analytical bounds are tight. 
In particular,  in case the lower bound does not hold, Inequality~\eqref{eq:term_2} remains valid with the smaller constant $\epsilon_{\mathrm{f}}:=\frac{\gamma_{1,\mathrm{f}}}{\underline{c}_{\mathrm{f}}}-1$.\footnote{%
We have $\min_{u\in\mathbb{U}}V_{\mathrm{f}}(f(x,u))+\ell(x,u)=V_{1,\mathrm{f}}(x)\leq \gamma_{1,\mathrm{f}}\sigma(x)\leq \dfrac{\gamma_{1,\mathrm{f}}}{\underline{c}_{\mathrm{f}}}V_{\mathrm{f}}(x)$ and hence Inequality~\eqref{eq:term_2} holds with the specified $\epsilon_{\mathrm{f}}$. 
Note that the resulting constant $\epsilon_{\mathrm{f}}$ can be negative. In this case, Theorem~\ref{thm:grimm_terminal} ensures asymptotic stability for any $N>0$. }
Similarly, in case the upper bound does not hold, we can choose the smaller constant $\gamma_{1,\mathrm{f}}:=(1+\epsilon_{\mathrm{f}})\overline{c}_{\mathrm{f}}$. 
In Section~\ref{sec:term_design}, we investigate some simple design methods for the terminal cost and compute corresponding constants $\epsilon_{\mathrm{f}}$.

\begin{remark} 
\label{rk:terminal_set}
(Terminal set constraints)
In most of the related literature (cf.~\cite{rawlings2017model,limon2006stability,grune2008infinite,magni2001stabilizing,reble2012unifying,kohler2021stability})
the (approx.) CLF condition~\eqref{eq:term_2} is only assumed in some local terminal set. 
Similar to~\cite{limon2006stability,kohler2021stability}, such relaxed local assumptions can be considered by noting that the terminal set constraint is implicitly satisfied on some sublevel sets of the value/Lyapunov function, compare also the discussion in Remark~\ref{rk:constraints}.
\end{remark}

\subsection{Stability and performance analysis}
\label{sec:terminal_theory}
The following theorem extends the stability and performance analysis in Theorem~\ref{thm:grimm} to general terminal costs.  
\begin{theorem}
\label{thm:grimm_terminal}
Let Assumptions~\ref{ass:set}, \ref{ass:detect}, \ref{ass:stab_term}, and \ref{ass:term} hold. 
Then, for any $x\in X$, the function $Y_{N,\mathrm{f}}:=V_{N,\mathrm{f}}+W$ satisfies
\begin{subequations}
\label{eq:Lyap_terminal}
\begin{align}
\label{eq:Lyap_terminal_1}
\epsilon_{\mathrm{o}}\sigma(x)\leq Y_{N,\mathrm{f}}(x)\leq& (\gamma_{N,\mathrm{f}}+\gamma_{\mathrm{o}})\sigma(x),\\
\label{eq:Lyap_terminal_2}
Y_{N,\mathrm{f}}(f(x,\mu_{\mathrm{f}}(x)))-Y_{N,\mathrm{f}}(x)\leq &-\epsilon_{\mathrm{o}}\alpha_{N,\mathrm{f}} \sigma(x),
\end{align}
\end{subequations}
with $\alpha_{N,\mathrm{f}}=:1-\dfrac{(\gamma_{N,\mathrm{f}}+\gamma_{\mathrm{o}})\epsilon_{\mathrm{f}}\overline{\gamma}_{\mathrm{f}}}{\epsilon_{\mathrm{o}}((N-1)\epsilon_{\mathrm{o}}(1+\epsilon_{\mathrm{f}})+\overline{\gamma}_{\mathrm{f}})}$. 
If additionally, $N>\underline{N}_{\mathrm{f}}:=1+\dfrac{\epsilon_{\mathrm{f}}}{1+\epsilon_{\mathrm{f}}}\dfrac{\overline{\gamma}_{\mathrm{f}}(\overline{\gamma}_{\mathrm{f}}+\gamma_{\mathrm{o}})}{\epsilon_{\mathrm{o}}^2}-\dfrac{\overline{\gamma}_{\mathrm{f}}}{\epsilon_{\mathrm{o}}(1+\epsilon_{\mathrm{f}})}$, then $\alpha_{N,\mathrm{f}}\in(0,1]$. 
Furthermore, for any $x_0\in X$ the set $\mathcal{A}$ is asymptotically stable for the corresponding closed loop $x_{\mu_{\mathrm{f}}}(\cdot)$ and the following performance bound holds
\begin{align}
\label{eq:performance_terminal}
&\alpha_{N,\mathrm{f}}(\mathcal{J}^{\mu_{\mathrm{f}}}_\infty(x_0)+W(x_0))\nonumber\\
\leq&\left[1+\frac{\overline{c}_{\mathrm{f}}}{\epsilon_{\mathrm{o}}}\left[1-\frac{\epsilon_{\mathrm{o}}}{\overline{\gamma}_{\mathrm{f}}+\gamma_{\mathrm{o}}}\right]^N \right]Y_{\infty}(x_0).
\end{align}
\end{theorem}
\change{The proof is detailed in Appendix~\ref{app:grimm_terminal}.} 
In case the terminal cost $V_{\mathrm{f}}$ is a very bad approximation of a CLF ({$\epsilon_{\mathrm{f}} \gg 1$), the bounds on the prediction horizon $N$ to ensure asymptotic stability essentially reduce to the bounds in Theorem~\ref{thm:grimm} without a terminal cost (assuming $\overline{\gamma}\approx\overline{\gamma}_{\mathrm{f}}$). 
In case the terminal cost is a sufficiently close approximation to the CLF ($\epsilon_{\mathrm{f}}<\frac{\epsilon_{\mathrm{o}}}{\overline{\gamma}_{\mathrm{f}}+\gamma_{\mathrm{o}}}$), we can ensure asymptotic stability for any prediction horizon $N\geq 1$. 
This improves the standard results in~\cite{rawlings2017model,limon2006stability}, which assume $\epsilon_{\mathrm{f}}=0$. 

In addition to the possible reduction of the stabilizing prediction horizon $N$, Inequality~\eqref{eq:performance_terminal} shows how a large terminal cost ($\overline{c}_{\mathrm{f}}\geq 1$) may deteriorate the closed-loop performance guarantees.  
In case of positive definite stage costs $\ell$, similar performance bounds capturing the effect of the terminal cost can be found in~\cite[Prop.~7]{jadbabaie2005stability}, \cite[Thm.~6.2]{grune2008infinite},  and \cite[Thm.~5]{kohler2021stability}. 
Further, in case $V_{\mathrm{f}}$ is a CLF ($\epsilon_{\mathrm{f}}=0$), we have $\alpha_{N,\mathrm{f}}=1$ and the performance bound~\eqref{eq:performance_terminal} is analogous to the exponentially decaying \textit{regret} bounds in~\cite{zhang2021regret} for unconstrained LQR.

 \subsection{Linear programming analysis}
\label{sec:terminal_LP}
The following theorem improves the quantitative bounds in Theorem~\ref{thm:grimm_terminal} by extending the LP analysis from Theorem~\ref{thm:main}. 
\begin{theorem}
\label{thm:terminal_LP}
Let Assumptions~\ref{ass:detect}, \ref{ass:stab_term}, and  \ref{ass:term} hold. 
Then, for any $x,x_0\in X$, Inequalities~\eqref{eq:Lyap_terminal} and \eqref{eq:performance_terminal} hold with $\alpha_{N,\mathrm{f}}$ according to the following LP:
\begin{subequations}
\label{eq:LP_terminal}
\begin{align}
\label{eq:LP_terminal_cost}
&(\alpha_{N,\mathrm{f}}-1)\epsilon_{\mathrm{o}}\nonumber\\
:=&\min_{\tilde{\ell},\tilde{W},\tilde{\sigma},\tilde{V},\tilde{V}_{\mathrm{f}}}
\sum_{k=1}^{N-1}\tilde{\ell}_k+\tilde{V}_{\mathrm{f}}-\tilde{V}\\
\label{eq:LP_terminal_nonneg_1}
\mathrm{s.t. ~}&\tilde{\sigma}_0=1,~\tilde{\ell}_k\geq 0,~k\in\mathbb{I}_{[0,N-1]},~\tilde{V}_{\mathrm{f}}\geq 0,\\
\label{eq:LP_terminal_nonneg_2}
&\tilde{\sigma}_k\geq 0,~\tilde{W}_k\geq 0,~k\in\mathbb{I}_{[0,N]},\\
\label{eq:LP_terminal_W_bound}
&\underline{\gamma}_{\mathrm{o}}\tilde{\sigma}_k\leq \tilde{W}_k\leq \overline{\gamma}_{\mathrm{o}}\tilde{\sigma}_k,~k\in\mathbb{I}_{[0,N]},\\
\label{eq:LP_terminal_W_decrease}
&\tilde{W}_{k+1}-\tilde{W}_k\leq -\epsilon_{\mathrm{o}}\tilde{\sigma}_k+\tilde{\ell}_k,~k\in\mathbb{I}_{[0,N-1]},\\
\label{eq:LP_terminal_V_bound}
&\sum_{j=k}^{N-1}\tilde{\ell}_j+\tilde{V}_{\mathrm{f}}\leq \gamma_{N-k,\mathrm{f}}\tilde{\sigma}_k,~k\in\mathbb{I}_{[0,N-1]},\\
\label{eq:LP_terminal_V_next_bound}
&\tilde{V}\leq \sum_{j=1}^{k-1}\tilde{\ell}_j+\gamma_{N-k+1,\mathrm{f}}\tilde{\sigma}_k,~k\in\mathbb{I}_{[1,N]},\\
\label{eq:LP_terminal_Vf_next_bound}
&\tilde{V}\leq \sum_{j=1}^{N-1}\tilde{\ell}_j+(1+\epsilon_{\mathrm{f}})\tilde{V}_{\mathrm{f}},\\
\label{eq:LP_terminal_Vf_bound}
&\underline{c}_{\mathrm{f}}\tilde{\sigma}_N\leq\tilde{V}_{\mathrm{f}}\leq \overline{c}_{\mathrm{f}}\tilde{\sigma}_N. 
\end{align}
\end{subequations}
\end{theorem}
\change{The proof is detailed in Appendix~\ref{app:terminal_LP}.} 
Similar to Theorem~\ref{thm:main}, the LP~\eqref{eq:LP_terminal} computes  a worst-case feasible sequence for the decrease of the value function $V_N$, while additional taking the terminal cost $V_{\mathrm{f}}$ (Ass.~\ref{ass:term}) into account. 
Compared to the LP~\eqref{eq:LP}, the main difference is the fact that the terminal cost $\tilde{V}_{\mathrm{f}}$ appears in the cost function~\eqref{eq:LP_terminal_cost} and in the candidate solution~\eqref{eq:LP_terminal_Vf_next_bound} with the factor $\epsilon_{\mathrm{f}}$.

 \subsection{Analytical solution}
\label{sec:terminal_analytic} 
In the following, Theorems~\ref{thm:analytic_terminal_pdf} and \ref{thm:analytic_terminal_sigma_W} provide an analytical solution for the LP from Theorem~\ref{thm:terminal_LP} for the special cases that $\tilde{\sigma}=\tilde{\ell}$ and $\tilde{\sigma}=\tilde{W}$, respectively. 
%
\begin{theorem}
\label{thm:analytic_terminal_pdf}
Suppose $\underline{\gamma}_{\mathrm{o}}=\overline{\gamma}_{\mathrm{o}}=0$, $\epsilon_{\mathrm{o}}=1$.  
Then, the solution $\alpha_{N,\mathrm{f}}$ to the LP~\eqref{eq:LP_terminal} satisfies
\begin{align}
\label{eq:analytic_terminal_pdf}
&\alpha_{N,\mathrm{f}}\geq \hat{\alpha}_{N,\mathrm{f}}\\
:=&1-\dfrac{\epsilon_{\mathrm{f}}(\gamma_{N,\mathrm{f}}-1)\prod_{j=1}^{N-1}(\gamma_{N-j+1,\mathrm{f}}-1)}{(1+\epsilon_{\mathrm{f}})\prod_{j=1}^{N-1}\gamma_{N-j+1,\mathrm{f}}-\epsilon_{\mathrm{f}}\prod_{j=1}^{N-1}(\gamma_{N-j+1,\mathrm{f}}-1)}.\nonumber
\end{align}
Furthermore, $\hat{\alpha}_{N,\mathrm{f}}=\alpha_{N,\mathrm{f}}$ if for all $k,k_2\in\mathbb{I}_{\geq 0}$: 
\begin{subequations}
\label{eq:submult_pdf_term}
\begin{align}
\label{eq:submult_pdf_term_12}
\gamma_{k,\mathrm{f}}=\sum_{j=0}^{k-1}c_j+c_{k,\mathrm{f}},~ 
&c_{k+k_2}\leq c_{k}c_{k_2},~ 
c_{k+k_2,\mathrm{f}}\leq c_{k}c_{k_2,\mathrm{f}},\\
\label{eq:submult_pdf_term_3}
&c_{k}+c_{k+1,\mathrm{f}}\leq (1+\epsilon_{\mathrm{f}})c_{k,\mathrm{f}}.
\end{align}
\end{subequations}
In addition, we have $\alpha_{N,\mathrm{f}}>0$ for 
\begin{align}
\label{eq:hat_alpha_explicit_N_terminal_pdf}
\change{N>\underline{N}_{\mathrm{f}}:=1+\dfrac{\log(\overline{\gamma}_{\mathrm{f}})-\log(1+1/\epsilon_{\mathrm{f}})}{\log(\overline{\gamma}_{\mathrm{f}})-\log(\overline{\gamma}_{\mathrm{f}}-1)}.}
\end{align} 
\end{theorem}
\change{The proof is detailed in Appendix~\ref{app:analytic_terminal_pdf}.} 
Regarding the conditions~\eqref{eq:submult_pdf_term} to ensure $\hat{\alpha}_{N,\mathrm{f}}=\alpha_{N,\mathrm{f}}$: Inequalities~\eqref{eq:submult_pdf_term_12} correspond to the  the submultiplicativity condition in~\cite[Equ. (3.3)]{grune2010analysis} (which holds, e.g., if $c_k,c_{k,\mathrm{f}}$ are exponentially decaying in $k$), while Inequality~\eqref{eq:submult_pdf_term_3} reflects that these bounds also satisfy Inequality~\eqref{eq:term_2}.

%
\begin{theorem}
\label{thm:analytic_terminal_sigma_W}
Suppose $\underline{\gamma}_{\mathrm{o}}=\overline{\gamma}_{\mathrm{o}}=1$, and $\epsilon_{\mathrm{o}}\in(0,1)$. 
Then, the solution $\alpha_{N,\mathrm{f}}$ to the LP~\eqref{eq:LP_terminal} satisfies
\begin{align}
\label{eq:analytic_terminal_sigma_W}
&(1-\alpha_{N,\mathrm{f}})\epsilon_{\mathrm{o}}\leq(1- \hat{\alpha}_{N,\mathrm{f}})\epsilon_{\mathrm{o}}\\
:=&\dfrac{\epsilon_{\mathrm{f}}\gamma_{1,\mathrm{f}}(\gamma_{N,\mathrm{f}}+\eta)
\prod_{j=0}^{N-2}(\eta+\gamma_{N-j,\mathrm{f}})}{(1+\epsilon_{\mathrm{f}})\prod_{j=0}^{N-1}(1+\gamma_{N-j,\mathrm{f}})
-\epsilon_{\mathrm{f}}\gamma_{1,\mathrm{f}}\prod_{j=0}^{N-2}(\eta+\gamma_{N-j,\mathrm{f}})}\nonumber
\end{align}
and $\alpha_{N,\mathrm{f}}>0$ holds if
\begin{align}
\label{eq:hat_alpha_explicit_N_terminal_detect}
N>\underline{N}_{\mathrm{f}}=
\dfrac{\log(\overline{\gamma}_{\mathrm{f}})
 -\log(\epsilon_{\mathrm{o}})-\log(1+{1}/{\epsilon_{\mathrm{f}}})}{\log(1+\overline{\gamma}_{\mathrm{f}})-\log(\overline{\gamma}_{\mathrm{f}}+\eta)}.
\end{align} 
If additionally $\gamma_{k,\mathrm{f}}=\overline{\gamma}_{\mathrm{f}}$, $\forall k\in\mathbb{I}_{\geq 1}$, then $\hat{\alpha}_{N,\mathrm{f}}=\alpha_{N,\mathrm{f}}$.  
\end{theorem}
\change{The proof is detailed in Appendix~\ref{app:analytic_terminal_sigma_W}.} 
In case we have a very  bad terminal cost ($\epsilon_{\mathrm{f}}\gg 1$), the derived bounds in Theorem~\ref{thm:analytic_terminal_pdf} and
\ref{thm:analytic_terminal_sigma_W} essentially recover the results without terminal cost in Theorem~\ref{thm:analytic_grune} and
\ref{thm:analytic_sigma_W}, respectively, assuming $\gamma_k\approx\gamma_{k,\mathrm{f}}$. 
Furthermore, if the terminal cost is a close approximation to a CLF ($\epsilon_{\mathrm{f}}$ small), we can ensure stability with a significantly shorter horizon $N$.

\section{Quadratic input-output costs}
 \label{sec:design}
In this section, we discuss the special case, in which $\ell$ is a positive definite function with respect to an output $y=h(x)$. 
\begin{assumption}(Stabilizability and detectability) There exist constants $C_0\geq 0$, $C_1\geq 1$, $\rho\in[0,1)$, $c_{\mathrm{o}}>0$, 
 such that the following conditions hold for any $(x,u)\in X\times\mathbb{U}$: 
\label{ass:exp_IO}
\begin{enumerate}[label=\alph*)]
\item Quadratic input-output stage cost: 
$\ell(x,u)=\|h(x)\|_Q^2+\|u\|_R^2$ with $h(0)=0$, $f(0,0)=0$, $Q,R$ positive definite and $h$ Lipschitz continous.
\label{ass:exp_IO_a}
\item Quadratically bounded terminal cost: 
$V_{\mathrm{f}}(x)\leq C_0\|x\|^2$. 
\label{ass:exp_IO_b}
\item Exponential stabilizability: 
For any $x\in X$, there exists an input trajectory $u\in\mathbb{U}^\infty$ satisfying $\|x_u(k,x)\|^2+\|u(k)\|^2\leq C_1\rho^k\|x\|^2$, $k\in\mathbb{I}_{\geq 0}$. 
\label{ass:exp_IO_c}
\item Exponential input-output-to-state stability (IOSS): There exists a quadratically lower and upper bounded IOSS Lyapunov function $\tilde{W}:X\rightarrow\mathbb{R}_{\geq 0}$ satisfying 
 $\tilde{W}(f(x,u))-\tilde{W}(x)\leq -\|x\|^2+c_{\mathrm{o}}(\|u\|^2+\|h(x)\|^2)$.
\label{ass:exp_IO_d}
\end{enumerate}
\end{assumption}
\begin{proposition}
\label{prop:exp_IO}
Let Assumption~\ref{ass:exp_IO} hold.
Then, Assumptions~\ref{ass:set}, \ref{ass:detect}, and  \ref{ass:stab_term} hold with  $\mathcal{A}=\{0\}$, $\sigma(x)=\|x\|^2$, $W(x)=\epsilon_{\mathrm{o}}\tilde{W}(x)$, $\gamma_{k,\mathrm{f}}=C_2\dfrac{1-\rho^k}{1-\rho}+C_0C_1\rho^k$, $\alpha_1(r)=\alpha_2(r)=r^2$, and apropriate constants $\epsilon_{\mathrm{o}},C_2>0$. 
\end{proposition}
\change{The proof is detailed in Appendix~\ref{app:exp_IO}.} 
In case of quadratic input output stage costs $\ell$, 
this proposition ensures that the overall conditions are (exponential) stabilizability and zero-state detectability/IOSS~\cite[Def.~3.7]{cai2008input}, which reduce
to stabilizability and detectability in case of linear systems.

\begin{remark}(Input-output models)
\label{rk:IO_model}
An important special case are nonlinear autoregressive models with exogenous inputs (NARX), in which case a simple (non-minimal) state is given by $x(t)=(Q^{1/2}y(t-\nu),\dots,Q^{1/2}y(t-1),R^{1/2}u(t-\nu),\dots,R^{1/2}u(t-1))$, where $\nu\in\mathbb{I}_{\geq 1}$ corresponds to the lag/observability-index. 
Such systems are intrinsically final-state observable (cf.~\cite[Def.~4.29]{rawlings2017model}). 
Thus, in case of quadratic input-output stage costs (cf. Ass.~\ref{ass:exp_IO}\ref{ass:exp_IO_a}), Assumption~\ref{ass:detect} holds with $\epsilon_{\mathrm{o}}=1/\nu$ and the following quadratic IOSS Lyapunov function
\begin{align*}
W(x(t))=&\|x(t)\|_{P_{\mathrm{o}}}^2\\
:=&\sum_{k=1}^\nu \frac{\nu+1-k}{\nu}\left(\|y(t-k)\|_Q^2+\|u(t-k)\|_R^2\right).
\end{align*}
\end{remark}

\begin{remark}(Large enough input regularization)
\label{rk:stable_R_large}
For open-loop asymptotically stable systems, one may intuitively conjecture that a large enough input regularization ensures asymptotic stability. 
Theorem~\ref{thm:analytic_sigma_W} quantifiably confirms this hypothesis. 
In particular, suppose we have an exponential input-to-state stable (ISS) system, i.e., there exists an ISS Lyapunov function $W(x)$ and an exponential decay rate $\eta\in[0,1)$, such that $W(f(x,u))\leq \eta W(x)+\|u\|^2$ for any $(x,u)\in X\times\mathbb{U}$. 
Given the stage cost $\ell(x,u)=\|h(x)\|_Q^2+r\|u\|^2$, $r>0$, with $h$ Lipschitz continuous, Assumption~\ref{ass:stab} holds with $\sigma=W$ and a uniform constant $\overline{\gamma}>0$ independent of $r$. 
Thus, by scaling $\tilde{W}(x)=\tilde{\sigma}(x)=r\cdot W(x)$, Assumptions~\ref{ass:stab} and \ref{ass:detect} hold with $\epsilon_{\mathrm{o}}=1-\eta>0$ and $\tilde{\overline{\gamma}}=\overline{\gamma}/r$ for all $r>0$. 
Hence, by choosing the input regularization $r> \overline{\gamma}/\epsilon_{\mathrm{o}}$, Theorem~\ref{thm:analytic_sigma_W} ensures asymptotic stability with any prediction horizon $N\geq1$. 
To the best knowledge of the authors, such quantitative design results are not available in the existing literature and can, in particular, not be concluded using the methods in~\cite{grune2009analysis,grune2010analysis}. 
\end{remark}

\begin{remark}
\label{rk:output_cost}
(Output cost)
In many applications, we may be interested in pure output stage costs $\ell(x,u)=\|h(x)\|_Q^2$, e.g., if only a desired output setpoint/trajectory is specified without a corresponding input (cf.~\cite{limon2018nonlinear,ferramosca2010mpc,Koehler2020Regulation}). 
However,  Assumption~\ref{ass:detect} is in general not valid for such stage costs, even if the output is observable (cf.~\cite[Example~1]{hoger2019relation}). 
In the following, we briefly discuss how the theoretical results in this paper can be applied to such output stage costs. 
Consider for simplicity a single-input-single-output system with a well-defined relative degree $d$ and suppose that the zero dynamics are exponentially stable.
Then, the stage cost $\tilde{\ell}(x,u)=\|h(x)\|_Q^2+\|h(x_u(d+1,x))\|_Q^2$ is cost-detectable, i.e., satisfies Assumption~\ref{ass:detect} with $\sigma(x)=\|x\|^2$ (cf.~\cite[Prop.~4]{Koehler2020Regulation}). 
Hence, if the detectable output stage cost $\tilde{\ell}$ is used, asymptotic stability can be directly analysed based on the derived results. 
Furthermore, one can show that the MPC scheme with the output cost $\ell$ is equivalent to an MPC scheme with stage cost $\tilde{\ell}$, horizon $N-d-1$, and some non-negative terminal cost (cf. \cite[Thm~2]{Koehler2020Regulation}).  
Consequently, the stability of MPC formulations with the singular output costs $\ell$ can also be analysed using the results for general terminal costs (Sec.~\ref{sec:terminal}). 
\end{remark}

\begin{remark}(Non-quadratic stage costs)
\label{rk:non_quad}
In case the system is \textit{not} exponentially stabilizable, the considered assumptions cannot be satisfied with a quadratically bounded state measure $\sigma$.  
Instead, a non-quadratic state measure $\sigma$ can, e.g., be found using the concept of homogeneity~\cite{tuna2006shorter,grimm2005model,coron2020model}. 
Consequently, this case typically necessitates the usage of a non-quadratic stage cost~\cite{worthmann2015model,muller2017quadratic}.  
\end{remark} 

\section{Terminal cost designs}
\label{sec:term_design}
In the following, we show how Assumption~\ref{ass:term} can be constructively satisfied with the simple terminal \change{cost} 
 from~\cite{grune2009analysis,grune2010analysis} or the finite-tail cost from~\cite{magni2001stabilizing,kohler2021stability}.
\begin{proposition}
\label{prop:simple_penalty}
Consider $V_{\mathrm{f}}(x)=\omega \sigma(x)$ with some $\omega>0$ and let Assumption~\ref{ass:stab_term} hold. 
Then, Assumption~\ref{ass:term} holds with $\underline{c}_{\mathrm{f}}=\overline{c}_{\mathrm{f}}=\omega$, $1+\epsilon_{\mathrm{f}}={\gamma_{1,\mathrm{f}}}/{\omega}$. 
\end{proposition}
\change{The proof is detailed in Appendix~\ref{app:simple_penalty}.} 
In case $\sigma=\ell_{\min}$, this terminal cost corresponds to the approach in~\cite[Thm.~5.4]{grune2010analysis} and 
in this special case, Theorem~\ref{thm:analytic_terminal_pdf} exactly recovers the existing LP results in~\cite{grune2010analysis} with\footnote{%
The definition of the horizon length differs by $1$ and thus the term corresponds to $\gamma_2-\omega$ in \cite{grune2010analysis}.} $\frac{\epsilon_{\mathrm{f}}}{1+\epsilon_{\mathrm{f}}}\gamma_{1,\mathrm{f}}=\gamma_{1,\mathrm{f}}-\omega$. 
Considering $\sigma(x)=\|x\|^2$, the formulas in Proposition~\ref{prop:exp_IO} yield $\lim_{\omega\rightarrow\infty}\epsilon_{\mathrm{f}}=C_1\rho-1$. 
Thus, unless $C_1$ is close to one (the terminal cost is approximately a CLF), the benefit of a large weight $\omega$ on the last state is limited.\footnote{%
These bounds improve if a multi-step implementation is used~\cite{grune2010analysis}.} 
An important special case corresponds to $\min_{u\in\mathbb{U}}\sigma(f(x,u))\leq \sigma(x)$, which is for example studied in~\cite{grimm2005model}. 
In this case $\lim_{\omega\rightarrow\infty}\epsilon_{\mathrm{f}}= 0$ and thus a high penalty $\omega$ yields stability with short horizons. 
This case is particularly relevant for Lyapunov stable systems, such as undamped oscillators or non-holonomic integrators. 
\begin{proposition}
\label{prop:finite_tail}
Suppose there exist constants $C_\ell\geq 1$, $\rho\in[0,1)$, and a control law $\kappa$, such that $\ell(x_u(k,x),u(k))\leq C_\ell\rho^k\ell_{\min}(x)$ for all $x\in X$, $k\in\mathbb{I}_{\geq 0}$ with $u(\cdot)=\kappa(x(\cdot))$. 
Consider $\sigma=\ell_{\min}$, $V_{\mathrm{f}}(x)=\sum_{k=0}^{M-1}\ell(x_u(k,x),u(k))$ with $u(\cdot)=\kappa(x_u(\cdot))$ and some $M\in\mathbb{I}_{\geq 1}$. 
Then, Assumption~\ref{ass:term} holds with $\overline{c}_{\mathrm{f}}=\frac{1-\rho^M}{1-\rho}C_\ell$, 
$\underline{c}_{\mathrm{f}}=1$, $\epsilon_{\mathrm{f}}=C_\ell\frac{1-\rho}{\rho^{-M}-1}$.  
\end{proposition}
\change{The proof is detailed in Appendix~\ref{app:finite_tail}.} 
The conditions used in the finite-tail cost construction essentially require $\ell$ to be positive definite (as also considered in~\cite{magni2001stabilizing,kohler2021stability}), since Inequality~\eqref{eq:term_2} in Assumption~\ref{ass:term} requires $V_{\mathrm{f}}$ to be positive definite. 
For $M=1$, the resulting bound $\epsilon_{\mathrm{f}}$ is quite similar to the simple terminal \change{cost} in Proposition~\ref{prop:simple_penalty}. 
By increasing the appended horizon $M$ the constant $\epsilon_{\mathrm{f}}$ can be made arbitrary small, in which case stability holds with arbitrary short horizons $N$. 
Note that typically $\rho\ll 1-{1}/{\overline{\gamma}}$ and thus increasing the extended horizon $M$ has a larger impact on the stability results in Theorem~\ref{thm:analytic_terminal_pdf} compared to the prediction horizon $N$ (cf.~\cite[Rk.~6]{kohler2021stability}). 

\section{Numerical examples}
\label{sec:num}
The following examples demonstrate how the derived stability results can be used to efficiently design a stabilizing MPC \change{for linear and nonlinear systems}. 
The corresponding code is available online.\footnote{%
gitlab.ethz.ch/ics/NMPC-StabilityAnalysis}
\subsection{Linear example - chain of mass-spring-dampers}
\label{sec:num_linear}
We consider a chain of $L=6$ mass-spring-damper systems, corresponding to a stable linear system with $x\in X=\mathbb{R}^{2\cdot L}=\mathbb{R}^{12}$, $u\in\mathbb{U}=[-1,1]$, compare Figure~\ref{fig:illustration}.\footnote{%
 The control input $u$ acts as a force on the last mass. The first mass is connected to a rigid wall. 
Each subsystem has mass of $m=1$, spring constant $k=10$, damping constant $d=2$, resulting in a weakly damped system with $4\cdot m\cdot d\ll k^2$. 
The state is given by $x=[z_1,\dot{z}_1,\dots,z_L,\dot{z}_L]^\top$. 
The discrete-time system is defined using an exact discretization with a sampling time of $h=1$.}
This system is severely underactuated, with medium scale ($n>10$), and hence also poses a challenge in the standard design of terminal ingredients. 
\change{To be more specific, typically a terminal cost and terminal set constraint are chosen based on the LQR with the maximal positive invariant set. However, the corresponding polytopic set becomes intractable for medium scale systems and both components destroy the sparsity of the optimal control problem, which is crucial to allow for efficient distributed solutions, compare also the motivation in~\cite{giselsson2013feasibility,conte2016distributed}. 
Furthermore, for such underactuated systems, a terminal set constraint can result in a small region of attraction or computationally infeasible requirements on the prediction horizon $N$.} 
\begin{figure}[hbtp]
\includegraphics[width=0.45\textwidth]{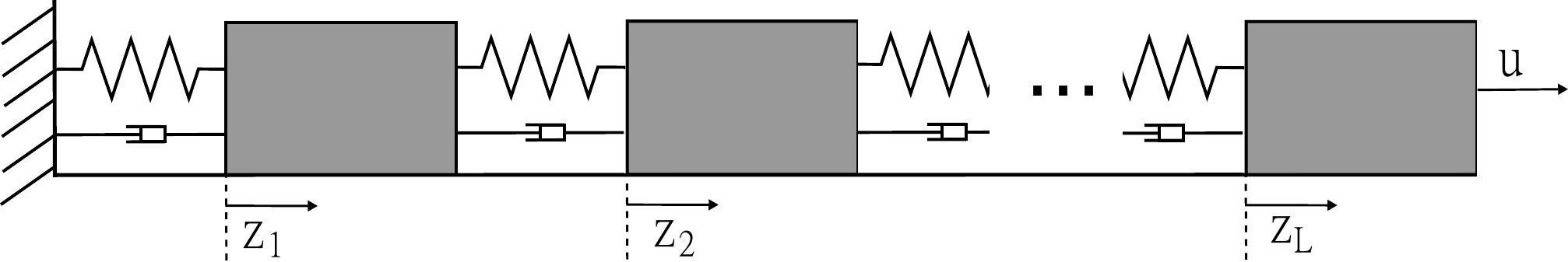}
\caption{Chain of masses with positions $z_i$, $i\in\mathbb{I}_{[1,L]}$ connected by spring and damper elements.}
\label{fig:illustration}
\end{figure}

\subsubsection*{\change{Short horizons can destabilize the system}}
The overall goal is to drive the first mass to the origin using the stage cost $\ell(x,u)=z_1^2+q\|x\|^2+r\|u\|^2$, where $z_1$ is the position of the first mass and $q,r>0$ are tuning variables. 
First, we study the local stability of the MPC using the finite-horizon LQR formula. 
Using $q=q_0=10^{-4}$ and $r=r_0=10^{-5}$, we find that the closed loop is unstable for $N=2$, locally asymptotically stable for $N=3$, and again unstable for $N=4,5$. 
Hence, depending on the choice of the design variables, the closed loop may be unstable and due to the non-monotonicity w.r.t. $N$, a direct tuning based on simulations may be non-trivial.

\subsubsection*{\change{Cost controllability and detectability}}
\change{In the following, we discuss satisfaction of Assumptions~\ref{ass:stab} and \ref{ass:detect}. 
To this end, we distinguish the two options:  
\begin{itemize}
\item
Directly utilizing positive definiteness of the stage cost $\ell$ without using any storage function, i.e., $\sigma=\ell_{\min}$, $W=0$ (Thm.~\ref{thm:analytic_grune}/\ref{thm:analytic_terminal_pdf})
\item Using a separate positive definite storage function with $\sigma=W$ (Thm.~\ref{thm:analytic_sigma_W}/\ref{thm:analytic_terminal_sigma_W}).
\end{itemize}
The system is open-loop stable with a maximal absolute eigenvalue of $0.943$.
As such, for a given (quadratic) state measure $\sigma$, the smallest constants $\gamma_k$ satisfying Assumption~\ref{ass:stab} can be computed with the open-loop control $u=0$ and the finite-horizon Lyapunov equation. 
For a fixed value $\epsilon_{\mathrm{o}}>0$, the "best" quadratic storage $W=\|x\|_{P_{\mathrm{o}}}^2$ satisfying Inequality~\eqref{eq:detect_2} with $\sigma=W$ can be computed using a semi-definite program (SDP), which minimizes $\overline{\gamma}$. 
We use $4$ logarithmic grid points for $\epsilon_{\mathrm{o}}\in[10^{-3},1-10^{-4}]$ and choose the one with the smallest bound~\eqref{eq:hat_alpha_explicit_N}.}

\subsubsection*{\change{Efficient stage cost tuning}}
\change{In the following, we use the results based on positive definite stage costs ($\sigma=\ell$, Thm.~\ref{thm:analytic_grune}) and detectable stage costs ($\sigma=W$, Thm.~\ref{thm:analytic_sigma_W}) to find suitable parameters $q,r,N$ that guarantee global asymptotic stability.}
In addition, we utilize the more general results in Theorems~\ref{thm:analytic_terminal_pdf} and \ref{thm:analytic_terminal_sigma_W} with terminal costs using a final state weighting (Prop.~\ref{prop:simple_penalty}) with $\omega=10$ and the finite-tail cost (Prop.~\ref{prop:finite_tail}) with $M=10$.\footnote{%
In Proposition~\ref{prop:simple_penalty}, we choose the penalty $\tilde{\omega}>0$ such that $\tilde{\omega}\sigma(x)\leq \omega\ell_{\min}(x)$, to ensure that the scaling is independent of $\sigma(x)$
}

Given the cost $q=q_0,r=r_0$, both theorems can only guarantee asymptotic stability with quite conservative bounds $\underline{N}\gg 10^5$. 
Based on the intuition of Theorem~\ref{thm:analytic_grune} \change{($\sigma=\ell$)}, this conservatism is due to the large cost-controllability constant $\overline{\gamma}>10^4$, which can be reduced by increasing the penalty $q$. 
On the other hand, Remark~\ref{rk:stable_R_large} based on Theorem~\ref{thm:analytic_sigma_W} \change{($\sigma=W$)} suggests that the guarantees improve if the input penalty $r$ increased. 
Hence, we consider different values $q$ and $r$ and compute the corresponding prediction horizon bounds $\underline{N}$, compare Figure~\ref{fig:qr}. 
\begin{figure}[hbtp]
\begin{center}
\includegraphics[width=0.4\textwidth]{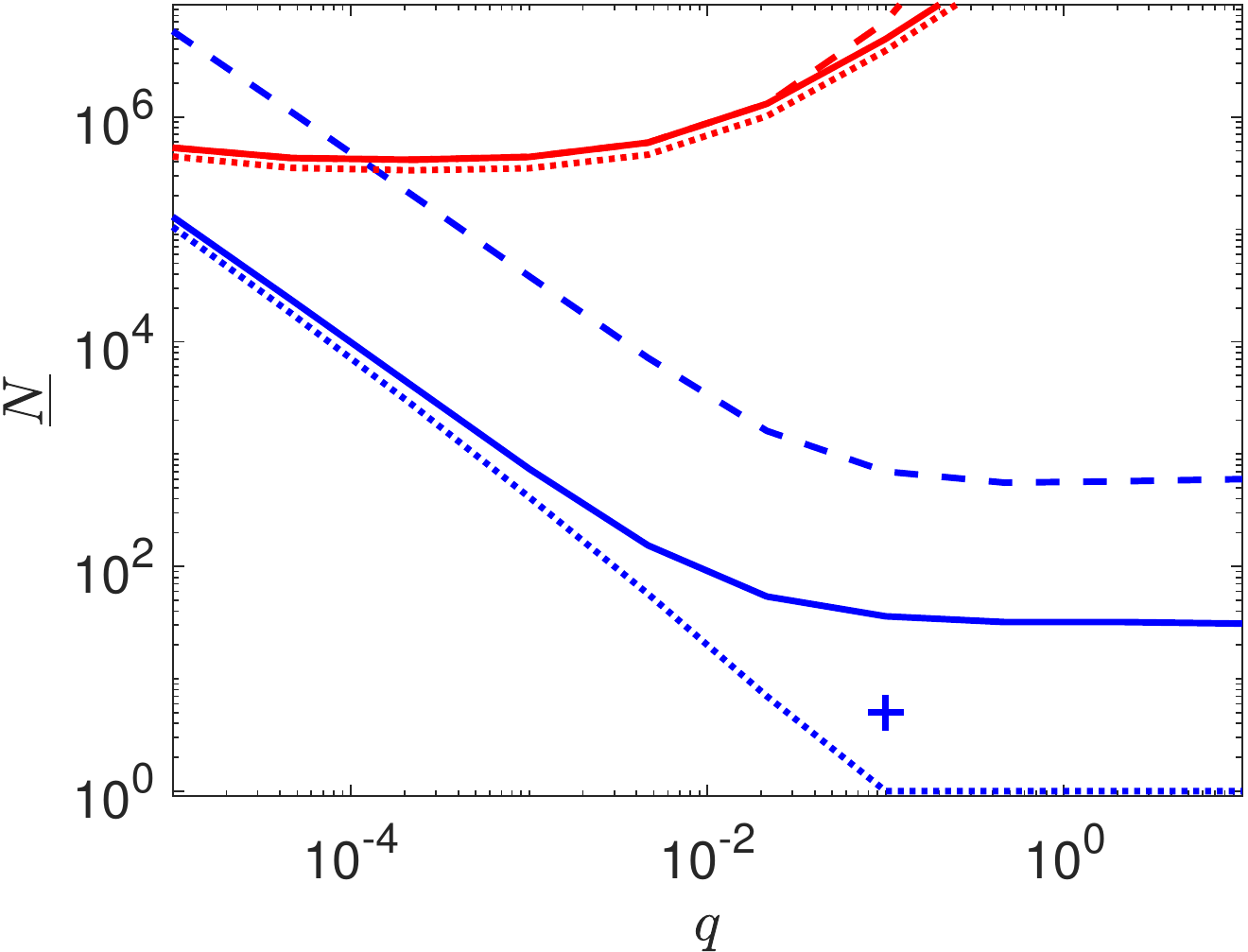}
\includegraphics[width=0.4\textwidth]{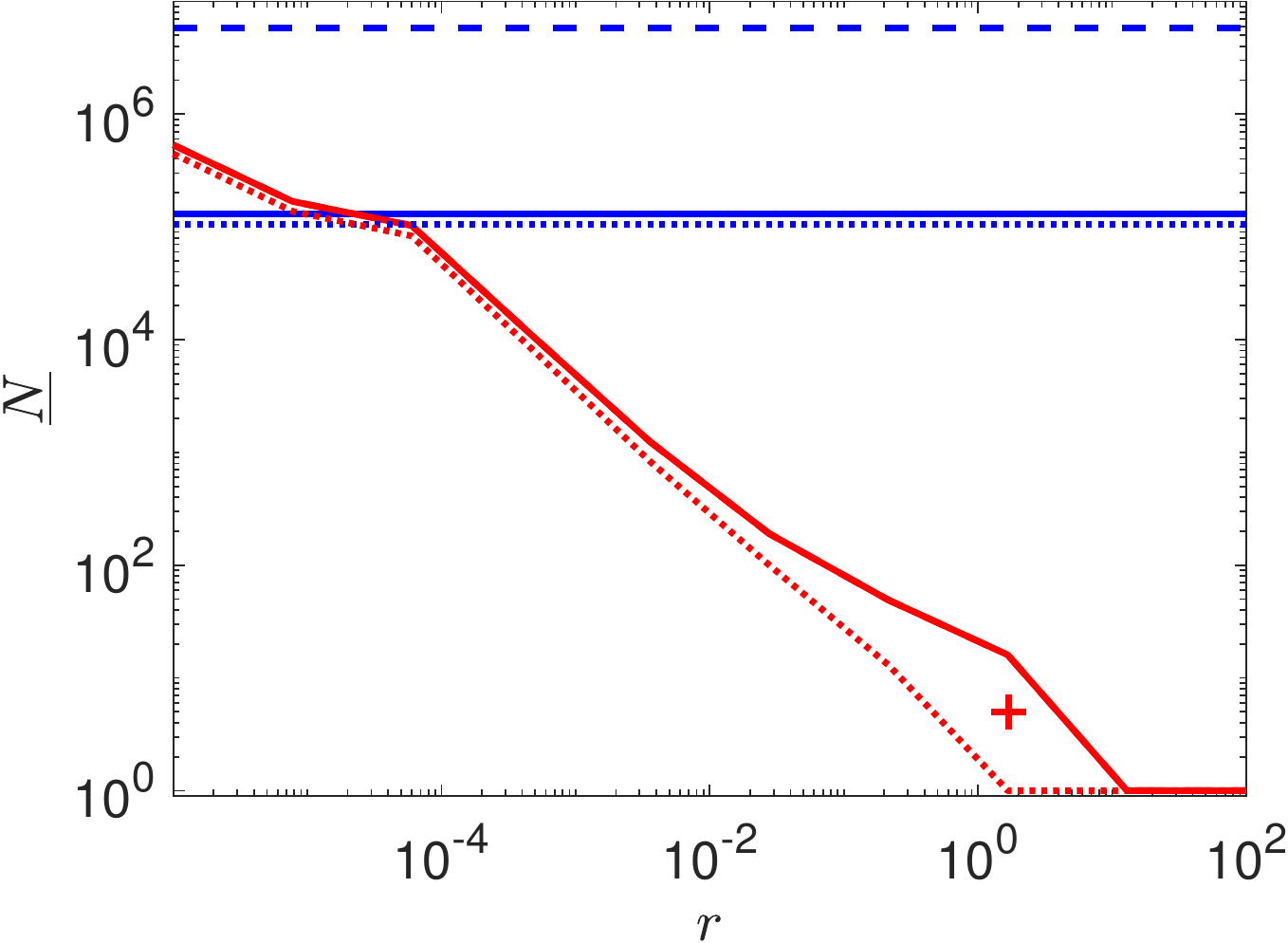}
\end{center}
\caption{\change{Chain of mass-spring-dampers:} Sufficient prediction horizon $\underline{N}$ for varying $q$ \change{(top)} and varying $r$ \change{(bottom)}, for the analysis with $\sigma=\ell$ (blue, \change{Thm.~\ref{thm:analytic_grune}/\ref{thm:analytic_terminal_pdf}}) or $\sigma=W$ (red, \change{Thm.~\ref{thm:analytic_sigma_W}/\ref{thm:analytic_terminal_sigma_W}}):
No terminal cost (solid), terminal weighting $\omega$ (dashed, \change{Prop.~\ref{prop:simple_penalty}}) and finite-tail cost (dotted, \change{Prop.~\ref{prop:finite_tail}}).
\change{The crosses highlight two designs with finite-tail cost that are guaranteed stable (since they are above the dotted line), which are used in the later closed-loop simulations.}}
\label{fig:qr}
\end{figure}

The guarantees provided by Theorem~\ref{thm:analytic_grune} \change{($\sigma=\ell$)} improve with increasing penalty $q$, while the guarantees in Theorem~\ref{thm:analytic_sigma_W} \change{($\sigma=W$)} deteriorate.
By choosing $q=10$, we can ensure global stability with a prediction horizon of $N=31$.
Considering an increasing input penalty $r$, the guarantees provided by Theorem~\ref{thm:analytic_grune} \change{($\sigma=\ell$)} remain invariant, while the bounds provided by Theorem~\ref{thm:analytic_sigma_W} \change{($\sigma=W$)} improve monotonically. 
In particular, with $r=13$, we can ensure asymptotic stability with $N=1$. 
For comparison, the simpler bounds from~\cite{grimm2005model,Koehler2020Regulation} (cf. Thm.~\ref{thm:grimm}), are significantly more conservative and require a horizon of $N=13$ with same parameters to conclude stability. 
In both cases, the simple terminal \change{cost} (Prop.~\ref{prop:simple_penalty}) does not significantly improve the results.
However, the finite-tail cost (Prop.~\ref{prop:finite_tail}) allows for significantly shorter horizons $N$ or more degree of freedom in the cost tuning.

\subsubsection*{\change{Closed-loop simulation}}
For illustration, we generated exemplary closed-loop trajectories (cf. Figure~\ref{fig:cl}) 
using $N=5$ and the following design choices: 
\begin{itemize}
\item \change{Based on Theorem~\ref{thm:analytic_terminal_pdf} ($\sigma=\ell$), $q=10^{-1}$,   $r=r_0$  with $V_{\mathrm{f}}$ using the finite-tail cost (Prop.~\ref{prop:finite_tail}).} 
\item \change{Based on Theorem~\ref{thm:analytic_sigma_W} ($\sigma=W$), $q=q_0$, $r=1.7$, with $V_{\mathrm{f}}$ using the finite-tail cost (Prop.~\ref{prop:finite_tail}).}
\item The unstable parametrization $q=q_0$, $r=r_0$, $V_{\mathrm{f}}=0$. 
\end{itemize}
\begin{figure}[hbtp] 
\begin{center} 
\includegraphics[width=0.4\textwidth]{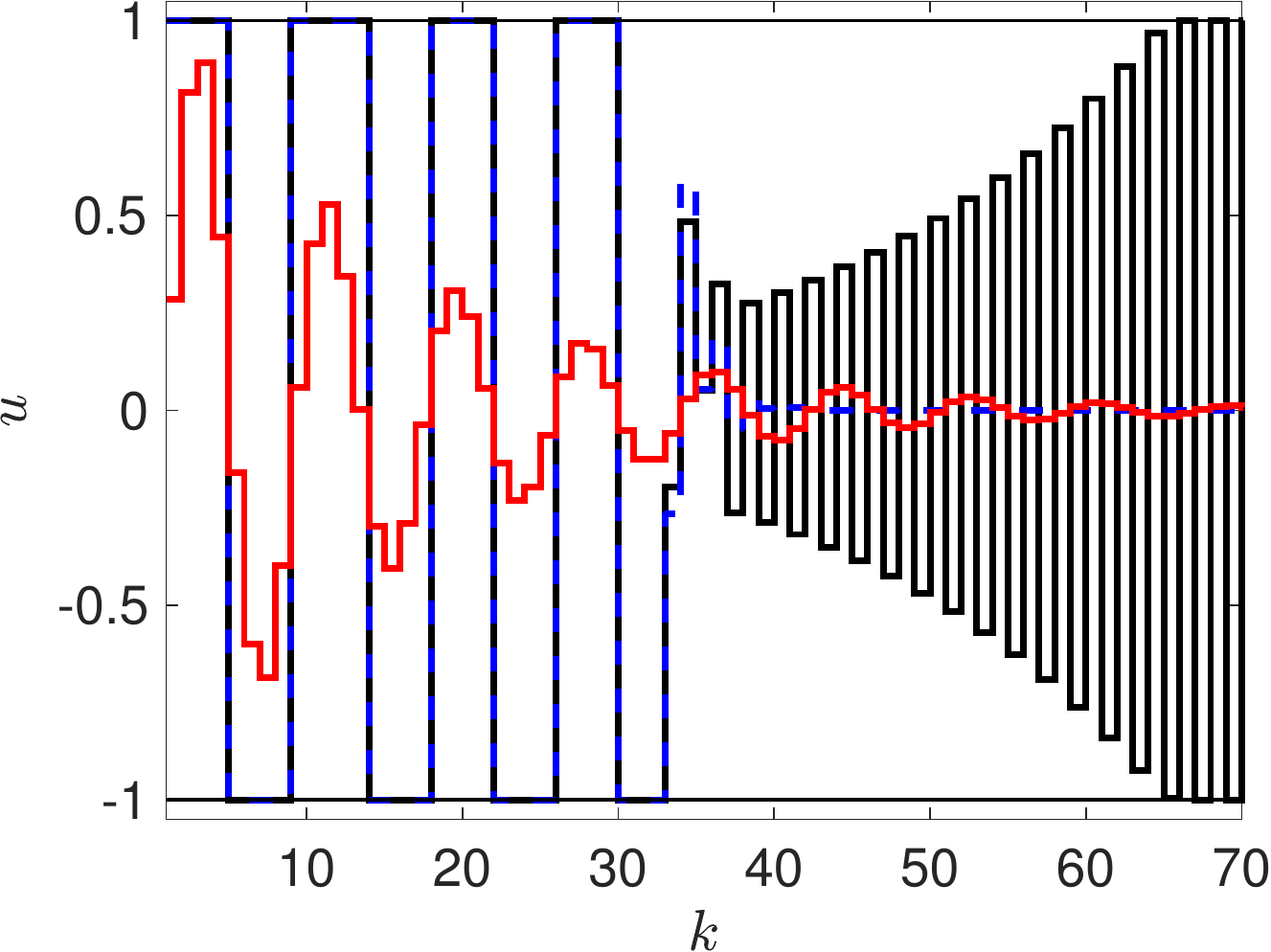}
\end{center}
\caption{\change{Chain of mass-spring-dampers: Closed-loop trajectories with horizon $N=5$ using different design parameters: Based on Theorem~\ref{thm:analytic_terminal_pdf} ($\sigma=\ell$, blue), Theorem~\ref{thm:analytic_sigma_W} ($\sigma=W$, red), and the unstable design (black).}}
\label{fig:cl}
\end{figure}
\change{The tuning based on Theorem~\ref{thm:analytic_terminal_pdf} ($\sigma=\ell$, blue) results in a closed loop that is very similar to the infinite-horizon optimal controller, which starts with a bang-bang control and quickly converges to the origin. 
The tuning based on Theorem~\ref{thm:analytic_terminal_sigma_W} ($\sigma=W$, red) uses a significantly larger input penalty $r$, which results in a slower convergence. 
The unstable design ($q=q_0,r=r_0$, black) is initially almost indistinguishable from an optimal controller.} 
However, eventually small oscillations increase in amplitude and the system converges to a limit cycle with a bang-bang control. 
If the MPC design/tuning were instead based on short simulations/experiments, one might easily conclude that the initial design is already well behaved, which demonstrates the relevance of the stability analysis provided in this paper.

\subsubsection*{\change{Summary}}
Overall, we have considered a simple linear example, where heuristic tuning approaches and standard design procedures for the terminal ingredients are challenging to apply. 
With the derived theory, we could efficiently tune the MPC cost to ensure asymptotic stability with short prediction horizons. 
The overall offline computation to determine the corresponding prediction horizons $\underline{N}$ and plots for all different combinations of $q,r$ and terminal ingredients required roughly $34$ seconds on a Laptop using Matlab with SeDuMi-1.3.

\subsection{Nonlinear example - four tank system}
\label{sec:num_nonlinear}
The following example demonstrates how the derived theoretical results can be similarly applied to nonlinear systems.
To this end, we consider the four-tank system from~\cite{raff2006nonlinear}, which can be compactly represented by:
\begin{align*}
\dot{x}_1=&-c_1\sqrt{x_1}+c_{12}\sqrt{x_2}+c_{1,\mathrm{u}}u_1,\quad 
\dot{x}_2=-c_2\sqrt{x}_2+c_{2,\mathrm{u}}u_2,\\
\dot{x}_3=&-c_3\sqrt{x_3}+c_{3,4}\sqrt{x_4}+c_{3,\mathrm{u}}u_2,\quad 
\dot{x}_4=-c_4\sqrt{x_4}+c_{4,\mathrm{u}}u_1,
\end{align*}
with positive constants $c_i$, $c_{i,\mathrm{u}}>0$, $i\in\mathbb{I}_{[1,4]}$, $c_{1,2},c_{3,4}>0$. 
Here, the states $x_i\geq 0$ corresponds to non-negative water levels and the inputs $u_i\geq 0$ represent the water flow. 
The model constants, constraints, and desired setpoint $x_{\mathrm{s}},u_{\mathrm{s}}$ are taken from~\cite{raff2006nonlinear}. 
This example is often used as a motivation for stability analysis of MPC schemes, as even reasonably long prediction horizons can result in an unstable closed loop (cf. \cite{raff2006nonlinear,kohler2021stability}). 
The system is discretized using  4th-order Runge-Kutta with a sampling time of $T_{\mathrm{s}}=3s$.

\subsubsection*{Unstable closed loop} 
We consider the stage cost $\ell(x,u)=\|y\|^2+r\|u-u_{\mathrm{s}}\|^2+q\|x-x_{\mathrm{s}}\|^2$, $r=10^{-2}$, $q\geq 0$, which primarily penalizes the output $y=[x_1-x_{\mathrm{s},1};x_3-x_{\mathrm{s},3}]^\top$, which has an unstable zero dynamics. 
Even for a relatively large horizon of $N=14$ (corresponding to $42$ seconds), a corresponding MPC scheme does not stabilize this system, unless an additional regularization $q\gg 0$ is added to the cost. 
In the following, we utilize the developed theoretical results, to derive stabilizing MPC formulations by systematically changing the prediction horizon $N$, the stage cost $\ell$, or adding a terminal cost $V_{\mathrm{f}}$.

\subsubsection*{Cost detectability and controllability}
In order to apply the theoretical results in Theorem~\ref{thm:analytic_sigma_W}/\ref{thm:analytic_terminal_sigma_W}, we need to find a non-trivial storage function $W$ certifying Assumption~\ref{ass:detect}. 
In principle, such a storage function $W$ can be constructed analogous to the linear example in Section~\ref{sec:num_linear}, i.e., consider a quadratic function $W(x)$ and pose Inequality~\eqref{eq:detect_2} as an SDP by using a linear difference inclusion. 
However, such methods often result in a significant engineering effort when considering more complex and larger scale nonlinear systems and is hence not pursued. 
Instead, we demonstrate how the relevant constants $\epsilon_{\mathrm{o}},\gamma_{\mathrm{o}},\gamma_k$ can be numerically computed/approximated without requiring a simple analytical/quadratic Lyapunov function. 
By investigating the dynamics, we can see that the output $y$ is observable with lag $\nu=2$ and hence we can study the MPC scheme using the non-minimal state $x(k)=[y(k-1)^\top,y(k-2)^\top,u(k-1)^\top,u(k-2)^\top]^\top\in\mathbb{R}^8$.\footnote{%
\change{If the minimal state $x=[x_1,x_2,x_3,x_4]^\top\in\mathbb{R}^4$ is considered, the past stage cost is not unique and one can instead define $W(x)=\min_{u\in\mathbb{U}^\nu,\tilde{x}\in X: x_u(\nu,\tilde{x})=x} \sum_{j=0}^{\nu-1}\dfrac{j+1}{\nu}\ell(x_{u}(j,\tilde{x}),u(j))$.
}}
Thus,  Assumption~\ref{ass:detect} holds with $W=\sigma$, $\underline{\gamma}_{\mathrm{o}}=\overline{\gamma}_{\mathrm{o}}=1$, $\epsilon_{\mathrm{o}}=1/\nu$, for any $q\geq 0$, using:
\begin{align}
\label{eq:four_tank_storage}
W(x(k))=\sum_{j=1}^{\nu}\dfrac{\nu+1-j}{\nu}\ell(x(k-j),u(k-j)), 
\end{align}
 compare also Remark~\ref{rk:IO_model}.
For the theoretical results in Theorem~\ref{thm:analytic_grune}/\ref{thm:analytic_terminal_pdf}, no storage function $W$ needs to be constructed, however, the results are only applicable if $\sigma=\ell_{\min}$ is positive definite, i.e., $q>0$. 
For both theoretical results, we numerically compute the cost controllability constants $\gamma_k\geq 0$, $k\in\mathbb{I}_{[0,300]}$ (Ass.~\ref{ass:stab}),  by considering a compact interval $x-x_{\mathrm{s}}\in[-5,5]^4$, gridding the set of initial states using $15^4\approx 5\cdot 10^4$ points, simulating the open-loop system using $u=u_{\mathrm{s}}$, and recoding the maximum cost.

\begin{figure}[hbtp]
\begin{center}
\includegraphics[width=0.4\textwidth]{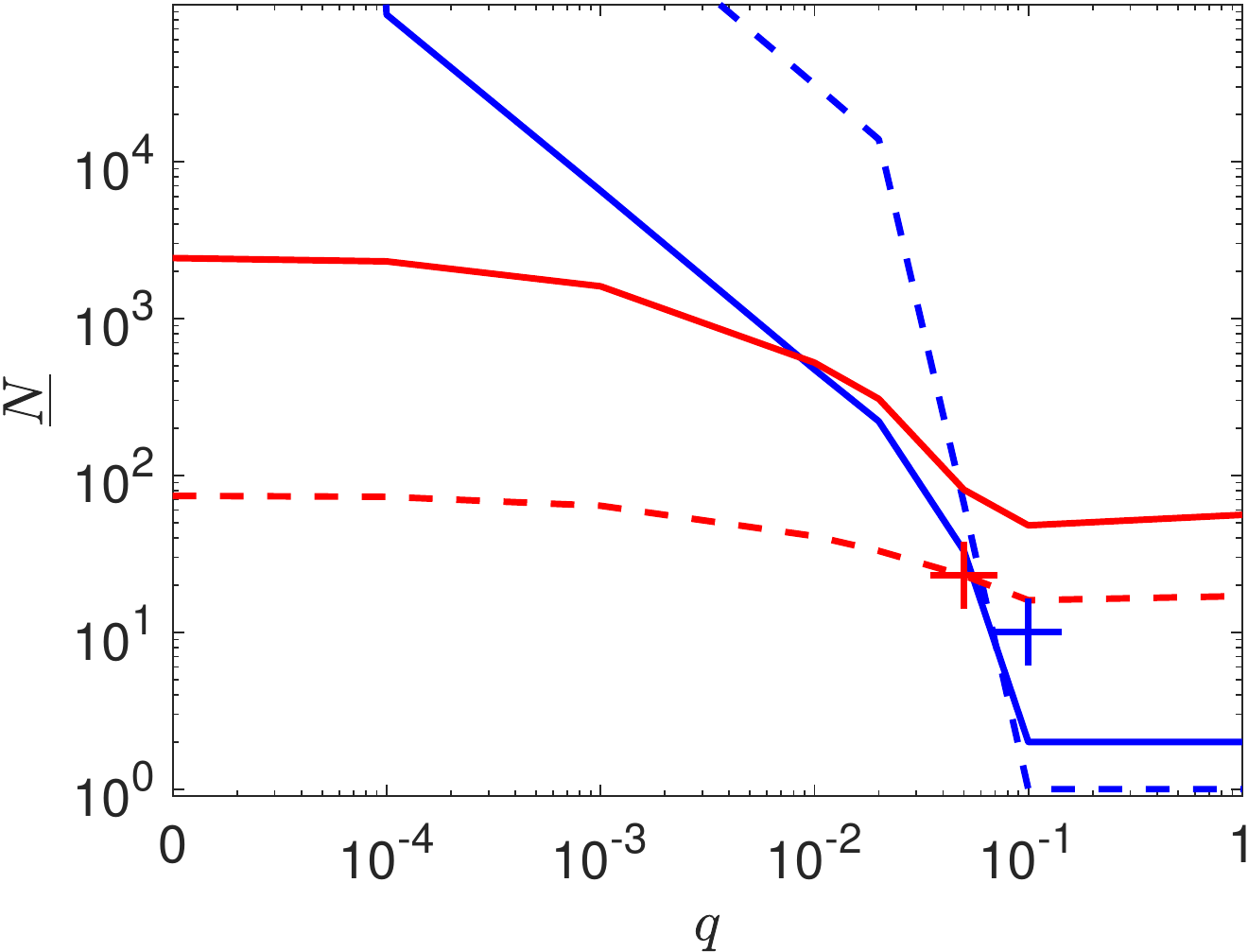}
\end{center}
\caption{\change{Four tank: Sufficient prediction horizon $\underline{N}$ for varying $q$ for the analysis with $\sigma=\ell$ (blue, \change{Thm.~\ref{thm:analytic_grune}/\ref{thm:analytic_terminal_pdf}}) or $\sigma=W$ (red, \change{Thm.~\ref{thm:analytic_sigma_W}/\ref{thm:analytic_terminal_sigma_W}}):
No terminal cost (solid) and terminal weighting $\omega$ (dashed, \change{Prop.~\ref{prop:simple_penalty}}).
The crosses highlight two guaranteed stable designs with the simple terminal cost which are used in the later closed-loop simulations.}}
\label{fig:fourtank_q}
\end{figure}
  
\subsubsection*{Efficient MPC tuning}
In the following, we use the theoretical results to find suitable MPC parametrization that ensures stability by adjusting the tunable weight $q\geq 0$ and investigate the simple\footnote{
\change{Note that also in case $\sigma=W$, based on Equation~\eqref{eq:four_tank_storage}, the terminal cost simply increases the weighting of the last $\nu=2$ stage costs in the horizon.  
}} terminal weighting $V_{\mathrm{f}}=\omega\sigma=\omega W$ (Prop.~\ref{prop:simple_penalty}) with $\omega=10^3$, compare Figure~\ref{fig:fourtank_q}. 

For $q\rightarrow 0$, the results based on positive definite stage costs $\sigma=\ell$ (Thm.~\ref{thm:analytic_grune}/\ref{thm:analytic_terminal_pdf}) deteriorate with $\underline{N}\rightarrow\infty$. 
The bounds using a storage function $\sigma=W$ are also valid for $q=0$ and with a terminal cost we can guarantee stability with a horizon of $\underline{N}=74$. 
As $q$ increases, both bounds improve and as $q\geq 0.1$, the results in based on Theorem~\ref{thm:analytic_grune}/\ref{thm:analytic_terminal_pdf} can even ensure stability with $\underline{N}=2/1$, since the terminal cost $V_{\mathrm{f}}=\omega \ell_{\min}$ results in a CLF, i.e., Assumption~\ref{ass:stab_term} holds with $\epsilon_{\mathrm{f}}=0$. 

\subsubsection*{Closed-loop simulation}
For illustration, we generated exemplary closed-loop trajectories (cf. Fig.~\ref{fig:q_tank_closedloop}).  
We see that the initial design fails to stabilize the desired steady-state, compare also the experiments in~\cite{raff2006nonlinear}, while both proposed designs successfully stabilize the desired steady-state. 
Notably, the design based on positive definite stage costs is stabilizing with a significantly shorter horizon, however, partially due to the larger required weighting ($q=0.1$ vs. $q=0.05$), the overall output tracking error $\|y\|^2$ is significantly larger.

\begin{figure}[hbtp]
\begin{center}
\includegraphics[width=0.4\textwidth]{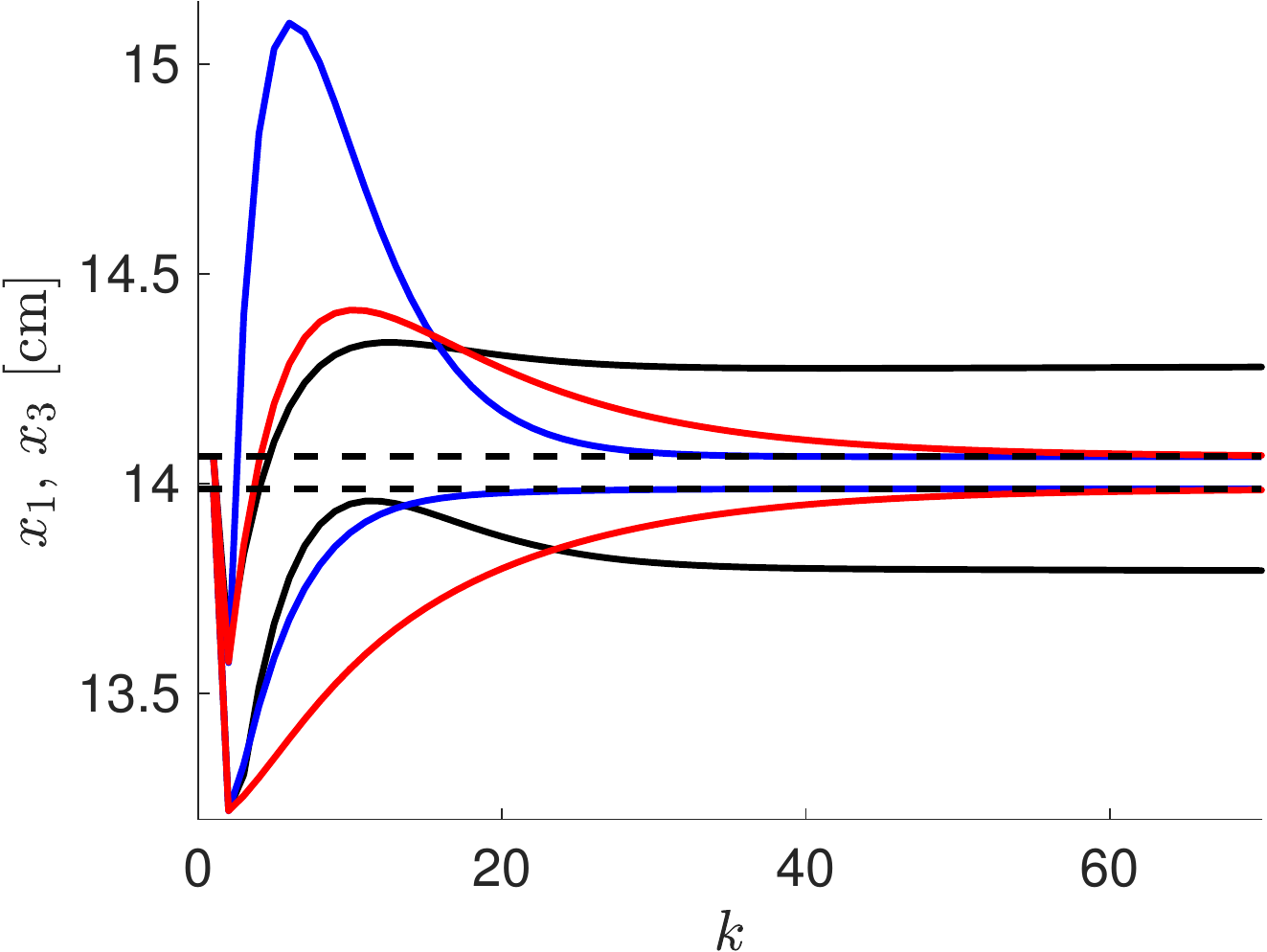}
\end{center}
\caption{\change{Four tank: Closed-loop trajectories of $x_1,x_3$ for initial (unstable) design (black, $q=0$, $N=14$); stabilizing MPC based on $\sigma=\ell$ with terminal cost (blue, $q=0.1$, $N=10$, $\omega=10^3$), and stabilizing MPC based on $\sigma=W$ with terminal cost (red, $q=0.05$, $N=23$, $\omega=10^3$). Reference steady-state $x_{\mathrm{s}}$ is dashed in black.}}
\label{fig:q_tank_closedloop}
\end{figure}

\subsubsection*{Summary}
Similar to the linear example (Sec.~\ref{sec:num_linear}), we have seen how the derived theory can be used to efficiently design a stabilizing NMPC scheme by properly choosing the prediction horizon $N$, the stage cost $\ell$, and possible a terminal cost $V_{\mathrm{f}}$. 
When considering nonlinear systems, the benefits of the derived theory become more pronounced as few alternative systematic designs are available. 
However, the numerical computation of the involved constants $\gamma_k$ becomes more challenging, especially when considering medium to large scale systems.  
In particular, even though the dimension of this nonlinear system is relatively small with $x\in\mathbb{R}^4$, the cost-controllability constants $\gamma_k$ were only numerically approximated on a (coarse) grid  and the overall computation times to derive the prediction horizons $\underline{N}$ for all considered parameter combinations took approximately 45 minutes using Matlab on a Laptop.

\section{Conclusion}
\label{sec:sum}
We have presented a stability and performance analysis for NMPC schemes with positive semi-definite (detectable) stage costs and a general positive-definite terminal costs using an LP analysis. 
The presented results recover existing results for positive definite cost functions~\cite{grune2009analysis,grune2010analysis} as a special case and drastically improve the existing bounds for detectable costs~\cite{grimm2005model,Koehler2020Regulation}. 

\appendix
Appendix~\ref{sec:app_1} contains the proofs for the general results with terminal costs (Sec.~\ref{sec:terminal}). 
Appendix~\ref{sec:app_2} shows that the results in Section~\ref{sec:theory} directly follow as a special case. 
Appendix~\ref{sec:app_3} proves the results in Section~\ref{sec:design} and \ref{sec:term_design}. 
Appendix~\ref{sec:app_4} contains some auxiliary  lemmas used in the proofs. 

\subsection{Proofs - Section~\ref{sec:terminal}}
\label{sec:app_1}

 \subsubsection{Proof of Theorem~\ref{thm:grimm_terminal}}
\label{app:grimm_terminal}
\textbf{Part I: }The upper bound in~\eqref{eq:Lyap_terminal_1} follows directly from \eqref{eq:stab_term} and \eqref{eq:detect_1}.
The lower bound in~\eqref{eq:Lyap_terminal_1} follows from~\eqref{eq:detect_2} with $V_{N,\mathrm{f}}\geq \ell_{\min}$, $W\geq 0$ using $N\geq 1$. 
Abbreviate $\overline{x}_k=x_{u_{N,\mathrm{f},x}^*}(k,x)$, $k\in\mathbb{I}_{[0,N]}$.  
Using~\eqref{eq:detect_2} in a telescopic sum and $W(\overline{x}_N)\geq 0$, we have 
\begin{align}
\label{eq:sigma_sum_terminal}
V_{\mathrm{f}}(\overline{x}_N)+\epsilon_{\mathrm{o}}\sum_{k=0}^{N-1}\sigma(\overline{x}_k)\leq Y_{N,\mathrm{f}}(x)
\stackrel{\eqref{eq:Lyap_terminal_1}}{\leq}&(\gamma_{N,\mathrm{f}}+\gamma_{\mathrm{o}})\sigma(x).
\end{align}
 In the following, we use a case distinction to show 
\begin{align}
\label{eq:V_decrease_terminal}
V_{N,\mathrm{f}}(f(x,\mu_{\mathrm{f}}))+\ell(x,\mu_{\mathrm{f}}(x))-V_{N,\mathrm{f}}(x)
{\leq} (1-\alpha_{N,\mathrm{f}})\epsilon_{\mathrm{o}}\sigma(x). 
\end{align}
\textbf{Case 1: } Suppose
\begin{align}
\label{eq:grimm_terminal_case1}
V_{\mathrm{f}}(\overline{x}_N)\leq \dfrac{\overline{\gamma}_{\mathrm{f}}(\gamma_{N,\mathrm{f}}+\gamma_{\mathrm{o}})}{(N-1)\epsilon_{\mathrm{o}}(1+\epsilon_{\mathrm{f}})+\overline{\gamma}_{\mathrm{f}}}\sigma(x). 
\end{align}
We can upper bound the value function at the next state by \change{invoking} Assumption~\ref{ass:term}, i.e.,
\begin{align*}
&V_{N,\mathrm{f}}(f(x,\mu_{\mathrm{f}}))+\ell(x,\mu_{\mathrm{f}}(x))-V_{N,\mathrm{f}}(x)\nonumber\\
\leq& \min_{u\in\mathbb{U}}\ell(\overline{x}_N,u)+V_{\mathrm{f}}(f(\overline{x}_N,u))-V_{\mathrm{f}}(\overline{x}_N)\nonumber\\
\stackrel{\eqref{eq:term_2}}{\leq }&\epsilon_{\mathrm{f}}V_{\mathrm{f}}(\overline{x}_N)\stackrel{\eqref{eq:grimm_terminal_case1}}{\leq} \underbrace{\dfrac{\epsilon_{\mathrm{f}}\overline{\gamma}_{\mathrm{f}}(\gamma_{N,\mathrm{f}}+\gamma_{\mathrm{o}})}{(N-1)\epsilon_{\mathrm{o}} (1+\epsilon_{\mathrm{f}})+\overline{\gamma}_{\mathrm{f}}}}_{=(1-\alpha_{N,\mathrm{f}})\epsilon_{\mathrm{o}}}\sigma(x).
\end{align*}
\textbf{Case 2: } In case~\eqref{eq:grimm_terminal_case1} does not hold, Inequality~\eqref{eq:sigma_sum_terminal} ensures
\begin{align*}
&\epsilon_{\mathrm{o}}\sum_{k=0}^{N-1}\sigma(\overline{x}_k)\nonumber\\
\leq& 
(\gamma_{N,\mathrm{f}}+\gamma_{\mathrm{o}})\left(1-\dfrac{\overline{\gamma}_{\mathrm{f}}}{(N-1)\epsilon_{\mathrm{o}}(1+\epsilon_{\mathrm{f}})+\overline{\gamma}_{\mathrm{f}}}\right)\sigma(x)\nonumber\\
=&(\gamma_{N,\mathrm{f}}+\gamma_{\mathrm{o}})\dfrac{(N-1)\epsilon_{\mathrm{o}} (1+\epsilon_{\mathrm{f}})}{(N-1)\epsilon_{\mathrm{o}}(1+\epsilon_{\mathrm{f}})+\overline{\gamma}_{\mathrm{f}}}\sigma(x).
\end{align*}
Hence, there exists a $k_x\in\mathbb{I}_{[1,N-1]}$, such that
\begin{align}
\label{eq:sigma_k_x_terminal}
\sigma(\overline{x}_{k_x})\leq &\dfrac{\gamma_{N,\mathrm{f}}+\gamma_{\mathrm{o}}}{\epsilon_{\mathrm{o}}(N-1)}\dfrac{(N-1)\epsilon_{\mathrm{o}}(1+\epsilon_{\mathrm{f}})}{(N-1)\epsilon_{\mathrm{o}}(1+\epsilon_{\mathrm{f}})+\overline{\gamma}_{\mathrm{f}}}\sigma(x)\nonumber\\
=&\dfrac{(\gamma_{N,\mathrm{f}}+\gamma_{\mathrm{o}})(1+\epsilon_{\mathrm{f}})}{(N-1)\epsilon_{\mathrm{o}}(1+\epsilon_{\mathrm{f}})+\overline{\gamma}_{\mathrm{f}}}\sigma(x).
\end{align}
Correspondingly, we can upper bound the value function at the next state by \change{invoking} Assumption~\ref{ass:stab_term} at $k_x$, i.e.,
\begin{align*}
&V_{N,\mathrm{f}}(f(x,\mu_{\mathrm{f}}(x)))+\ell(x,\mu_{\mathrm{f}}(x))\nonumber\\
\leq& \mathcal{J}_{k_x}(x,u^*_{N,\mathrm{f},x})+V_{N-k_x+1,\mathrm{f}}(\overline{x}_{k_x})\nonumber\\
\stackrel{\eqref{eq:stab_term}}{\leq} &V_{N,\mathrm{f}}(x)-V_{\mathrm{f}}(\overline{x}_N)+
\overline{\gamma}_{\mathrm{f}}\sigma(\overline{x}_{k_x})\nonumber\\
\stackrel{\eqref{eq:grimm_terminal_case1},\eqref{eq:sigma_k_x_terminal}}{\leq} &V_{N,\mathrm{f}}(x)
- \dfrac{\overline{\gamma}_{\mathrm{f}}(\gamma_{N,\mathrm{f}}+\gamma_{\mathrm{o}})}{(N-1)\epsilon_{\mathrm{o}}(1+\epsilon_{\mathrm{f}})+\overline{\gamma}_{\mathrm{f}}}\sigma(x)\nonumber\\
&+\dfrac{\overline{\gamma}_{\mathrm{f}}(\gamma_{N,\mathrm{f}}+\gamma_{\mathrm{o}})(1+\epsilon_{\mathrm{f}})}{(N-1)\epsilon_{\mathrm{o}}(1+\epsilon_{\mathrm{f}})+\overline{\gamma}_{\mathrm{f}}}\sigma(x)\nonumber\\
=&V_{N,\mathrm{f}}(x)+\underbrace{\dfrac{\epsilon_{\mathrm{f}}\overline{\gamma}_{\mathrm{f}}(\gamma_{N,\mathrm{f}}+\gamma_{\mathrm{o}})}{(N-1)\epsilon_{\mathrm{o}} (1+\epsilon_{\mathrm{f}})+\overline{\gamma}_{\mathrm{f}}}}_{=(1-\alpha_{N,\mathrm{f}})\epsilon_{\mathrm{o}}}
\sigma(x).
\end{align*}
\textbf{Part II: }
Condition~\eqref{eq:Lyap_terminal_2} follows by combining Inequality~\eqref{eq:V_decrease_terminal} with \eqref{eq:detect_2}. 
Asymptotic stability of $\mathcal{A}$ follows by applying Inequalities~\eqref{eq:set} to \eqref{eq:Lyap_terminal} with the Lyapunov function $Y_{N,\mathrm{f}}$ and $\alpha_{N,\mathrm{f}}>0$ for $N>\underline{N}_{\mathrm{f}}$.
Regarding the performance bound~\eqref{eq:performance_terminal}, note that
\begin{align*}
&V_{N,\mathrm{f}}(x_{\mu_{\mathrm{f}}}(k+1))-V_{N,\mathrm{f}}(x_{\mu_{\mathrm{f}}}(k))\nonumber\\
&+(1-\alpha_{N,\mathrm{f}})(W(x_{\mu_{\mathrm{f}}}(k+1))-W(x_{\mu_{\mathrm{f}}}(k)))\nonumber\\
\stackrel{\eqref{eq:detect_2},\eqref{eq:Lyap_terminal_2}}{\leq}&-\alpha_{N,\mathrm{f}}\ell(x_{\mu_{\mathrm{f}}},\mu_{\mathrm{f}}(x_{\mu_{\mathrm{f}}}(k))).
\end{align*}
Using this inequality in a telescopic sum and $W\geq 0$, $V_{N,\mathrm{f}}\geq 0$, we get
\begin{align*}
\alpha_{N,\mathrm{f}}(\mathcal{J}^{\mu_{\mathrm{f}}}_\infty(x_0)+W(x_0))\leq Y_{N,\mathrm{f}}(x_0).
\end{align*}
Using the infinite-horizon optimal trajectory, i.e.,  $u_{\infty,x}^*$, as a feasible candidate solution yields  
$V_{N,\mathrm{f}}(x)\leq V_\infty(x)+V_{\mathrm{f}}(x_{u_{\infty,x}^*}(N,x))$.
Note that $\lim_{N\rightarrow\infty}V_{N,\mathrm{f}}+W=Y_\infty$ and hence $Y_\infty$ satisfies Inequalities~\eqref{eq:Lyap_terminal} with $\gamma_{N,\mathrm{f}}=\overline{\gamma}_{\mathrm{f}}$, $\alpha_{N,\mathrm{f}}=1$.
By combining Inequalities~\eqref{eq:Lyap_terminal} for $N=\infty$ in a telescopic sum, the infinite-horizon optimal trajectory satisfies 
\begin{align*}
Y_\infty(x_{u^*_{\infty,x}}(N,x))\leq \left(1-\dfrac{\epsilon_{\mathrm{o}}}{\overline{\gamma}_{\mathrm{f}}+\gamma_{\mathrm{o}}}\right)^NY_\infty(x).  
\end{align*}
Lastly, utilizing $V_{\mathrm{f}}(x)\stackrel{\eqref{eq:term_1}}{\leq} \overline{c}_{\mathrm{f}}\sigma(x)\leq \dfrac{\overline{c}_{\mathrm{f}}}{\epsilon_{\mathrm{o}}}Y_\infty(x)$, we arrive at the performance bound~\eqref{eq:performance_terminal}.  \qed


 \subsubsection{Proof of Theorem~\ref{thm:terminal_LP}}
\label{app:terminal_LP}
In case Inequality~\eqref{eq:V_decrease_terminal} holds for any possible solution, then also the remaining derivations in Theorem~\ref{thm:grimm_terminal} hold, with stability only for $\alpha_{N,\mathrm{f}}>0$. 
In Part I, we derive a general optimization problem to compute a constant $\alpha_{N,\mathrm{f}}$ satisfying Inequality~\eqref{eq:V_decrease_terminal}. 
In Part II, we show that this is equivalent to the LP~\eqref{eq:LP_terminal}. \\
\textbf{Part I: }
Denote $\tilde{\ell}_k=\ell(x_{u^*_{N,\mathrm{f},x}}(k,x),u^*_{N,\mathrm{f},x}(k))$, $k\in\mathbb{I}_{[0,N-1]}$, $\tilde{W}_k=W(x_{u^*_{N,\mathrm{f},x}}(k,x))$, $\tilde{\sigma}_k=\sigma(x_{u^*_{N,\mathrm{f},x}}(k,x))$, $k\in\mathbb{I}_{[0,N]}$, $\tilde{V}_{\mathrm{f}}=V_{\mathrm{f}}(x_{u^*_{N,\mathrm{f},x}}(N,x))$ and $\tilde{V}=V_{N,\mathrm{f}}(x_{u^*_{N,\mathrm{f},x}}(1,x))$. 
Given that $V_{N,\mathrm{f}}(x)=\sum_{k=0}^{N-1}\tilde{\ell}_k+\tilde{V}_{\mathrm{f}}$,  Inequality~\eqref{eq:V_decrease_terminal} is equivalent to
\begin{align}
\label{eq:LP_desired} 
\tilde{V}-\sum_{j=1}^{N-1}\tilde{\ell}_j-\tilde{V}_{\mathrm{f}}\leq (1-\alpha_{N,\mathrm{f}})\epsilon_{\mathrm{o}}\tilde{\sigma}_0.
\end{align}
In case $\tilde{\sigma}_0=0$, the derivation in Theorem~\ref{thm:grimm_terminal} directly ensures that~\eqref{eq:LP_desired} holds for any $\alpha_{N,\mathrm{f}}\in\mathbb{R}$. 
Thus, in the following we consider w.l.o.g. the case $\tilde{\sigma}_0>0$. 
All the considered quantities are non-negative, i.e., 
\begin{align}
\label{eq:LP_bounds_1}
&\tilde{\ell}_k\geq 0,~k\in\mathbb{I}_{[0,N-1]},~\tilde{V}_{\mathrm{f}}\geq 0,~
\tilde{\sigma}_k\geq 0,~\tilde{W}_k\geq 0,~k\in\mathbb{I}_{[0,N]}.
\end{align}
Furthermore, Assumption~\ref{ass:detect} implies
\begin{align}
\label{eq:LP_bounds_2}
\underline{\gamma}_{\mathrm{o}}\tilde{\sigma}_k\leq \tilde{W}_k\stackrel{\eqref{eq:detect_1}}{\leq}& \overline{\gamma}_{\mathrm{o}}\tilde{\sigma}_k,~k\in\mathbb{I}_{[0,N]},\\
\label{eq:LP_bounds_3}
\tilde{W}_{k+1}-\tilde{W}_k\stackrel{\eqref{eq:detect_2}}{\leq}& -\epsilon_{\mathrm{o}}\tilde{\sigma}_k+\tilde{\ell}_k,~k\in\mathbb{I}_{[0,N-1]}.
\end{align}
The principle of optimality implies $\sum_{j=k}^{N-1}\tilde{\ell}_j+\tilde{V}_{\mathrm{f}}=V_{N-k,\mathrm{f}}(x_{u_{N,\mathrm{f},x}}^*(k,x))$. 
Thus, Assumption~\ref{ass:stab_term} yields
\begin{align}
\label{eq:LP_bounds_4}
&\sum_{j=k}^{N-1}\tilde{\ell}_j+\tilde{V}_{\mathrm{f}}\stackrel{\eqref{eq:stab_term}}{\leq} \gamma_{N-k,\mathrm{f}}\tilde{\sigma}_k,~k\in\mathbb{I}_{[0,N-1]}.
\end{align}
A feasible input sequence at the next time step is given by $u_{N,\mathrm{f},x}^*(j)$, $j\in\mathbb{I}_{[1,k-1]}$ appended by the optimal solution starting at the state $x_{u_{N,\mathrm{f},x}}^*(k,x)$, i.e., 
\begin{align*}
\tilde{V}\leq \sum_{j=1}^{k-1}\tilde{\ell}_j+V_{N-k+1,\mathrm{f}}(x_{u_{N,\mathrm{f},x}}^*(k,x)),~k\in\mathbb{I}_{[1,N]}.
\end{align*}
Using Inequality~\eqref{eq:stab_term}, this implies 
\begin{align}
\label{eq:LP_bounds_5}
&\tilde{V}\leq \sum_{j=1}^{k-1}\tilde{\ell}_j+\gamma_{N-k+1,\mathrm{f}}\tilde{\sigma}_k,~k\in\mathbb{I}_{[1,N]}.
\end{align}
Appending the candidate from Assumption~\ref{ass:term} at $k=N$ yields
\begin{align}
\label{eq:LP_bounds_6}
&\tilde{V}\leq \sum_{j=1}^{N-1}\tilde{\ell}_j+(1+\epsilon_{\mathrm{f}})\tilde{V}_{\mathrm{f}}.
\end{align}
Furthermore, Assumption~\ref{ass:term} implies
\begin{align}
\label{eq:LP_bounds_7}
\underline{c}_{\mathrm{f}}\tilde{\sigma}_N\leq\tilde{V}_{\mathrm{f}}\leq \overline{c}_{\mathrm{f}}\tilde{\sigma}_N. 
\end{align}
Consider the following optimization problem:
\begin{subequations}
\label{eq:opt_general}
\begin{align}
\label{eq:opt_general_1}
\alpha_{N,\mathrm{f}}=&\inf_{\tilde{\ell},\tilde{W},\tilde{\sigma},\tilde{V},\alpha_{N,\mathrm{f}},\tilde{V}_{\mathrm{f}}}\alpha_{N,\mathrm{f}}\\
\label{eq:opt_general_2}
\text{s.t. }&~\tilde{\sigma}_0\epsilon_{\mathrm{o}}(1-\alpha_{N,\mathrm{f}})\geq \tilde{V}-\sum_{k=1}^{N-1}\tilde{\ell}_k-\tilde{V}_{\mathrm{f}},\\
\label{eq:opt_general_3}
&\tilde{\sigma}_0>0,~\eqref{eq:LP_bounds_1},\eqref{eq:LP_bounds_2},\eqref{eq:LP_bounds_3},\eqref{eq:LP_bounds_4},\eqref{eq:LP_bounds_5}, \eqref{eq:LP_bounds_6}, \eqref{eq:LP_bounds_7}.
\end{align}
\end{subequations}
Given that any optimal trajectory satisfies~\eqref{eq:LP_bounds_1}--\eqref{eq:LP_bounds_7}, the constant $\alpha_{N,\mathrm{f}}$ computed using~\eqref{eq:opt_general} is guaranteed to satisfy~\eqref{eq:V_decrease_terminal}, compare also arguments in~\cite[Thm.~4.2, Cor.~4.5]{grune2009analysis}. \\ 
\textbf{Part II: }Except for the constraint~\eqref{eq:opt_general_2}, the optimization problem~\eqref{eq:opt_general} only consists of linear terms. 
For a fixed $\alpha_{N,\mathrm{f}}$, the constraints in~\eqref{eq:opt_general} are linear in the decision variables. 
Thus, we can w.l.o.g. restrict $\tilde{\sigma}_0=1$ by rescaling the other variables (cf.~\cite[Lemma~4.6]{grune2009analysis}). 
Furthermore, given that the minimum $\alpha_{N,\mathrm{f}}$ is only constrained by~\eqref{eq:opt_general_2} (and guaranteed to be finite by Theorem~\ref{thm:grimm_terminal}), any minimizer satisfies this constraint with equality.  
Hence, using $\tilde{\sigma}_0=1$, we can directly plug in the constraint~\eqref{eq:opt_general_2} in the objective of~\eqref{eq:opt_general}, which results in the LP~\eqref{eq:LP_terminal}. \qed

\subsubsection{Proof of Theorem~\ref{thm:analytic_terminal_pdf}} 
\label{app:analytic_terminal_pdf}
 Given $\underline{\gamma}_{\mathrm{o}}=\overline{\gamma}_{\mathrm{o}}=0$, $\epsilon_{\mathrm{o}}=1$, we have\footnote{%
Inequality~\eqref{eq:LP_terminal_W_decrease} yields $\tilde{\ell}_k\geq \tilde{\sigma}_k$ and one can show that the minimizer satisfies the constraint with equality. For $k=N$, define $\tilde{\ell}_N:=\tilde{\sigma}_N$.}  $\tilde{W}=0$, $\tilde{\ell}=\tilde{\sigma}$.
Thus, the LP~\eqref{eq:LP_terminal} reduces to 
\begin{subequations}
\label{eq:LP_terminal_pdf}
\begin{align}
\alpha_{N,\mathrm{f}}
:=&\min_{\tilde{\ell},\tilde{V},\tilde{V}_{\mathrm{f}}}
\sum_{k=1}^{N-1}\tilde{\ell}_k+\tilde{V}_{\mathrm{f}}-\tilde{V}+1\nonumber\\
\label{eq:LP_terminal_pdf_nonneg}
\mathrm{s.t. ~}&\tilde{\ell}_0=1,~\tilde{\ell}_k\geq 0,~k\in\mathbb{I}_{[1,N]},~\tilde{V}_{\mathrm{f}}\geq 0,\\
\label{eq:LP_terminal_pdf_V_bound}
&\sum_{j=k}^{N-1}\tilde{\ell}_j+\tilde{V}_{\mathrm{f}}\leq \gamma_{N-k,\mathrm{f}}\tilde{\ell}_k,~k\in\mathbb{I}_{[0,N-1]},\\
\label{eq:LP_terminal_pdf_V_next_bound}
&\tilde{V}\leq \sum_{j=1}^{k-1}\tilde{\ell}_j+\gamma_{N-k+1,\mathrm{f}}\tilde{\ell}_k,~k\in\mathbb{I}_{[1,N]},\\
\label{eq:LP_terminal_pdf_Vf_next_bound}
&\tilde{V}\leq \sum_{j=1}^{N-1}\tilde{\ell}_j+(1+\epsilon_{\mathrm{f}})\tilde{V}_{\mathrm{f}},\\
\label{eq:LP_terminal_pdf_Vf_bound}
&\underline{c}_{\mathrm{f}}\tilde{\ell}_N\leq\tilde{V}_{\mathrm{f}}\leq \overline{c}_{\mathrm{f}}\tilde{\ell}_N. 
\end{align}
\end{subequations} 
Consider the following simpler LP: 
\begin{subequations}
\label{eq:LP_terminal_pdf_reduced}
\begin{align}
\hat{\alpha}_{N,\mathrm{f}}
:=&\min_{\tilde{\ell},\tilde{V},\tilde{V}_{\mathrm{f}}}
\sum_{k=1}^{N-1}\tilde{\ell}_k+\tilde{V}_{\mathrm{f}}-\tilde{V}+1\\
\mathrm{s.t. ~}
\label{eq:LP_terminal_pdf_reduced_V_bound}
&\sum_{j=1}^{N-1}\tilde{\ell}_j+\tilde{V}_{\mathrm{f}}\leq \gamma_{N,\mathrm{f}}-1,\\ 
\label{eq:LP_terminal_pdf_reduced_V_next_bound}
&\tilde{V}\leq \sum_{j=1}^{k-1}\tilde{\ell}_j+\gamma_{N-k+1,\mathrm{f}}\tilde{\ell}_k,~k\in\mathbb{I}_{[1,N-1]},\\
\label{eq:LP_terminal_pdf_reduced_Vf_next_bound}
&\tilde{V}\leq \sum_{j=1}^{N-1}\tilde{\ell}_j+(1+\epsilon_{\mathrm{f}})\tilde{V}_{\mathrm{f}}.
\end{align}
\end{subequations}
Denote a minimizer to~\eqref{eq:LP_terminal_pdf_reduced} by $\tilde{\ell}^*,\tilde{V}^*,\tilde{V}_{\mathrm{f}}^*$. 
Compared to~\eqref{eq:LP_terminal_pdf}, in this LP we dropped the constraints~\eqref{eq:LP_terminal_pdf_V_bound} for $k\in\mathbb{I}_{[1,N-1]}$, the constraint~\eqref{eq:LP_terminal_pdf_V_next_bound} for $k=N$,  the constraint~\eqref{eq:LP_terminal_pdf_Vf_bound}, and the non-negativity constraints~\eqref{eq:LP_terminal_pdf_nonneg}, and plugged in $\tilde{\ell}_0=1$. 
Hence, we have $\alpha_{N,\mathrm{f}}\geq \hat{\alpha}_{N,\mathrm{f}}$. \\
\textbf{Part I: }First, we show that there exists a minimizer that satisfies~\eqref{eq:LP_terminal_pdf_reduced_V_bound} with equality. 
Consider $\tilde{\ell}_1=\tilde{\ell}_1^*+c_1$, $\tilde{\ell}_k=\tilde{\ell}_k^*$, $k\in\mathbb{I}_{[1,N-1]}$, $\tilde{V}_{\mathrm{f}}=\tilde{V}_{\mathrm{f}}^*$, $\tilde{V}=\tilde{V}^*-c_1$, with $c_1\geq 0$ such that Inequality~\eqref{eq:LP_terminal_pdf_reduced_V_bound} holds with equality.  
This candidate satisfies Inequalities~\eqref{eq:LP_terminal_pdf_reduced_V_next_bound}--\eqref{eq:LP_terminal_pdf_reduced_Vf_next_bound}
since $N\geq 1$ implies $\gamma_{N,\mathrm{f}}\geq 1$. 
The corresponding cost remains unchanged, i.e., $\tilde{\ell}_1-\tilde{V}=\tilde{\ell}_1^*-\tilde{V}^*$. 
Hence, we have constructed a minimizer satisfying~\eqref{eq:LP_terminal_pdf_reduced_V_bound} with equality. \\
\textbf{Part II: }
Given that there exists a minimizer that satisfies~\eqref{eq:LP_terminal_pdf_reduced_V_bound}  with equality, we get the following equivalent LP:
\begin{subequations}
\label{eq:LP_terminal_pdf_reduced2}
\begin{align}
\hat{\alpha}_{N,\mathrm{f}}:=&\min_{\tilde{\ell},\tilde{V}}\gamma_{N,\mathrm{f}}-\tilde{V}\\
\mathrm{s.t. ~}
\label{eq:LP_terminal_pdf_reduced2_V_bound}
&\tilde{V}\leq \sum_{j=1}^{k-1}\tilde{\ell}_j+\gamma_{N-k+1,\mathrm{f}}\tilde{\ell}_k,~k\in\mathbb{I}_{[1,N-1]},\\
\label{eq:LP_terminal_pdf_reduced2_Vf_next_bound}
&\tilde{V}\leq -\epsilon_{\mathrm{f}}\sum_{j=1}^{N-1}\tilde{\ell}_j+(1+\epsilon_{\mathrm{f}})(\gamma_{N,\mathrm{f}}-1).
\end{align}
\end{subequations}
Suppose a minimizer satisfies Inequality~\eqref{eq:LP_terminal_pdf_reduced2_Vf_next_bound} strictly. 
Then, there exists a small enough constants $c_1>0$, such that $\tilde{\ell}_1=\tilde{\ell}_1^*+c_1$, $\tilde{\ell}_k=\tilde{\ell}_k^*$, $k\in\mathbb{I}_{[1,N-1]}$, $\tilde{V}_{\mathrm{f}}=\tilde{V}_{\mathrm{f}}^*$, $\tilde{V}=\tilde{V}^*-c_1$ satisfies Inequality~\eqref{eq:LP_terminal_pdf_reduced2_Vf_next_bound}. 
This candidate also satisfies Inequalities~\eqref{eq:LP_terminal_pdf_reduced2_V_bound} using  $\gamma_{N,\mathrm{f}}\geq 1$.
Furthermore, this candidate has a strictly smaller cost, resulting in a contradiction. 
Hence, Inequality~\eqref{eq:LP_terminal_pdf_reduced2_Vf_next_bound} holds with equality and the minimum satisfies
\begin{align}
\label{eq:LP_terminal_pdf_temp0}
\hat{\alpha}_{N,\mathrm{f}}=
\gamma_{N,\mathrm{f}}-\tilde{V}^*=1-\epsilon_{\mathrm{f}}(\gamma_{N,\mathrm{f}}-1)+\epsilon_{\mathrm{f}}\sum_{j=1}^{N-1}\tilde{\ell}^*_j.
\end{align}
\textbf{Part III: }
By $\overline{\ell}$ we denote the solution satisfying~\eqref{eq:LP_terminal_pdf_reduced2_V_bound} with equality for all $k\in\mathbb{I}_{[1,N-1]}$. 
Since any minimizer also satisfies Inequalities~\eqref{eq:LP_terminal_pdf_reduced2_V_bound}, we have
\begin{align*}
\sum_{j=1}^{k-1}(\overline{\ell}_j-\tilde{\ell}_j^*)+\gamma_{N-k+1,\mathrm{f}}(\overline{\ell}_k-\tilde{\ell}_k^*)\leq 
0,\quad k\in\mathbb{I}_{[1,N-1]}.
\end{align*}
Using Lemma~\ref{lemma:sum_nonnegative_app} from Appendix~\ref{sec:app_4} and $\gamma_{k,\mathrm{f}}\geq 1$ recursively, this implies 
$\sum_{j=1}^{N-1}\overline{\ell}_j-\tilde{\ell}_j^*\leq 0$. 
Considering the cost~\eqref{eq:LP_terminal_pdf_temp0}, this implies that the solution $\overline{\ell}$ satisfying the Inequalities~\eqref{eq:LP_terminal_pdf_reduced2_V_bound} with equality is a minimizer. \\
\textbf{Part IV: }
Given that Inequalities~\eqref{eq:LP_terminal_pdf_reduced2_V_bound} hold with equality for $k\in\mathbb{I}_{[1,N-1]}$, we have 
\begin{align*}
\overline{\ell}_{k+1}=\dfrac{\gamma_{N-k+1,\mathrm{f}}-1}{\gamma_{N-k,\mathrm{f}}}\overline{\ell}_k,\quad k\in\mathbb{I}_{[1,N-2]}.
\end{align*}
Using this recursively, we get 
\begin{align}
\label{eq:analytic_terminal_pdf_recursion}
\overline{\ell}_k=\prod_{j=1}^{k-1}\dfrac{\gamma_{N-j+1,\mathrm{f}}-1}{\gamma_{N-j,\mathrm{f}}}\overline{\ell}_1, k\in\mathbb{I}_{[1,N-1]}.
\end{align}
Given that Inequality~\eqref{eq:LP_terminal_pdf_reduced2_V_bound} for $k=N-1$ and  Inequality~\eqref{eq:LP_terminal_pdf_reduced2_Vf_next_bound} hold with equality, we have
\begin{align}
\label{eq:LP_terminal_pdf_temp1}
\sum_{k=1}^{N-1}\overline{\ell}_k+\dfrac{\gamma_{2,\mathrm{f}}-1}{1+\epsilon_{\mathrm{f}}}\overline{\ell}_{N-1}
=\gamma_{N,\mathrm{f}}-1.
\end{align}
Equation~\eqref{eq:analytic_terminal_pdf_recursion} ensures that the following equality holds
\begin{align}
\label{eq:LP_terminal_pdf_temp2}
&\prod_{j=1}^{N-2}\gamma_{N-j,\mathrm{f}}\sum_{k=1}^{N-1}\dfrac{\overline{\ell}_k}{\overline{\ell}_1}
=\sum_{k=1}^{N-1}\prod_{j=1}^{k-1}(\gamma_{N-j+1,\mathrm{f}}-1)\prod_{j=k+1}^{N-1}\gamma_{N-j+1,\mathrm{f}}\nonumber\\
=&\prod_{j=1}^{N-1}\gamma_{N-j+1,\mathrm{f}}-\prod_{j=1}^{N-1}(\gamma_{N-j+1,\mathrm{f}}-1),
\end{align}
where the last equality uses the induction proof in~\cite[Lemma~10.2]{grune2010analysis} with $\delta_{N-1-j}=-\gamma_{N-j+1,\mathrm{f}}$. 
Plugging this expression into Equation~\eqref{eq:LP_terminal_pdf_temp1} and using the expression for $\overline{\ell}_{N-1}$ from~\eqref{eq:analytic_terminal_pdf_recursion} yields
\begin{align}
\label{eq:LP_terminal_pdf_temp3}
&\overline{\ell}_1
\left(
\prod_{j=1}^{N-1}\gamma_{N-j+1,\mathrm{f}}-\dfrac{\epsilon_{\mathrm{f}}}{1+\epsilon_{\mathrm{f}}}\prod_{j=1}^{N-1}(\gamma_{N-j+1,\mathrm{f}}-1)
\right)\\
=&(\gamma_{N,\mathrm{f}}-1)\prod_{j=1}^{N-2}\gamma_{N-j,\mathrm{f}}. \nonumber
\end{align}
Thus, the minimum satisfies
\begin{align*}
&1-\dfrac{1-\hat{\alpha}_{N,\mathrm{f}}}{(\gamma_{N,\mathrm{f}}-1)\epsilon_{\mathrm{f}}}
\stackrel{\eqref{eq:LP_terminal_pdf_temp0}}{=}\dfrac{1}{\gamma_{N,\mathrm{f}}-1}\sum_{j=1}^{N-1}\overline{\ell}_j\\
\stackrel{\eqref{eq:analytic_terminal_pdf_recursion}}{=}&\dfrac{1}{\gamma_{N,\mathrm{f}}-1}\sum_{k=1}^{N-1}\prod_{j=1}^{k-1}\dfrac{\gamma_{N-j+1,\mathrm{f}}-1}{\gamma_{N-j,\mathrm{f}}}\overline{\ell}_1\\
\stackrel{\eqref{eq:LP_terminal_pdf_temp3}}{=}&\dfrac{\sum_{k=1}^{N-1}\prod_{j=1}^{k-1}(\gamma_{N-j+1,\mathrm{f}}-1)
\prod_{j=k+1}^{N-1}\gamma_{N-j+1,\mathrm{f}}}{
\prod_{j=1}^{N-1}\gamma_{N-j+1,\mathrm{f}}-\dfrac{\epsilon_{\mathrm{f}}}{1+\epsilon_{\mathrm{f}}}\prod_{j=1}^{N-1}(\gamma_{N-j+1,\mathrm{f}}-1)
}\\
\stackrel{\eqref{eq:LP_terminal_pdf_temp2}}{=}&(1+\epsilon_{\mathrm{f}})\dfrac{\prod_{j=1}^{N-1}\gamma_{N-j+1,\mathrm{f}}-\prod_{j=1}^{N-1}(\gamma_{N-j+1,\mathrm{f}}-1)}{(1+\epsilon_{\mathrm{f}})\prod_{j=1}^{N-1}\gamma_{N-j+1,\mathrm{f}}-\epsilon_{\mathrm{f}}\prod_{j=1}^{N-1}(\gamma_{N-j+1,\mathrm{f}}-1)}\\
=&1-\dfrac{\prod_{j=1}^{N-1}(\gamma_{N-j+1,\mathrm{f}}-1)}{
(1+\epsilon_{\mathrm{f}})\prod_{j=1}^{N-1}\gamma_{N-j+1,\mathrm{f}}-\epsilon_{\mathrm{f}}\prod_{j=1}^{N-1}(\gamma_{N-j+1,\mathrm{f}}-1)
},
\end{align*}
which is equivalent to~\eqref{eq:analytic_terminal_pdf}. \\
\textbf{Part V: }In the following, we show that the minimizer of~\eqref{eq:LP_terminal_pdf_reduced} is also a minimizer of~\eqref{eq:LP_terminal_pdf}, i.e., $\alpha_{N,\mathrm{f}}=\hat{\alpha}_{N,\mathrm{f}}$. 
The non-negativity constraint~\eqref{eq:LP_terminal_pdf_nonneg} for the stage cost $\tilde{\ell}$ holds using the linear relation of $\overline{\ell}_k$ and $\overline{\ell}_1$~\eqref{eq:analytic_terminal_pdf_recursion} and Equation~\eqref{eq:LP_terminal_pdf_temp3} with $\gamma_{N,\mathrm{f}}-1\geq 0$. 
Non-negativity of the terminal cost $\tilde{V}_{\mathrm{f}}$ follows by 
\begin{align*}
&\dfrac{\tilde{V}_{\mathrm{f}}^*}{\gamma_{N,\mathrm{f}}-1}=1-\dfrac{1}{\gamma_{N,\mathrm{f}}-1}\sum_{k=1}^{N-1}\overline{\ell}_j=\dfrac{1-\hat{\alpha}_{N,\mathrm{f}}}{(\gamma_{N,\mathrm{f}}-1)\epsilon_{\mathrm{f}}}\geq 0,
\end{align*}
where the first equality uses that~\eqref{eq:LP_terminal_pdf_reduced_V_bound} holds with equality. 

Consider $\tilde{\ell}_N^*=\tilde{V}^*_{\mathrm{f}}/\underline{c}_{\mathrm{f}}$, which satisfies Inequality~\eqref{eq:LP_terminal_pdf_Vf_bound} by definition. 
Given that Inequality~\eqref{eq:LP_terminal_pdf_Vf_next_bound} holds with equality, Inequality~\eqref{eq:LP_terminal_pdf_V_next_bound} for $k=N$ reduces to $\tilde{\ell}_N^*\geq \dfrac{1+\epsilon_{\mathrm{f}}}{\gamma_{1,\mathrm{f}}}\tilde{V}_{\mathrm{f}}^*$, 
which holds using~\eqref{eq:term_3}. 
Given that Inequality~\eqref{eq:LP_terminal_pdf_reduced_V_bound} holds with equality, satisfaction of Inequality~\eqref{eq:LP_terminal_pdf_V_bound} for some $k\in\mathbb{I}_{[1,N-1]}$ is equivalent to
\begin{align*}
\sum_{j=1}^{k-1}\overline{\ell}_j+\gamma_{N-k,\mathrm{f}}\overline{\ell}_k\geq \gamma_{N,\mathrm{f}}-1.
\end{align*}
Multiplying by $\prod_{j=1}^{k-1}\gamma_{N-j,\mathrm{f}}$ and dividing by $\overline{\ell}_1$, the left hand side yields
\begin{align*}
&\gamma_{N-k+1,\mathrm{f}}\sum_{j=1}^{k-1}\prod_{l=1}^{j-1}(\gamma_{N-l+1,\mathrm{f}}-1)\prod_{l=j+1}^{k-1}\gamma_{N-l+1,\mathrm{f}}\\
&+\gamma_{N-k,\mathrm{f}}\prod_{j=1}^{k-1}(\gamma_{N-j+1,\mathrm{f}}-1)\\
=&\prod_{j=1}^{k}\gamma_{N-j+1,\mathrm{f}}
-(\gamma_{N-k+1,\mathrm{f}}-\gamma_{N-k,\mathrm{f}})\prod_{j=1}^{k-1}(\gamma_{N-j+1,\mathrm{f}}-1),
\end{align*}
where the equality follows similar to~\eqref{eq:LP_terminal_pdf_temp2}. 
The corresponding right hand side is given by
\begin{align*}
&\dfrac{\gamma_{N,\mathrm{f}}-1}{\overline{\ell}_1}\prod_{j=1}^{k-1}\gamma_{N-j,\mathrm{f}}\\
\stackrel{\eqref{eq:LP_terminal_pdf_temp3}}{=}&\left[\prod_{j=1}^{N-1}\gamma_{N-j+1,\mathrm{f}}-\dfrac{\epsilon_{\mathrm{f}}}{1+\epsilon_{\mathrm{f}}}\prod_{j=1}^{N-1}(\gamma_{N-j+1,\mathrm{f}}-1)\right]\dfrac{1}{\prod_{j=k}^{N-2}\gamma_{N-j,\mathrm{f}}}.
\end{align*}
By multiplying with $(1+\epsilon_{\mathrm{f}})\prod_{j=k}^{N-2}\gamma_{N-j,\mathrm{f}}$, the condition can be simplified to 
\begin{align*}
&\epsilon_{\mathrm{f}}\prod_{j=k}^{N-1}(\gamma_{N-j+1,\mathrm{f}}-1)\\
\geq&(1+\epsilon_{\mathrm{f}})(\gamma_{N-k+1,\mathrm{f}}-\gamma_{N-k,\mathrm{f}})\prod_{j=k+1}^{N-1}\gamma_{N-j+1,\mathrm{f}}.
\end{align*} 
Using the formula~\eqref{eq:submult_pdf_term_12}, this can be written equivalently as 
\begin{align}
\label{eq:iff_condition_pdf_term}
&\epsilon_{\mathrm{f}}\left[\sum_{l=0}^{N-k}c_l+c_{N-k+1,\mathrm{f}}\right]\prod_{j=2}^{N-k}(\gamma_{j,\mathrm{f}}-1)\nonumber\\
\geq&(1+\epsilon_{\mathrm{f}})(\gamma_{N-k+1,\mathrm{f}}-\gamma_{N-k,\mathrm{f}})\prod_{j=2}^{N-k}\gamma_{j,\mathrm{f}}.
\end{align} 
Inequality~\eqref{eq:iff_condition_pdf_term} holds using $c_0\geq 0$ and Inequality~\eqref{eq:lemma_submult_induction_pdf} from Lemma~\ref{lemma:submult_induction_pdf} in Appendix~\ref{sec:app_4} for $l=0$. \\
\textbf{Part VI: }
Using $\gamma_{k,\mathrm{f}}=\overline{\gamma}_{\mathrm{f}}$,  $k\in\mathbb{I}_{[1,N]}$, we have $\alpha_{N,\mathrm{f}}=\hat{\alpha}_{N,\mathrm{f}}>0$ with~\eqref{eq:analytic_terminal_pdf} if
\begin{align*}
\epsilon_{\mathrm{f}}\overline{\gamma}_{\mathrm{f}}(\overline{\gamma}_{\mathrm{f}}-1)^{N-1}<(1+\epsilon_{\mathrm{f}})\overline{\gamma}_{\mathrm{f}}^{N-1}.
\end{align*}
Applying the logarithm yields $\underline{N}_{\mathrm{f}}$ in Equation~\eqref{eq:hat_alpha_explicit_N_terminal_pdf}. \qed

\subsubsection{Proof of Theorem~\ref{thm:analytic_terminal_sigma_W}} 
\label{app:analytic_terminal_sigma_W}
The constraints~\eqref{eq:LP_terminal_W_bound} are trivially satisfied with equality using  $\underline{\gamma}_{\mathrm{o}}=\overline{\gamma}_{\mathrm{o}}=1$, i.e., $\tilde{W}_k=\tilde{\sigma}_k$, $k\in\mathbb{I}_{[0,N]}$.
Hence, the LP~\eqref{eq:LP_terminal} reduces to the following LP: 
\begin{subequations}
\label{eq:LP_terminal_detect}
\begin{align}
\label{eq:LP_terminal_detect_cost}
&(\alpha_{N,\mathrm{f}}-1)\epsilon_{\mathrm{o}}\\
:=&\min_{\tilde{\ell},\tilde{\sigma},\tilde{V},\tilde{V}_{\mathrm{f}}}
\sum_{k=1}^{N-1}\tilde{\ell}_k+\tilde{V}_{\mathrm{f}}-\tilde{V}\nonumber\\
\label{eq:LP_terminal_detect_nonneg_1}
\mathrm{s.t. ~}&\tilde{\sigma}_0=1,~\tilde{\ell}_k\geq 0,~k\in\mathbb{I}_{[0,N-1]},~\tilde{V}_{\mathrm{f}}\geq 0,\\
\label{eq:LP_terminal_detect_nonneg_2}
&\tilde{\sigma}_k\geq 0,~k\in\mathbb{I}_{[0,N]},\\
\label{eq:LP_terminal_detect_W_decrease}
&\tilde{\sigma}_{k+1}\leq \eta\tilde{\sigma}_k+\tilde{\ell}_k,~k\in\mathbb{I}_{[0,N-1]},\\
\label{eq:LP_terminal_detect_V_bound}
&\sum_{j=k}^{N-1}\tilde{\ell}_j+\tilde{V}_{\mathrm{f}}\leq \gamma_{N-k,\mathrm{f}}\tilde{\sigma}_k,~k\in\mathbb{I}_{[0,N-1]},\\
\label{eq:LP_terminal_detect_V_next_bound}
&\tilde{V}\leq \sum_{j=1}^{k-1}\tilde{\ell}_j+\gamma_{N-k+1,\mathrm{f}}\tilde{\sigma}_k,~k\in\mathbb{I}_{[1,N]},\\
\label{eq:LP_terminal_detect_Vf_next_bound}
&\tilde{V}\leq \sum_{j=1}^{N-1}\tilde{\ell}_j+(1+\epsilon_{\mathrm{f}})\tilde{V}_{\mathrm{f}},\\
\label{eq:LP_terminal_detect_Vf_bound}
&\underline{c}_{\mathrm{f}}\tilde{\sigma}_N\leq\tilde{V}_{\mathrm{f}}\leq \overline{c}_{\mathrm{f}}\tilde{\sigma}_N. 
\end{align}
\end{subequations}
Consider the following simpler LP: 
\begin{subequations}
\label{eq:LP_terminal_detect_reduced}
\begin{align}
\label{eq:LP_terminal_detect_reduced_cost}
&(\hat{\alpha}_{N,\mathrm{f}}-1)\epsilon_{\mathrm{o}}\\
:=&\min_{\tilde{\ell},\tilde{\sigma},\tilde{V},\tilde{V}_{\mathrm{f}}}
\sum_{k=1}^{N-1}\tilde{\ell}_k+\tilde{V}_{\mathrm{f}}-\tilde{V}\nonumber\\
\mathrm{s.t. ~}&\tilde{\sigma}_0=1,\\
\label{eq:LP_terminal_detect_reduced_W_decrease}
&\tilde{\sigma}_{k+1}\leq \eta\tilde{\sigma}_k+\tilde{\ell}_k,~k\in\mathbb{I}_{[0,N-1]},\\
\label{eq:LP_terminal_detect_reduced_V_bound}
&\sum_{j=0}^{N-1}\tilde{\ell}_j+\tilde{V}_{\mathrm{f}}\leq \gamma_{N,\mathrm{f}},\\
\label{eq:LP_terminal_detect_reduced_V_next_bound}
&\tilde{V}\leq \sum_{j=1}^{k-1}\tilde{\ell}_j+\gamma_{N-k+1,\mathrm{f}}\tilde{\sigma}_k,~k\in\mathbb{I}_{[1,N]},\\
\label{eq:LP_terminal_detect_reduced_Vf_next_bound}
&\tilde{V}\leq \sum_{j=1}^{N-1}\tilde{\ell}_j+(1+\epsilon_{\mathrm{f}})\tilde{V}_{\mathrm{f}}.
\end{align}
\end{subequations}
Denote a minimizer to~\eqref{eq:LP_terminal_detect_reduced} by $\tilde{\ell}^*,\tilde{\sigma}^*,\tilde{V}^*,\tilde{V}_{\mathrm{f}}^*$. 
Compared to~\eqref{eq:LP_terminal_detect}, we dropped the constraints~\eqref{eq:LP_terminal_detect_V_bound} for $k\in\mathbb{I}_{[1,N-1]}$,  the constraint~\eqref{eq:LP_terminal_detect_Vf_bound}, and the non-negativity constraints~\eqref{eq:LP_terminal_detect_nonneg_1}--\eqref{eq:LP_terminal_detect_nonneg_2}. 
Hence, any minimizer to the LP~\eqref{eq:LP_terminal_detect} is a feasible candidate to the LP~\eqref{eq:LP_terminal_detect_reduced} and thus we have $\alpha_{N,\mathrm{f}}\geq \hat{\alpha}_{N,\mathrm{f}}$. 
In the following, we first derive the analytical solution to~\eqref{eq:LP_terminal_detect_reduced} and then show $\hat{\alpha}_{N,\mathrm{f}}=\alpha_{N,\mathrm{f}}$. \\
\textbf{Part I: }
First, we show that any minimizer satisfies~\eqref{eq:LP_terminal_detect_reduced_W_decrease} with equality. 
For contradiction, suppose there exists a $k' \in\mathbb{I}_{[0,N-1]}$, such that 
$\tilde{\sigma}^*_{k'+1}+c_1\leq\eta\tilde{\sigma}_{k'}^*+\tilde{\ell}_{k'}^*$ with some $c_1>0$.
Then, we can choose $\tilde{\sigma}_k=\tilde{\sigma}_k^*+c_1\eta^{k-1-k'}$, $k\in\mathbb{I}_{[k'+1,N-1]}$, $\tilde{\sigma}_k=\tilde{\sigma}^*_k$, $k\in\mathbb{I}_{[0,k']}$, $\tilde{\ell}_k=\tilde{\ell}^*_k$, $k\in\mathbb{I}_{[0,N-2]}$, which satisfies~\eqref{eq:LP_terminal_detect_reduced_W_decrease} for $k\in\mathbb{I}_{[0,N-2]}$. 
The constraint~\eqref{eq:LP_terminal_detect_reduced_W_decrease} for $k=N-1$ remains valid with $\tilde{\sigma}_N=\tilde{\sigma}_N^*+c_2\dfrac{1}{\gamma_{1,\mathrm{f}}+1}$, 
$\tilde{\ell}_{N-1}=\tilde{\ell}_{N-1}^*-c_2\dfrac{\gamma_{1,\mathrm{f}}}{\gamma_{1,\mathrm{f}}+1}$, 
 with $c_2=c_1\eta^{N-1-k'}>0$.
The constraints~\eqref{eq:LP_terminal_detect_reduced_V_next_bound}, $k\in\mathbb{I}_{[1,N-1]}$ remains valid with $\tilde{V}=\tilde{V}^*$ since $\tilde{\sigma}_k\geq \tilde{\sigma}_k^*$, $k\in\mathbb{I}_{[1,N-1]}$ and $\tilde{\ell}_k=\tilde{\ell}_k^*$, $k\in\mathbb{I}_{[0,N-2]}$. 
For $k=N$, the constraint~\eqref{eq:LP_terminal_detect_reduced_V_next_bound} remains valid because the increase in $\gamma_{1,\mathrm{f}}\tilde{\sigma}_{N}$ and the decrease in $\tilde{\ell}_{N-1}$ exactly cancels. 
The constraint~\eqref{eq:LP_terminal_detect_reduced_Vf_next_bound} holds by choosing $\tilde{V}_{\mathrm{f}}=\tilde{V}_{\mathrm{f}}^*-\dfrac{\tilde{\ell}_{N-1}-\tilde{\ell}_{N-1}^*}{1+\epsilon_{\mathrm{f}}}$, which satisfies
$\tilde{V}_{\mathrm{f}}+\tilde{\ell}_{N-1}-\tilde{V}_{\mathrm{f}}^*-\tilde{\ell}_{N-1}^*=\epsilon_{\mathrm{f}}\dfrac{\tilde{\ell}_{N-1}-\tilde{\ell}_{N-1}^*}{1+\epsilon_{\mathrm{f}}}<0$. 
Thus, the constraint~\eqref{eq:LP_terminal_detect_reduced_V_bound} holds and we have a strictly smaller cost.
Hence, Inequalities~\eqref{eq:LP_terminal_detect_reduced_W_decrease} hold with equality and using $\tilde{\sigma}_0=1$,  we can recursively compute: 
\begin{align}
\label{eq:LP_sigma_W_expl_sigma}
\tilde{\sigma}_k=\eta^k+\sum_{j=0}^{k-1}\eta^{k-1-j}\tilde{\ell}_j.
\end{align}
\textbf{Part II: }
 Next, we show that any minimizer satisfies~\eqref{eq:LP_terminal_detect_reduced_Vf_next_bound} with equality. 
For contradiction, suppose $\tilde{V}^*+c_1\leq\sum_{j=1}^{N-1}\tilde{\ell}_j^*+(1+\epsilon_{\mathrm{f}})\tilde{V}_{\mathrm{f}}^*$ with $c_1>0$. 
Consider $\tilde{\ell}=\tilde{\ell}^*$, $\tilde{V}=\tilde{V}^*$, $\tilde{\sigma}=\tilde{\sigma}^*$, $\tilde{V}_{\mathrm{f}}= \tilde{V}^*_{\mathrm{f}}-\frac{c_1}{1+\epsilon_{\mathrm{f}}}$. 
Inequalities~\eqref{eq:LP_terminal_detect_reduced_W_decrease} and \eqref{eq:LP_terminal_detect_reduced_V_next_bound} are trivially satisfied. 
Inequality~\eqref{eq:LP_terminal_detect_reduced_V_bound} holds since $\tilde{V}_{\mathrm{f}}\leq\tilde{V}^*_{\mathrm{f}}$. 
Inequality~\eqref{eq:LP_terminal_detect_reduced_Vf_next_bound} remains valid since the decrease in $\tilde{V}_{\mathrm{f}}$ is chosen sufficiently small. 
Thus, we have a feasible solution with strictly smaller cost, contradicting optimality. \\
\textbf{Part III: }
Next, we show that any minimizer satisfies~\eqref{eq:LP_terminal_detect_reduced_V_bound} with equality. 
For contradiction, suppose that
$c_0+\sum_{j=0}^{N-1}\tilde{\ell}^*_j+\tilde{V}_{\mathrm{f}}^*\leq \gamma_{N,\mathrm{f}}$, 
with $c_0>0$. 
Define $c_1=\dfrac{1+\epsilon_{\mathrm{f}}}{1+\epsilon_{\mathrm{f}}+\min_{k\in\mathbb{I}_{[1,N]}}\eta^{k-1}\gamma_{N-k+1,\mathrm{f}}}c_0>0$, $c_2=c_0-c_1>0$. 
Consider $\tilde{\ell}_0=\tilde{\ell}_0^*+c_1$, $\tilde{\ell}_k=\tilde{\ell}_k^*$, $k\in\mathbb{I}_{[1,N-1]}$, $\tilde{V}_{\mathrm{f}}=\tilde{V}_{\mathrm{f}}^*+c_2$, which also satisfies~\eqref{eq:LP_terminal_detect_reduced_V_bound}. 
The constraints~\eqref{eq:LP_terminal_detect_reduced_W_decrease} remain valid if we choose $\tilde{\sigma}_k=\tilde{\sigma}_k^*+c_1\eta^{k-1}$, $k\in\mathbb{I}_{[1,N]}$, which implies $\tilde{\sigma}_k>\tilde{\sigma}_k^*$, $k\in\mathbb{I}_{[1,N]}$. 
The constraint~\eqref{eq:LP_terminal_detect_reduced_Vf_next_bound} remains valid if we choose $\tilde{V}=\tilde{V}^*+(1+\epsilon_{\mathrm{f}})c_2$. 
The constraints~\eqref{eq:LP_terminal_detect_reduced_V_next_bound} remain valid since $\tilde{V}-\tilde{V}^*=(1+\epsilon_{\mathrm{f}})c_2\leq c_1\eta^{k-1}\gamma_{N-k+1,\mathrm{f}}$, $k\in\mathbb{I}_{[1,N]}$. 
The corresponding cost is strictly smaller with $\tilde{V}_{\mathrm{f}}-\tilde{V}-(\tilde{V}_{\mathrm{f}}^*-\tilde{V}^*)=-\epsilon_{\mathrm{f}}c_2<0$, thus contradicting optimality. \\
\textbf{Part IV: }
Given that Inequality~\eqref{eq:LP_terminal_detect_reduced_V_bound} holds with equality, we can substitute $\tilde{V}_{\mathrm{f}}=\gamma_{N,\mathrm{f}}-\sum_{j=0}^{N-1}\tilde{\ell}_j$.
Thus, the objective is equivalent to 
\begin{align*}
\sum_{k=1}^{N-1}\tilde{\ell}_k+\tilde{V}_{\mathrm{f}}-\tilde{V}=\gamma_{N,\mathrm{f}}-(\tilde{\ell}_0+\tilde{V})=\gamma_{N,\mathrm{f}}-\tilde{U},
\end{align*}
with a modified decision variable $\tilde{U}:=\tilde{V}+\tilde{\ell}_0$. 
Using additionally the explicit expression for $\tilde{\sigma}_k$ from Equation~\eqref{eq:LP_sigma_W_expl_sigma}, the LP~\eqref{eq:LP_terminal_detect_reduced} is equivalent to
\begin{subequations}
\label{eq:LP_terminal_detect_v2_reduced}
\begin{align}
\label{eq:LP_terminal_detect_v2_reduced_cost}
&\min_{\tilde{\ell},\tilde{U}}\gamma_{N,\mathrm{f}}-\tilde{U}\\
\mathrm{s.t. ~}
\label{eq:LP_terminal_detect_v2_reduced_V_next_bound}
&\tilde{U}\leq \sum_{j=0}^{k-1}(1+\gamma_{N-k+1,\mathrm{f}}\eta^{k-1-j})\tilde{\ell}_j+\gamma_{N-k+1,\mathrm{f}}\eta^k,\\
&k\in\mathbb{I}_{[1,N]},\nonumber\\
\label{eq:LP_terminal_detect_v2_reduced_Vf_next_bound}
&\tilde{U}\leq (1+\epsilon_{\mathrm{f}})\gamma_{N,\mathrm{f}}-\epsilon_{\mathrm{f}}\sum_{j=0}^{N-1}\tilde{\ell}_j.
\end{align}
\end{subequations}
Next, we show that for $k=1$ the constraint~\eqref{eq:LP_terminal_detect_v2_reduced_V_next_bound} is satisfied with equality. 
For contradiction, suppose
$\tilde{U}^*< (1+\gamma_{N,\mathrm{f}})\tilde{\ell}_0^*+\gamma_{N,\mathrm{f}}\eta.$ 
Define $c_{0,1}:=\min_{k\in\mathbb{I}_{[1,N-1]}}\dfrac{1+\gamma_{N-k+1,\mathrm{f}}\eta^{k-1}}{1+\gamma_{N-k+1,\mathrm{f}}\eta^{k-2}}>1$. 
Consider $\tilde{\ell}_0=\tilde{\ell}_0^*+c_1$, $\tilde{\ell}_1=\tilde{\ell}_1^*-c_1\sqrt{c_{0,1}}$, $\tilde{\ell}_k=\tilde{\ell}_k^*$, $k\in\mathbb{I}_{[2,N-1]}$, $\tilde{U}=\tilde{U}^*$ with some $c_1>0$. 
For $k=1$, the constraint~\eqref{eq:LP_terminal_detect_v2_reduced_V_next_bound} is strictly satisfies if $c_1>0$ is chosen sufficiently small.
For $k\in\mathbb{I}_{[2,N]}$, the constraints~\eqref{eq:LP_terminal_detect_v2_reduced_V_next_bound} are strictly satisfied by
\begin{align*}
&(1+\gamma_{N-k+1,\mathrm{f}}\eta^{k-1})c_1-(1+\gamma_{N-k+1,\mathrm{f}}\eta^{k-2})c_1\sqrt{c_{0,1}}\\
\geq& (1+\gamma_{N-k+1,\mathrm{f}}\eta^{k-2})(c_{0,1}-\sqrt{c_{0,1}})c_1>0.
\end{align*}
The constraint~\eqref{eq:LP_terminal_detect_v2_reduced_Vf_next_bound} is strictly satisfied by
$\epsilon_{\mathrm{f}}c_1(\sqrt{c}_{0,1}-1)>0$. 
Hence, we can increase $\tilde{U}$, resulting in a feasible candidate solution with a smaller cost, thus contradicting optimality. \\
\textbf{Part V: }
Given that for $k=1$ Inequality~\eqref{eq:LP_terminal_detect_v2_reduced_V_next_bound} holds with equality, we can substitute $\tilde{U}=(1+\gamma_{N,\mathrm{f}})\tilde{\ell}_0+\gamma_{N,\mathrm{f}}\eta$. 
Furthermore, from the Part II of the proof, we know that Inequality~\eqref{eq:LP_terminal_detect_v2_reduced_Vf_next_bound} holds with equality. 
Thus, we can pose the following equivalent LP:
\begin{subequations}
\label{eq:LP_terminal_detect_v3_reduced}
\begin{align}
\label{eq:LP_terminal_detect_v3_reduced_cost}
&\min_{\tilde{\ell}}\gamma_{N,\mathrm{f}}(1-\eta)-(1+\gamma_{N,\mathrm{f}})\tilde{\ell}_0\\
\mathrm{s.t. ~}
\label{eq:LP_terminal_detect_v3_reduced_V_next_bound}
& (\gamma_{N,\mathrm{f}}-\gamma_{N-k+1,\mathrm{f}}\eta^{k-1})(\tilde{\ell}_0+\eta)\\
&\leq \sum_{j=1}^{k-1}(1+\gamma_{N-k+1,\mathrm{f}}\eta^{k-1-j})\tilde{\ell}_j,~k\in\mathbb{I}_{[2,N]},\nonumber\\
\label{eq:LP_terminal_detect_v3_reduced_Vf_next_bound}
&\epsilon_{\mathrm{f}}\sum_{j=1}^{N-1}\tilde{\ell}_j+(1+\epsilon_{\mathrm{f}}+\gamma_{N,\mathrm{f}})\tilde{\ell}_0= (1-\eta+\epsilon_{\mathrm{f}})\gamma_{N,\mathrm{f}}.
\end{align}
\end{subequations}
In the following, we show that there exists a minimizer satisfying~\eqref{eq:LP_terminal_detect_v3_reduced_V_next_bound} with equality for $k\in\mathbb{I}_{[2,N]}$.
Denote the solution satisfying~\eqref{eq:LP_terminal_detect_v3_reduced_V_next_bound} with equality by $\overline{\ell}_k$. 
Given that the minimizers also satisfies the Inequalities~\eqref{eq:LP_terminal_detect_v3_reduced_V_next_bound}, we have
\begin{align*}
&\sum_{j=1}^{k-1}(1+\gamma_{N-k+1,\mathrm{f}}\eta^{k-1-j})\overline{\ell}_k\nonumber\\
=&(\gamma_{N,\mathrm{f}}-\gamma_{N-k+1,\mathrm{f}}\eta^{k-1})(\tilde{\ell}^*_0+\eta)\nonumber\\
\leq&\sum_{j=1}^{k-1}(1+\gamma_{N-k+1,\mathrm{f}}\eta^{k-1-j})\tilde{\ell}_j^*,~k\in\mathbb{I}_{[2,N]}.
\end{align*}
This is equivalent to
\begin{align*}
\sum_{j=1}^{k-1}\underbrace{(1+\gamma_{N-k+1,\mathrm{f}}\eta^{k-1-j})}_{=:\epsilon_{j,k}}(\overline{\ell}_j-\tilde{\ell}^*_j)\leq 0,~k\in\mathbb{I}_{[2,N]},
\end{align*}
with $\epsilon_{j,k}$ positive and non-decreasing (in $j$). 
Applying Lemma~\ref{lemma:sum_nonnegative_app} from Appendix~\ref{sec:app_3} recursively, this implies
$\sum_{j=1}^{k-1}\overline{\ell}_j-\tilde{\ell}_j^*\leq 0$, $k\in\mathbb{I}_{[2,N]}$. 
Using this inequality for $k=N$, Equality~\eqref{eq:LP_terminal_detect_v3_reduced_Vf_next_bound} ensures $\overline{\ell}_0\geq \tilde{\ell}_0^*$, 
i.e., $\overline{\ell}$ is a minimizer and in the following we consider $\tilde{\ell}^*=\overline{\ell}$. \\
\textbf{Part VI: }
Given that Inequalities~\eqref{eq:LP_terminal_detect_v3_reduced_V_next_bound} hold with equality,  we can use Lemma~\ref{lemma:analytic_ab_app} from Appendix~\ref{sec:app_3} with $\delta_l=\gamma_{l,\mathrm{f}}$, which implies $\overline{\ell}_k=a_k(\tilde{\ell}_0+\eta)$, $k\in\mathbb{I}_{[1,N-1]}$ with $a_k$ according to Equation~\eqref{eq:ab_explicit_app}. 
Thus, Equation~\eqref{eq:LP_terminal_detect_v3_reduced_Vf_next_bound} yields
\begin{align*}
&(1-\eta+\epsilon_{\mathrm{f}})\gamma_{N,\mathrm{f}}\\
=&\epsilon_{\mathrm{f}}\sum_{j=1}^{N-1}\tilde{\ell}^*_j+(1+\epsilon_{\mathrm{f}}+\gamma_{N,\mathrm{f}})\tilde{\ell}^*_0\\
=&\tilde{\ell}^*_0\left(1+\epsilon_{\mathrm{f}}+\gamma_{N,\mathrm{f}}+\epsilon_{\mathrm{f}}\sum_{j=1}^{N-1}a_j\right)+\epsilon_{\mathrm{f}}\eta\sum_{j=1}^{N-1}a_j.
\end{align*}
Lemma~\ref{lemma:induction_formula_app} from Appendix~\ref{sec:app_3} with $\delta_l=\gamma_{l,\mathrm{f}}$ yields
\begin{align}
\label{eq:analytic_a_N}
&\sum_{k=1}^{N-1}a_k\stackrel{\eqref{eq:ab_explicit_app}}{=}\sum_{k=1}^{N-1}\dfrac{\gamma_{N-k+1,\mathrm{f}}-\eta\gamma_{N-k,\mathrm{f}}}{\gamma_{N-k,\mathrm{f}}+1}\prod_{j=0}^{k-2}\dfrac{\eta+\gamma_{N-j,\mathrm{f}}}{\gamma_{N-j-1,\mathrm{f}}+1}\nonumber\\
\stackrel{\eqref{eq:induction_formula_app_a}}{=}&\gamma_{N,\mathrm{f}}-\gamma_{1,\mathrm{f}}\prod_{j=0}^{N-2}\dfrac{\eta+\gamma_{N-j,\mathrm{f}}}{1+\gamma_{N-j-1,\mathrm{f}}}=:\gamma_{N,\mathrm{f}}-\overline{a}_N.
\end{align}
Thus, the minimizer of~\eqref{eq:LP_terminal_detect_v3_reduced} satisfies
\begin{align}
\label{eq:LP_terminal_detect_minimizer_0}
&\tilde{\ell}_0^*= \dfrac{(1-\eta)(1+\epsilon_{\mathrm{f}})\gamma_{N,\mathrm{f}}+\epsilon_{\mathrm{f}}\eta\overline{a}_N}{(1+\epsilon_{\mathrm{f}})(1+\gamma_{N,\mathrm{f}})-\epsilon_{\mathrm{f}}\overline{a}_N}.
\end{align}
Hence, the minimum to~\eqref{eq:LP_terminal_detect_v3_reduced} (and equally to~\eqref{eq:LP_terminal_detect_reduced}) 
satisfies
\begin{align*}
&(\hat{\alpha}_{N,\mathrm{f}}-1)\epsilon_{\mathrm{o}}\\
\stackrel{\eqref{eq:LP_terminal_detect_v3_reduced_cost}}{=}&\gamma_{N,\mathrm{f}}(1-\eta)-(1+\gamma_{N,\mathrm{f}})\tilde{\ell}_0^*\\
\stackrel{\eqref{eq:LP_terminal_detect_minimizer_0}}{=}&\gamma_{N,\mathrm{f}}(1-\eta)-(1+\gamma_{N,\mathrm{f}}) \dfrac{(1-\eta)(1+\epsilon_{\mathrm{f}})\gamma_{N,\mathrm{f}}+\epsilon_{\mathrm{f}}\eta\overline{a}_N}{(1+\epsilon_{\mathrm{f}})(1+\gamma_{N,\mathrm{f}})-\epsilon_{\mathrm{f}}\overline{a}_N}\\
=&-\overline{a}_N\epsilon_{\mathrm{f}}\dfrac{\gamma_{N,\mathrm{f}}+\eta}{(1+\epsilon_{\mathrm{f}})(1+\gamma_{N,\mathrm{f}})-\epsilon_{\mathrm{f}}\overline{a}_N}\\
\stackrel{\eqref{eq:analytic_a_N}}{=}&-
\dfrac{\epsilon_{\mathrm{f}}\gamma_{1,\mathrm{f}}(\gamma_{N,\mathrm{f}}+\eta)
\prod_{j=0}^{N-2}(\eta+\gamma_{N-j,\mathrm{f}})}{
(1+\epsilon_{\mathrm{f}})\prod_{j=0}^{N-1}(1+\gamma_{N-j,\mathrm{f}})
-\epsilon_{\mathrm{f}}\gamma_{1,\mathrm{f}}\prod_{j=0}^{N-2}(\eta+\gamma_{N-j,\mathrm{f}})
}.
\end{align*}
\textbf{Part VII: } Next we show the identity  $\alpha_{N,\mathrm{f}}=\hat{\alpha}_{N,\mathrm{f}}$, i.e., that the minimizer to~\eqref{eq:LP_terminal_detect_reduced} is also a minimizer to~\eqref{eq:LP_terminal_detect}, under the additional assumption $\gamma_{k,\mathrm{f}}=\overline{\gamma}_{\mathrm{f}}$ $\forall k\in\mathbb{I}_{\geq 1}$.
To this end, we need to ensure that the minimizer $\tilde{\ell}^*$ satisfies the non-negativity constraints~\eqref{eq:LP_terminal_detect_nonneg_1}
--\eqref{eq:LP_terminal_detect_nonneg_2}, the lower and upper bound~\eqref{eq:LP_terminal_detect_Vf_bound}, and the constraints~\eqref{eq:LP_terminal_detect_V_bound} for $k\in\mathbb{I}_{[1,N-1]}$. 
Non-negativity of the stage cost follows from $\tilde{\ell}^*_0\geq 0$ and $\tilde{\ell}_k^*=a_k(\tilde{\ell}_0^*+\eta)$, $k\in\mathbb{I}_{[1,N-1]}$ for $a_k\geq 0$.
Equation~\eqref{eq:ab_explicit_app} with $\delta_l=\gamma_{l,\mathrm{f}}$ ensures $a_k>0$ due to $\gamma_{k,\mathrm{f}}=\overline{\gamma}_{\mathrm{f}}$ and $\eta\in(0,1)$. 
 Non-negativity of $\tilde{\sigma}^*_k$ follows from the explicit formula~\eqref{eq:LP_sigma_W_expl_sigma}. 
The terminal cost satisfies 
\begin{align*}
(\hat{\alpha}_N-1)\epsilon_{\mathrm{o}}\stackrel{\eqref{eq:LP_terminal_detect_reduced_cost}}{=}
&\sum_{k=1}^{N-1}\tilde{\ell}^*_k+\tilde{V}^*_{\mathrm{f}}-\tilde{V}^*
\stackrel{\eqref{eq:LP_terminal_detect_reduced_Vf_next_bound}}{=}-\epsilon_{\mathrm{f}}\tilde{V}_{\mathrm{f}}^*.
\end{align*}
Thus, non-negativity of the terminal cost $\tilde{V}_{\mathrm{f}}$ follows from $\epsilon_{\mathrm{f}}>0$ and $\hat{\alpha}_{N,\mathrm{f}}\leq 1$.

Given that Inequalities~\eqref{eq:LP_terminal_detect_reduced_V_next_bound} for $k=N$ and Inequality~\eqref{eq:LP_terminal_detect_reduced_Vf_next_bound} hold with equality, we have 
$\tilde{\sigma}_N^*=\dfrac{1+\epsilon_{\mathrm{f}}}{\gamma_{1,\mathrm{f}}}\tilde{V}_{\mathrm{f}}^*$.
Hence, the lower and upper bound~\eqref{eq:LP_terminal_detect_Vf_bound} hold since
$\underline{c}_{\mathrm{f}}\leq \dfrac{\gamma_{1,\mathrm{f}}}{1+\epsilon_{\mathrm{f}}}\leq \overline{c}_{\mathrm{f}}$.
By utilizing the solution~\eqref{eq:LP_sigma_W_expl_sigma} and the fact that Inequality~\eqref{eq:LP_terminal_detect_reduced_V_bound} holds with equality,  conditions~\eqref{eq:LP_terminal_detect_V_bound} are equivalent to
\begin{align*}
\gamma_{N,\mathrm{f}}-\gamma_{N-k,\mathrm{f}}\eta^k\leq\sum_{j=0}^{k-1}(1+\gamma_{N-k,\mathrm{f}}\eta^{k-1-j})\tilde{\ell}_j^*,~k\in\mathbb{I}_{[1,N-1]}.
\end{align*}
Abbreviate $\tilde{\delta}_l:=\dfrac{\eta+\gamma_{N-l,\mathrm{f}}}{1+\gamma_{N-l-1,\mathrm{f}}}$. 
Utilizing the analytical solution from Lemma~\ref{lemma:analytic_ab_app}, the right hand side is equivalent to
\begin{align*}
&(1+\gamma_{N-k,\mathrm{f}}\eta^{k-1})\tilde{\ell}_0^*
+\sum_{j=1}^{k-1}(1+\gamma_{N-k,\mathrm{f}}\eta^{k-1-j})a_j(\tilde{\ell}_0^*+\eta)\\
&\stackrel{\eqref{eq:ab_explicit_app},\eqref{eq:induction_formula_app}}{=}(1+\gamma_{N-k,\mathrm{f}}\eta^{k-1})\tilde{\ell}_0^*+ (\tilde{\ell}_0^*+\eta)\cdot\\
&\left(\gamma_{N,\mathrm{f}}-\gamma_{N-k+1,\mathrm{f}}\prod_{l=0}^{k-2}\tilde{\delta}_l+\gamma_{N-k,\mathrm{f}}\prod_{l=0}^{k-2}\tilde{\delta}_l-\gamma_{N-k,\mathrm{f}}\eta^{k-1}\right).
\end{align*}
Thus, the inequalities can be equivalently  written as
\begin{align*}
&\gamma_{N,\mathrm{f}}+\eta\\\
\leq &
(\tilde{\ell}_0^*+\eta)\left(1+\gamma_{N,\mathrm{f}}-(\gamma_{N-k+1,\mathrm{f}}-\gamma_{N-k,\mathrm{f}})\prod_{l=0}^{k-2}\tilde{\delta}_{l}\right).
\end{align*}
Utilizing $\gamma_{N-k,\mathrm{f}}=\overline{\gamma}_{\mathrm{f}}$ and $\overline{a}_N\geq 0$, this condition holds with
\begin{align*}
\tilde{\ell}_0^*+\eta\stackrel{\eqref{eq:LP_terminal_detect_minimizer_0}}{=}&
 \dfrac{(1+\epsilon_{\mathrm{f}})(\gamma_{N,\mathrm{f}}+\eta)}{(1+\epsilon_{\mathrm{f}})(1+\gamma_{N,\mathrm{f}})-\epsilon_{\mathrm{f}}\overline{a}_N}
\geq\dfrac{\gamma_{N,\mathrm{f}}+\eta}{1+\gamma_{N,\mathrm{f}}}.
\end{align*}
\textbf{Part VIII: }
Under the assumption
$\gamma_{k,\mathrm{f}}=\overline{\gamma}_{\mathrm{f}}$ $\forall k\in\mathbb{I}_{[1,N]}$, $\alpha_{N,\mathrm{f}}>0$ reduces to
\begin{align*}
(\overline{\gamma}_{\mathrm{f}}+\underbrace{\eta+\epsilon_{\mathrm{o}}}_{=1})
\epsilon_{\mathrm{f}}\overline{\gamma}_{\mathrm{f}}(\overline{\gamma}_{\mathrm{f}}+\eta)^{N-1}<
\epsilon_{\mathrm{o}}(1+\epsilon_{\mathrm{f}})(1+\overline{\gamma}_{\mathrm{f}})^N.
\end{align*}
Applying the logarithm yields~\eqref{eq:hat_alpha_explicit_N_terminal_detect}. 
If $\gamma_{k,\mathrm{f}}=\overline{\gamma}_{\mathrm{f}}$ does not hold for some $k\in\mathbb{I}_{\geq 1}$, then by Assumption \ref{ass:stab_term} we still have $\gamma_{k,\mathrm{f}} \le \overline{\gamma}_{\mathrm{f}}$. Hence, if we replace $\overline{\gamma}_{\mathrm{f}}$ by $\gamma_{k,\mathrm{f}}$ in the LP~\eqref{eq:LP_terminal}, then the constraints are tightened and consequently the minimum and thus $\alpha_{N,\mathrm{f}}$ increases. Thus, if Inequality \eqref{eq:hat_alpha_explicit_N_terminal_detect} holds, then $\alpha_{N,\mathrm{f}}>0$ holds also if $\gamma_{k,\mathrm{f}} \ne \overline{\gamma}_{\mathrm{f}}$ for some $k\in\mathbb{I}_{\geq 1}$.
\qed


\subsection{Proofs - Section~\ref{sec:theory}}
\label{sec:app_2}
\change{In the following, we discuss how Theorems~\ref{thm:grimm}, \ref{thm:main}, \ref{thm:analytic_sigma_W}, \ref{thm:analytic_grune} from Section~\ref{sec:theory} follow as a special case of the more general results with terminal costs in Theorems~\ref{thm:grimm_terminal}, \ref{thm:terminal_LP}, \ref{thm:analytic_terminal_pdf}, \ref{thm:analytic_terminal_sigma_W} from Section~\ref{sec:terminal}.} 
In particular, since $V_{\mathrm{f}}=0$, we consider Assumption~\ref{ass:term} with $\overline{c}_{\mathrm{f}}=\underline{c}_{\mathrm{f}}=0$ and $\epsilon_{\mathrm{f}}=\infty$. 
Although in case $V_{\mathrm{f}}=0$, Inequality~\eqref{eq:term_2} cannot be satisfied, the results in Section~\ref{sec:terminal} remain applicable since our choice $\epsilon_{\mathrm{f}}=\infty$ ensures that the inequality is not used in the proofs. 
Furthermore, $V_{\mathrm{f}}=0$ ensures that Assumption~\ref{ass:stab_term} coincides with Assumption~\ref{ass:stab} by setting $\gamma_{k,\mathrm{f}}=\gamma_k$. 
The resulting formulas in the theorems can be viewed as the limit of the general formulas for $\epsilon_{\mathrm{f}}\rightarrow\infty$ . 

In the proof of Theorem~\ref{thm:grimm_terminal}, Case II is always active, which does not require Assumption~\ref{ass:term}. 
The corresponding constant satisfies $\alpha_N=\lim_{\epsilon_{\mathrm{f}}\rightarrow\infty}\alpha_{N,\mathrm{f}}$, while the performance bound~\eqref{eq:performance} follows from~\eqref{eq:performance_terminal} by using $\overline{c}_{\mathrm{f}}=0$. 

Regarding Theorem~\ref{thm:terminal_LP}, $\underline{c}_{\mathrm{f}}=\overline{c}_{\mathrm{f}}=0$ implies $\tilde{V}_{\mathrm{f}}=0$ and we dropp the constraints~\eqref{eq:LP_terminal_Vf_bound} and \eqref{eq:LP_terminal_Vf_next_bound}.

Regarding the proof of Theorem~\ref{thm:analytic_terminal_sigma_W}: 
In Part I, we can directly use $\tilde{\ell}_{N-1}-\tilde{\ell}_{N-1}^*<0$. 
Part II is unnecessary since the constraint~\eqref{eq:LP_terminal_detect_reduced_Vf_next_bound} is not included in the LP~\eqref{eq:LP}. 
In Part III, we choose $c_0=c_1$, $c_2=0$ with $\tilde{V}-\tilde{V}^*=c_0 \min_{k\in\mathbb{I}_{[1,N-1]}}\eta^{k-1}\gamma_{N-k+1}>0$. 
Part IV can again be used directly, if the constraint~\eqref{eq:LP_terminal_detect_v2_reduced_Vf_next_bound} is replaced by 
$\sum_{j=0}^{N-1}\tilde{\ell}_j\leq \gamma_N$, which corresponds to the limit  $\epsilon_{\mathrm{f}}\rightarrow\infty$. 
The remainder of the proof is analogous.

\subsection{Proofs - Section~\ref{sec:design}}
\label{sec:app_3}
\subsubsection{Proof of Proposition~\ref{prop:exp_IO}} 
\label{app:exp_IO}
Assumption~\ref{ass:set} holds trivially.  
Assumption~\ref{ass:stab_term} follows from $\ell$ quadratically upper bounded, $\sigma$ quadratic, the exponential decay $\rho\in[0,1)$, and the geometric series (cf.~\cite[Prop.~2]{Koehler2020Regulation}). 
Assumption~\ref{ass:detect} follows from choosing the scaling $\epsilon_{\mathrm{o}}>0$ appropriately (cf.~\cite[Prop.~3]{Koehler2020Regulation}). \qed

\subsubsection{Proof of Proposition~\ref{prop:simple_penalty}} 
\label{app:simple_penalty}
Inequalities~\eqref{eq:term_1} hold with equality using $\underline{c}_{\mathrm{f}}=\overline{c}_{\mathrm{f}}=\omega$. 
Inequality~\eqref{eq:term_2} follows with:
\begin{align*}
&\min_{u\in\mathbb{U}}V_{\mathrm{f}}(f(x,u))+\ell(x,u)=V_{1,\mathrm{f}}(x)\stackrel{\eqref{eq:stab_term}}{\leq}\gamma_{1,\mathrm{f}}\sigma(x).&\qed 
\end{align*}

\subsubsection{Proof of Proposition~\ref{prop:finite_tail}} 
\label{app:finite_tail}
The lower bound in Inequality~\eqref{eq:term_1} holds using $V_{\mathrm{f}}\geq \ell_{\min}=\sigma$. 
The upper bound follows with
\begin{align*}
V_{\mathrm{f}}(x)\leq \sum_{k=0}^{M-1}C_\ell\rho^k\sigma(x)=\dfrac{1-\rho^M}{1-\rho}C_\ell\sigma(x).
\end{align*}
Inequality~\eqref{eq:term_2} follows with
\begin{align*}
&\min_{u\in\mathbb{U}}V_{\mathrm{f}}(f(x,u))+\ell(x,u)-V_{\mathrm{f}}(x) \\
\leq& V_{\mathrm{f}}(f(x,\kappa(x)))+\ell(x,\kappa(x))-V_{\mathrm{f}}(x) \nonumber\\
=&\ell(x_u(M,x),u(M,x))\leq\dfrac{C_\ell\rho^M(1-\rho)}{1-\rho^M}V_{\mathrm{f}}(x)=\epsilon_{\mathrm{f}}V_{\mathrm{f}}(x),
\end{align*}
where the second inequality follows from the LP analysis in~\cite[Prop.~4]{kohler2021stability}. \qed


\subsection{Auxiliary lemmas}
\label{sec:app_4}

 \begin{lemma}
\label{lemma:induction_formula_app}
Given constants $\eta\in\mathbb{R}$, $\delta_l\in\mathbb{R}$, $l\in\mathbb{I}_{[1,N]}$, the following equations hold for any $k\in\mathbb{I}_{[1,N-1]}$:
\begin{subequations}
\label{eq:induction_formula_app}
\begin{align}
\label{eq:induction_formula_app_a}
&\sum_{j=1}^{k-1}\dfrac{\delta_{N-j+1}-\eta\delta_{N-j,\mathrm{f}}}{1+\delta_{N-j}}\prod_{l=0}^{j-2}\dfrac{\eta+\delta_{N-l}}{1+\delta_{N-l-1}}\nonumber\\
=&\delta_N-\delta_{N-k+1}\prod_{l=0}^{k-2}\dfrac{\eta+\delta_{N-l}}{1+\delta_{N-l-1}},\\
\label{eq:induction_formula_app_b}
&\sum_{j=1}^{k-1}\eta^{k-1-j}
\dfrac{\delta_{N-j+1}-\eta\delta_{N-j}}{1+\delta_{N-j}}\prod_{l=0}^{j-2}\dfrac{\eta+\delta_{N-l}}{1+\delta_{N-l-1}}\nonumber\\
=&
\prod_{l=0}^{k-2}\dfrac{\eta+\delta_{N-l}}{1+\delta_{N-l-1}}-\eta^{k-1}.
\end{align}
\end{subequations}
\end{lemma} 
\begin{proof}
We prove~\eqref{eq:induction_formula_app} using induction. 
For $k=1$, \eqref{eq:induction_formula_app_a} reduces to $0=\delta_N-\delta_N$. 
For the induction step, suppose~\eqref{eq:induction_formula_app_a} holds for some $k\in\mathbb{I}_{[1,N-2]}$ and hence
\begin{align*}
&\sum_{j=1}^{k}\dfrac{\delta_{N-j+1}-\eta\delta_{N-j}}{1+\delta_{N-j}}\prod_{l=0}^{j-2}\dfrac{\eta+\delta_{N-l}}{1+\delta_{N-l-1}}\nonumber\\
=&\sum_{j=1}^{k-1}\dfrac{\delta_{N-j+1}-\eta\delta_{N-j}}{1+\delta_{N-j}}\prod_{l=0}^{j-2}\dfrac{\eta+\delta_{N-l}}{1+\delta_{N-l-1}}\\+&\dfrac{\delta_{N-k+1}-\eta\delta_{N-k}}{1+\delta_{N-k}}\prod_{l=0}^{k-2}\dfrac{\eta+\delta_{N-l}}{1+\delta_{N-l-1}}\\
\stackrel{\eqref{eq:induction_formula_app_a}}{=}&
\delta_N-\delta_{N-k+1}\prod_{l=0}^{k-2}\dfrac{\eta+\delta_{N-l}}{1+\delta_{N-l-1}}\\
+&\dfrac{\delta_{N-k+1}-\eta\delta_{N-k}}{1+\delta_{N-k}}\prod_{l=0}^{k-2}\dfrac{\eta+\delta_{N-l}}{1+\delta_{N-l-1}}\\
=&\delta_N-\delta_{N-k}\prod_{l=0}^{k-1}\dfrac{\eta+\delta_{N-l}}{1+\delta_{N-l-1}}.
\end{align*}
Similarly, for $k=1$, \eqref{eq:induction_formula_app_b} reduces to $0=1-1$. 
For the induction step, suppose~\eqref{eq:induction_formula_app_b} for some $k\in\mathbb{I}_{[1,N-2]}$. Then
\begin{align*}
&\sum_{j=1}^{k}\eta^{k-j}
\dfrac{\delta_{N-j+1}-\eta\delta_{N-j}}{1+\delta_{N-j}}\prod_{l=0}^{j-2}\dfrac{\eta+\delta_{N-l}}{1+\delta_{N-l-1}}\nonumber\\
=&\dfrac{\delta_{N-k+1}-\eta\delta_{N-k}}{1+\delta_{N-k}}\prod_{l=0}^{k-2}\dfrac{\eta+\delta_{N-l}}{1+\delta_{N-l-1}}\nonumber\\
&+\eta\sum_{j=1}^{k-1}\eta^{k-1-j}
\dfrac{\delta_{N-j+1}-\eta\delta_{N-j}}{1+\delta_{N-j}}\prod_{l=0}^{j-2}\dfrac{\eta+\delta_{N-l}}{1+\delta_{N-l-1}}\nonumber\\
\stackrel{\eqref{eq:induction_formula_app_b}}{=}&
\left(\dfrac{\delta_{N-k+1}-\eta\delta_{N-k}}{1+\delta_{N-k}}+\eta\right)\prod_{l=0}^{k-2}\dfrac{\eta+\delta_{N-l}}{1+\delta_{N-l-1}}-\eta^k\\
=&\prod_{l=0}^{k-1}\dfrac{\eta+\delta_{N-l}}{1+\delta_{N-l-1}}-\eta^k.&\qedhere
\end{align*}
\end{proof}

\begin{lemma}
\label{lemma:analytic_ab_app}
Let $\delta_l,\eta,\tilde{\ell}_l\in\mathbb{R}$, $l\in\mathbb{I}_{[0,N]}$ be constants
satisfying the equations
\begin{align}
\label{eq:analytic_ab_app}
 (\delta_{N}-\delta_{N-k+1}\eta^{k-1})(\tilde{\ell}_0+\eta)=\sum_{j=1}^{k-1}(1+\delta_{N-k+1}\eta^{k-1-j})\tilde{\ell}_j,
\end{align}
for all $k\in\mathbb{I}_{[2,N]}$. 
Then, we have
\begin{align}
\label{eq:ab_explicit_app}
\tilde{\ell}_k=&a_k (\tilde{\ell}_0+\eta), ~k\in\mathbb{I}_{[1,N-1]},\\
a_k=&\dfrac{\delta_{N-k+1}-\eta\delta_{N-k}}{\delta_{N-k}+1}\prod_{j=0}^{k-2}\dfrac{\eta+\delta_{N-j}}{\delta_{N-j-1}+1},~ k\in\mathbb{I}_{[1,N-1]}.\nonumber
\end{align}
\end{lemma} 
\begin{proof}
Given~\eqref{eq:analytic_ab_app}, we can write $\tilde{\ell}_{k-1}$, $k\in\mathbb{I}_{[1,k']}$ as an affine function of $\tilde{\ell}_0$ using the following recursion:
\begin{align*}
\tilde{\ell}_{k-1}:=&\dfrac{\delta_N-\delta_{N-k+1}\eta^{k-1}}{1+\delta_{N-k+1}}(\tilde{\ell}_0+\eta)\\
&-\sum_{j=1}^{k-2}\dfrac{1+\delta_{N-k+1}\eta^{k-1-j}}{1+\delta_{N-k+1}}\tilde{\ell}_j.
\end{align*}
The result is given by $\tilde{\ell}_k=a_k(\tilde{\ell}_0+\eta)$, $k\in\mathbb{I}_{[1,N-1]}$ with recursively defined constants
\begin{align}
\label{eq:ab_recursion_app}
a_{k-1}=&\dfrac{\delta_N-\delta_{N-k+1}\eta^{k-1}}{1+\delta_{N-k+1}}
-\sum_{j=1}^{k-2}\dfrac{1+\delta_{N-k+1}\eta^{k-1-j}}{1+\delta_{N-k+1}}a_j,
\end{align}
for $k\in\mathbb{I}_{[2,N]}$.
Satisfaction of equation~\eqref{eq:ab_recursion_app} corresponds to the following equation for all $k\in\mathbb{I}_{[1,N-1]}$:
\begin{align*}
&(1+\delta_{N-k})a_{k}
\stackrel{\eqref{eq:ab_explicit_app}}{=}(\delta_{N-k+1}-\eta\delta_{N-k})\prod_{j=0}^{k-2}\dfrac{\eta+\delta_{N-j}}{\delta_{N-j-1}+1}\\
\stackrel{!}{=}&(\delta_N-\delta_{N-k}\eta^{k})
-\sum_{j=1}^{k-1}(1+\delta_{N-k}\eta^{k-j})a_j\\
\stackrel{\eqref{eq:ab_explicit_app}}{=}&\delta_N-\delta_{N-k}\eta^{k}\\
&-\sum_{j=1}^{k-1}(1+\delta_{N-k}\eta^{k-j})
\dfrac{\delta_{N-j+1}-\eta\delta_{N-j}}{\delta_{N-j}+1}\prod_{l=0}^{j-2}\dfrac{\eta+\delta_{N-l}}{\delta_{N-l-1}+1}.
\end{align*}
Satisfaction of this equation follows by applying the two equations from Lemma~\ref{lemma:induction_formula_app}:
\begin{align*}
&\sum_{j=1}^{k-1}(1+\delta_{N-k}\eta^{k-j})
\dfrac{\delta_{N-j+1}-\eta\delta_{N-j}}{\delta_{N-j}+1}\prod_{l=0}^{j-2}\dfrac{\eta+\delta_{N-l}}{\delta_{N-l-1}+1}\\
\stackrel{\eqref{eq:induction_formula_app}}{=}&\delta_N-\delta_{N-k+1}\prod_{l=0}^{k-2}\dfrac{\eta+\delta_{N-l}}{1+\delta_{N-l-1}}\\
&+\delta_{N-k}\eta\left(\prod_{l=0}^{k-2}\dfrac{\eta+\delta_{N-l}}{1+\delta_{N-l-1}}-\eta^{k-1}\right)\\
=&\delta_N-\delta_{N-k}\eta^k-(\delta_{N-k+1}-\eta\delta_{N-k})\prod_{l=0}^{k-2}\dfrac{\eta+\delta_{N-l}}{1+\delta_{N-l-1}}.&\qedhere
\end{align*}
\end{proof}

\begin{lemma}
\label{lemma:sum_nonnegative_app}
Suppose $\sum_{j=1}^{k-1}z_j\leq 0$ for all $k\in\mathbb{I}_{[2,k']}$, $k'\in\mathbb{I}_{\geq 1}$ and $\sum_{j=1}^{k'}\epsilon_jz_j\leq 0$ with a non-decreasing positive sequence $\epsilon_j>0$, $j\in\mathbb{I}_{[1,k']}$. Then
$\sum_{k=1}^{k'} z_k\leq 0$.
\end{lemma}
\begin{proof}
Define $\tilde{\epsilon}_k=\dfrac{\epsilon_k-\epsilon_{k-1}}{\epsilon_{k'}}\geq 0$, $k\in\mathbb{I}_{[2,k']}$. 
Multiplying the given inequalities by $\tilde{\epsilon}_k$, and the other inequality by $1/\epsilon_{k'}$ and summing up the inequalities yields
\begin{align*}
0\geq& \sum_{k=2}^{k'}\tilde{\epsilon}_k\left[\sum_{j=1}^{k-1}z_j\right] + \dfrac{1}{\epsilon_{k'}}\sum_{j=1}^{k'}\epsilon_j z_j\\
=&\sum_{k=1}^{k'}\left( \dfrac{\epsilon_k}{\epsilon_{k'}}+\sum_{j=k+1}^{k'}\tilde{\epsilon}_{j}\right)z_k=\sum_{k=1}^{k'}z_k,
\end{align*}
where the last equality uses a telescopic sum. 
\end{proof}


\begin{lemma}
\label{lemma:submult_induction_pdf}
Suppose conditions~\eqref{eq:submult_pdf_term} hold. 
Then, for any $k\in\mathbb{I}_{[1,N-1]}$, $l\in\mathbb{I}_{\geq 0}$, we have
\begin{align}
\label{eq:lemma_submult_induction_pdf}
&\epsilon_{\mathrm{f}}\left[\sum_{n=l+1}^{N-k+l}c_n+c_{N-k+1+l,\mathrm{f}}\right]\prod_{i=2}^{N-k}(\gamma_{i,\mathrm{f}}-1) \\
&-(1+\epsilon_{\mathrm{f}})(c_{N-k+l}+c_{N-k+1+l,\mathrm{f}}-c_{N-k+l,\mathrm{f}})\prod_{i=2}^{N-k}\gamma_{i,\mathrm{f}}\geq 0.\nonumber
\end{align}
\end{lemma}
\begin{proof}
This result is a modified version of~\cite[Lemma~10.1]{grune2010analysis}.
We show Inequality~\eqref{eq:lemma_submult_induction_pdf} using an induction proof w.r.t.. $k$. 
For $k=N-1$ and any $l\in\mathbb{I}_{\geq 0}$, \eqref{eq:lemma_submult_induction_pdf} reduces to
\begin{align*}
&(1+\epsilon_{\mathrm{f}})c_{1+l,\mathrm{f}}-c_{2+l,\mathrm{f}}-c_{1+l}\geq 0,
\end{align*}
which holds due to~\eqref{eq:submult_pdf_term_3}.
For the induction step, note that~\eqref{eq:lemma_submult_induction_pdf} is equivalent to non-negativity of the following expression
\begin{align*}
&(\gamma_{N-k,\mathrm{f}}-1)\epsilon_{\mathrm{f}}\left[\sum_{n=l+1}^{N-k+l}c_n+c_{N-k+1+l,\mathrm{f}}\right]\prod_{i=2}^{N-k-1}(\gamma_i-1)\\
&-\gamma_{N-k,\mathrm{f}}(1+\epsilon_{\mathrm{f}})\left[c_{N-k+l}+c_{N-k+1+l,\mathrm{f}}-c_{N-k+l,\mathrm{f}}\right]\prod_{i=2}^{N-k-1}\gamma_{i,\mathrm{f}} \nonumber\\
=&\gamma_{N-k,\mathrm{f}}\left(\epsilon_{\mathrm{f}}\left[\sum_{n=l+2}^{N-k+l}c_n+c_{N-k+1+l,\mathrm{f}}\right]\prod_{i=2}^{N-k-1}(\gamma_i-1)\right.\\
&\left.-(1+\epsilon_{\mathrm{f}})(c_{N-k+l}+c_{N-k+1+l,\mathrm{f}}-c_{N-k+l,\mathrm{f}})\prod_{i=2}^{N-k-1}\gamma_{i,\mathrm{f}}
\right)\\
&+\epsilon_{\mathrm{f}}\left[c_{l+1}\gamma_{N-k,\mathrm{f}}-\sum_{n=l+1}^{N-k+l}c_n-c_{N-k+1+l,\mathrm{f}}\right]\prod_{i=2}^{N-k-1}(\gamma_i-1). \end{align*}
Non-negativity of the first term follows by the induction assumption, i.e., \eqref{eq:lemma_submult_induction_pdf} holds with $k$ and $l$ replaced by $k+1$, $l+1$. 
Non-negativity of the second term follows from the sub-multiplicativity condition~\eqref{eq:submult_pdf_term_12} with $\gamma_{N-k,\mathrm{f}}=\sum_{n=0}^{N-k-1}c_n+c_{N-k,\mathrm{f}}$.
\end{proof}

\bibliographystyle{IEEEtran} 
\bibliography{Literature} 
 
\begin{IEEEbiography}[{\includegraphics[width=1in,height=1.25in,clip,keepaspectratio]{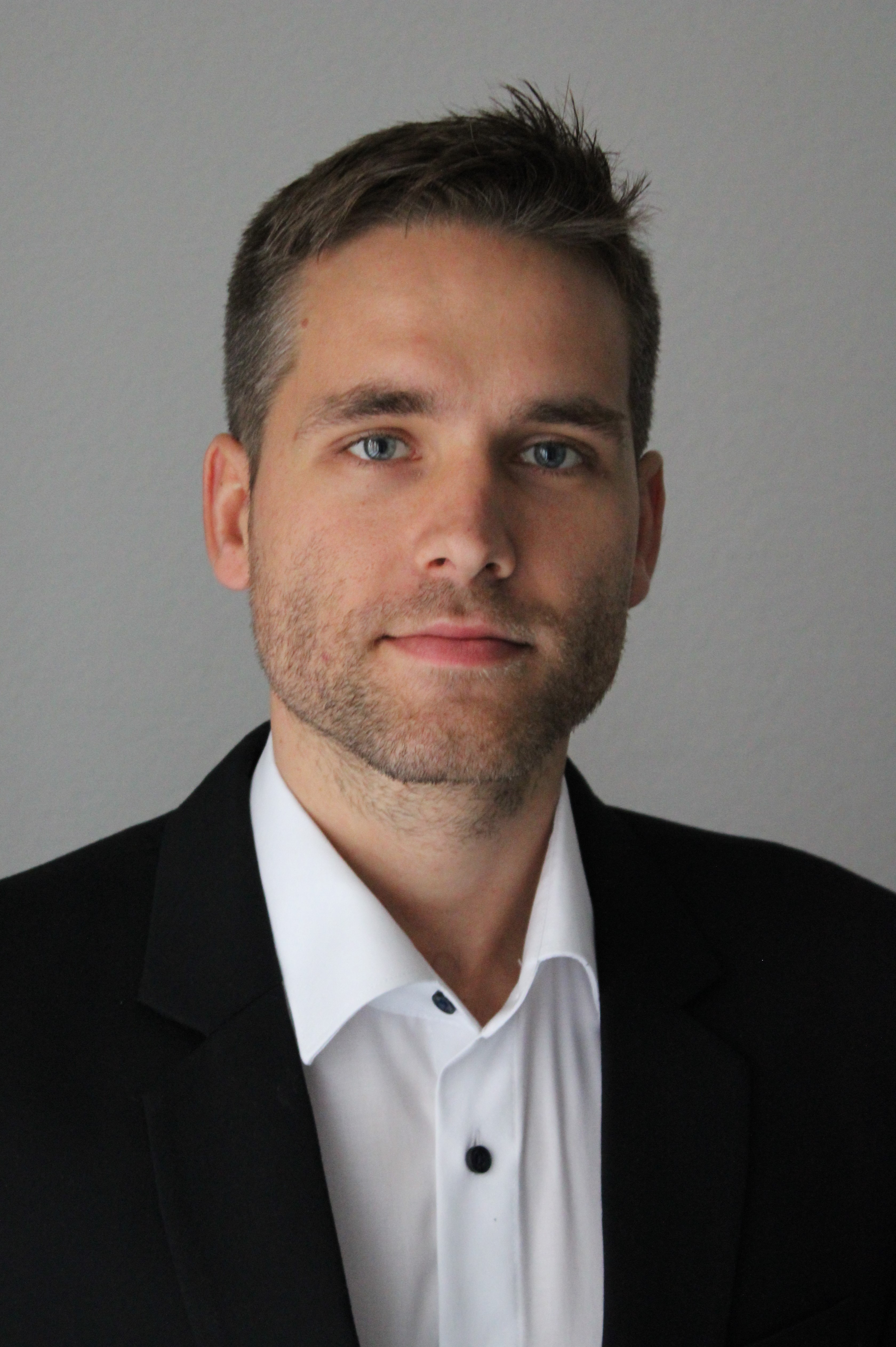}}]{Johannes K\"ohler}
received his Master degree in Engineering Cybernetics from the University of Stuttgart, Germany, in 2017. 
In 2021, he obtained a Ph.D. in mechanical engineering, also from the University of Stuttgart,
Germany, for which he received the 2021 European  Systems \& Control Ph.D. award. 
He is currently a postdoctoral researcher at the Institute for Dynamic Systems and Control (IDSC) at ETH Zürich.
His research interests are in the area of model predictive control and control and estimation for nonlinear uncertain systems. 
\end{IEEEbiography}

\begin{IEEEbiography}[{\includegraphics[width=1in,height=1.25in,clip,keepaspectratio]{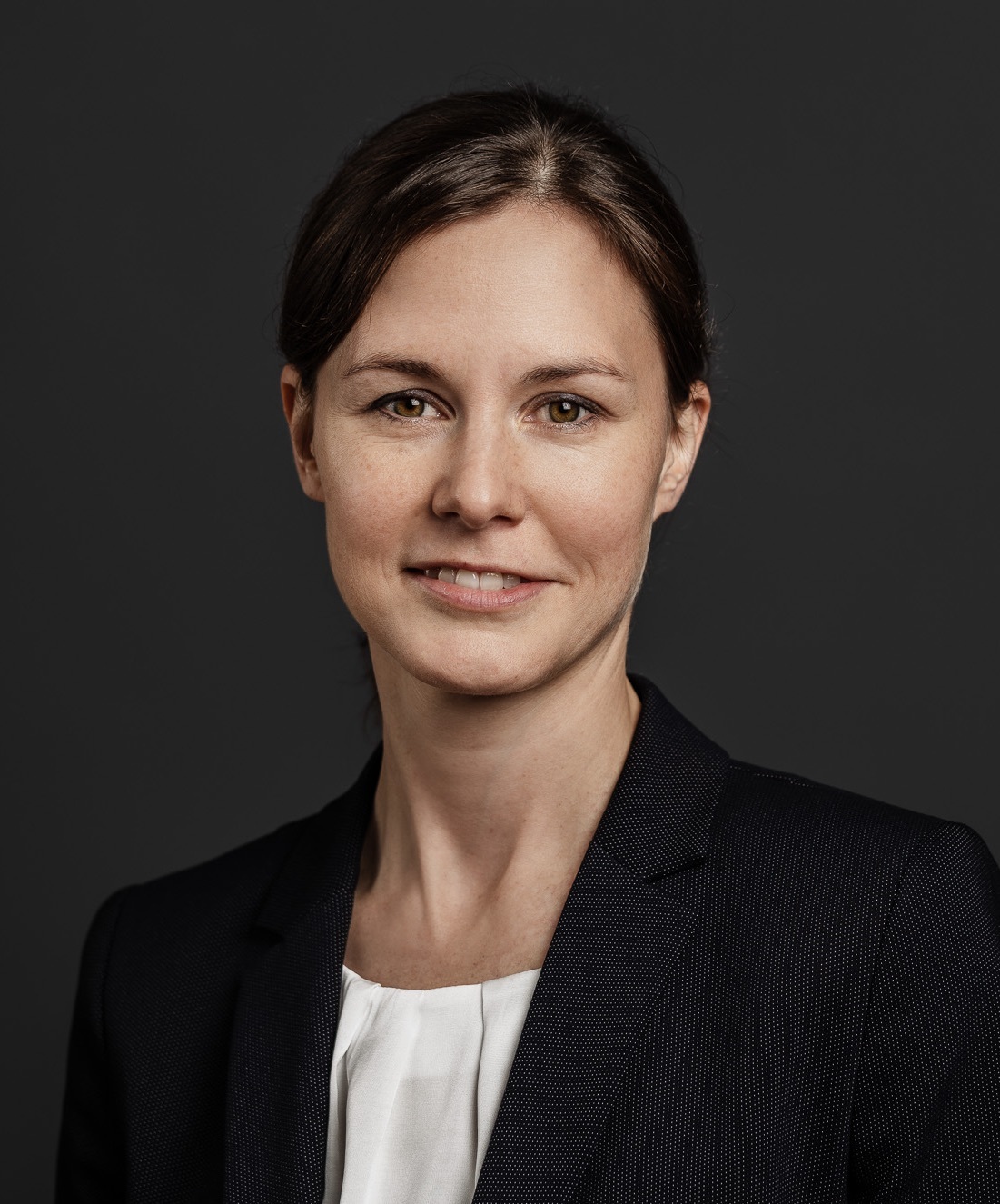}}]{Melanie N. Zeilinger}
 is an Associate Professor at ETH Zürich, Switzerland. 
She received the Diploma degree in engineering cybernetics from the University of Stuttgart, Germany, in 2006, and the Ph.D. degree with honors in electrical engineering from ETH Zürich, Switzerland, in 2011. 
From 2011 to 2012 she was a Postdoctoral Fellow with the Ecole Polytechnique Federale de Lausanne (EPFL), Switzerland.
She was a Marie Curie Fellow and Postdoctoral Researcher with the Max Planck Institute for Intelligent
Systems, Tübingen, Germany until 2015 and with the Department of Electrical Engineering and Computer Sciences at the University
of California at Berkeley, CA, USA, from 2012 to 2014. 
From 2018 to 2019 she was a professor at the University of Freiburg, Germany. 
Her current research interests include safe learning-based control, as well as distributed control and optimization, with applications to robotics and human-in-the loop control.
\end{IEEEbiography}

\begin{IEEEbiography}[{\includegraphics[width=1in,height=1.25in,clip,keepaspectratio]{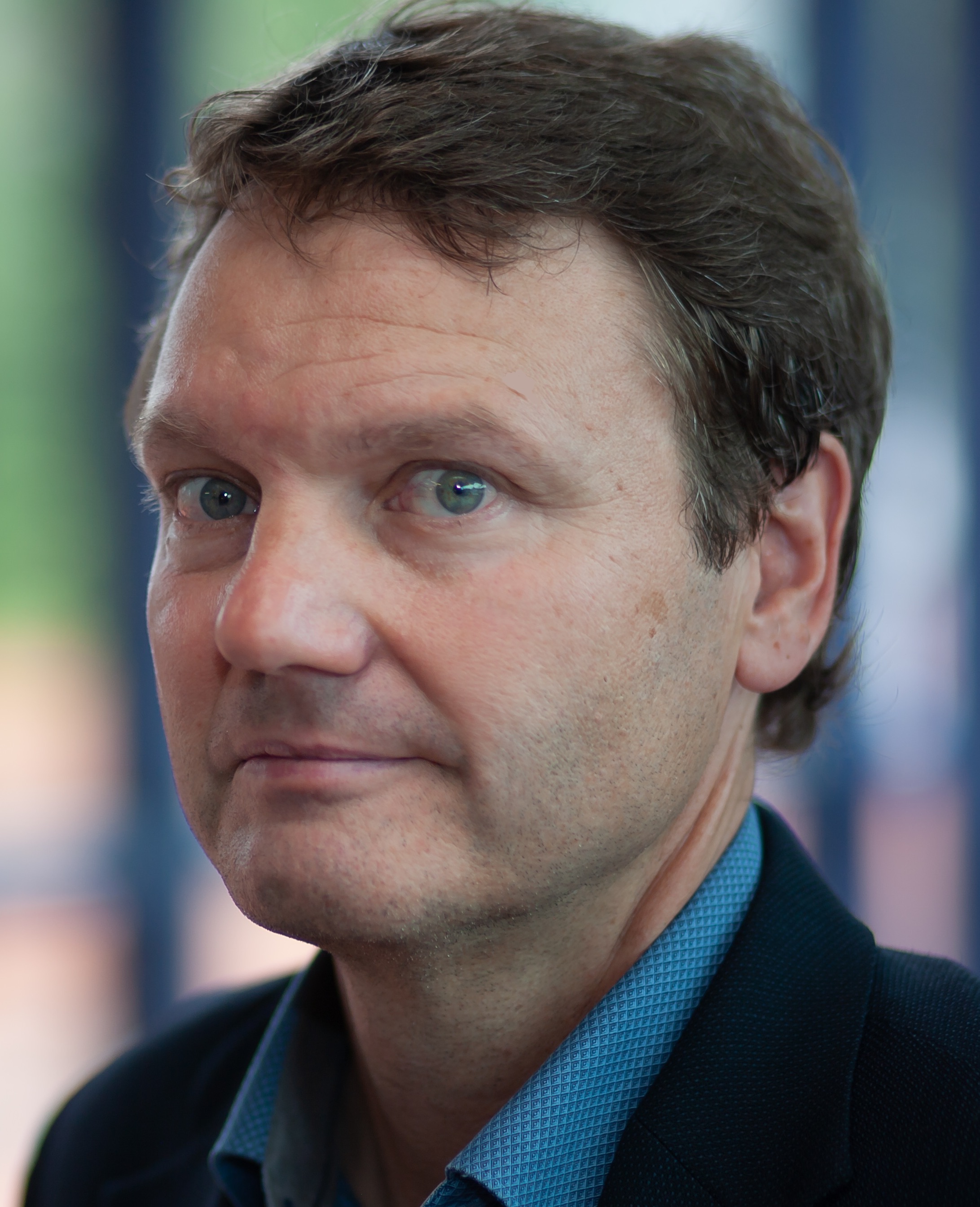}}]{Lars Grüne} 
 has been Professor for Applied Mathematics at the University of Bayreuth, Germany, since 2002. 
He received his Diploma and Ph.D. in Mathematics in 1994 and 1996, respectively, from the University of Augsburg and his habilitation from the J.W. Goethe University in Frankfurt/M in 2001. 
He held visiting positions at the Universities of Rome ‘Sapienza’ (Italy), Padova (Italy), Melbourne (Australia), Paris IX — Dauphine (France) and Newcastle (Australia). 
Prof. Grüne is Editor-in-Chief of the journal \textit{Mathematics of Control, Signals and Systems} (MCSS) and Associate Editor of various other journals, including the \textit{Journal of Optimization Theory and Applications} (JOTA), \textit{Mathematical Control and Related Fields} (MCRF) and the \textit{IEEE Control Systems Letters} (CSS-L). 
His research interests lie in the area of mathematical systems and control theory with a focus on numerical and optimization-based methods for nonlinear systems.
\end{IEEEbiography}

\end{document}